\documentclass[11pt,dvipsnames]{amsart}
\usepackage{amsmath,amssymb,amsfonts, mathtools}
\usepackage[mathscr]{eucal}
\usepackage[margin=1.5in]{geometry}
\usepackage{fullpage}

\usepackage{tikz, tikz-cd}%
\usepackage{adjustbox}

\usetikzlibrary{matrix,arrows,decorations.pathmorphing,cd,patterns,calc,backgrounds,patterns}
\tikzset{dummy/.style= {circle,fill,draw,inner sep=0pt,minimum size=1.2mm}}
\tikzset{vertex/.style={fill, circle, minimum size=.1cm, inner sep=0pt}}
\usetikzlibrary{decorations.pathreplacing, calc}
\usepackage{extarrows}

\usepackage{marvosym}
\usepackage{quiver} 
\usepackage{adjustbox}
\usepackage{graphicx,scalerel}

\usepackage[pagebackref, colorlinks, citecolor=DarkOrchid, linkcolor=PineGreen, urlcolor=NavyBlue]{hyperref}

\usepackage{comment}

\usepackage[textwidth=25mm, textsize=tiny]{todonotes}
\usepackage[final]{microtype}

\usepackage{enumerate}

\usepackage{appendix}

\numberwithin{equation}{section} 
\numberwithin{figure}{section}
\usepackage[nameinlink,capitalise,noabbrev]{cleveref}

\newcommand{\newrefformat}[2]{}

\usepackage{stmaryrd}%for mapsfrom

\newcommand\restr[2]{{% we make the whole thing an ordinary symbol
  \left.\kern-\nulldelimiterspace % automatically resize the bar with \right
  #1 % the function
  \vphantom{\big|} % pretend it's a little taller at normal size
  \right|_{#2} % this is the delimiter
  }}

\crefname{lemma}{Lemma}{Lemmas}
\crefname{theorem}{Theorem}{Theorems}
\crefname{definition}{Definition}{Definitions}
\crefname{proposition}{Proposition}{Propositions}
\crefname{remark}{Remark}{Remarks}
\crefname{observation}{Observation}{Observations}
\crefname{construction}{Construction}{Constructions}
\crefname{corollary}{Corollary}{Corollaries}
\crefname{question}{Question}{Questions}
\crefname{equation}{Equation}{Equations}
\crefname{construction}{Construction}{Constructions}
\crefname{ex}{Example}{Examples}
\crefname{appsec}{Appendix}{Appendices}
\crefname{subsection}{Subsection}{Subsections}
\Crefname{warning}{Warning}{Warnings}

\theoremstyle{plain}
\newtheorem{theorem}[equation]{Theorem}
\newtheorem{corollary}[equation]{Corollary}
\newtheorem{proposition}[equation]{Proposition}
\newtheorem{lemma}[equation]{Lemma}

\theoremstyle{definition}
\newtheorem{definition}[equation]{Definition}
\newtheorem{example}[equation]{Example}

\newtheorem{remark}[equation]{Remark}
\newtheorem{construction}[equation]{Construction}

\newtheorem{warning}[equation]{Warning}

\newtheorem*{notation}{Notation}

\crefname{introtheorem}{Theorem}{Theorems}

\usepackage{todonotes}

\DeclareMathOperator{\tr}{tr}

\DeclareMathOperator{\End}{End}
\DeclareMathOperator{\Aut}{Aut}
\DeclareMathOperator{\Hom}{Hom}

\DeclareMathOperator{\obj}{ob}
\DeclareMathOperator{\Mor}{Mor}
\DeclareMathOperator{\Proj}{Proj}

\DeclareMathOperator{\vol}{vol}
\def\CC{\mathcal{C}}
\def\DD{\mathcal{D}}
\def\EE{\mathcal{E}}

\def\PP{\mathcal{P}}
\def\RR{\mathcal{R}}

\def\bbR{\mathbb{R}}
\def\bbZ{\mathbb{Z}}
\def\wW{\mathcal{W}}

\def\gG{\mathcal{G}}
\def\mM{\mathcal{M}}

\tikzcdset{
  diagrams={>={Straight Barb[scale=0.5]}}
}
\tikzset{
  altstackar/.style={decorate, decoration={show path construction,
    lineto code={
      \path (\tikzinputsegmentfirst); \pgfgetlastxy{\xstart}{\ystart}
      \path (\tikzinputsegmentlast); \pgfgetlastxy{\xend}{\yend}
      \path ($(0,0)!1.5pt!(\ystart-\yend,\xend-\xstart)$); \pgfgetlastxy{\xperp}{\yperp}
      \foreach \n[evaluate=\n as \k using .5*#1-\n+.5] in {1,...,#1}{
        \ifodd\n{\draw[->, shorten <=2pt, shift={($\k*(\xperp,\yperp)$)}](\xstart,\ystart)--(\xend,\yend);}
        \else{\draw[<-, shorten >=2pt, shift={($\k*(\xperp,\yperp)$)}](\xstart,\ystart)--(\xend,\yend);}\fi
      }
    }
  }}, altstackar/.default={1}
}

\makeatletter
\def\makeCal#1{%
\expandafter\newcommand\csname c#1\endcsname{\mathcal{#1}}}
\def\makeBB#1{%
\expandafter\newcommand\csname b#1\endcsname{\mathbb{#1}}}
\def\makeFrak#1{%
\expandafter\newcommand\csname f#1\endcsname{\mathfrak{#1}}}
\def\makeScr#1{%
\expandafter\newcommand\csname s#1\endcsname{\mathscr{#1}}}

\count@=0
\loop
\advance\count@ 1
\edef\y{\@Alph\count@}%
\expandafter\makeCal\y
\expandafter\makeBB\y
\expandafter\makeFrak\y
\ifnum\count@<26
\repeat

\def\slashedarrowfill@#1#2#3#4#5{%
  $\m@th\thickmuskip0mu\medmuskip\thickmuskip\thinmuskip\thickmuskip
  \relax#5#1\mkern-7mu%
  \cleaders\hbox{$#5\mkern-2mu#2\mkern-2mu$}\hfill
  \mathclap{#3}\mathclap{#2}%
  \cleaders\hbox{$#5\mkern-2mu#2\mkern-2mu$}\hfill
  \mkern-7mu#4$%
}
\def\rightslashedarrowfill@{%
  \slashedarrowfill@\relbar\relbar\mapstochar\rightarrow}
\newcommand\xslashedrightarrow[2][]{%
  \ext@arrow 0055{\rightslashedarrowfill@}{#1}{#2}}
\makeatother

\AtEndDocument{%
\bibliographystyle{alpha}
\bibliography{./ref.bib}
}

\begin{document}

\author[Agarwal]{Sanjana Agarwal}
\address{Indiana University, Bloomington}
\email{sanjagar@iu.edu}

\author[Bilas]{Ramyak Bilas}
\address{Indiana University, Bloomington}
\email{rbilas@iu.edu}

\title[The Dennis trace for assembler $K$-theory]{The Dennis trace for assembler $K$-theory}

%%%%%%%%%%%%%%%%%%%%%%%%%%%%ABSTRACT%%%%%%%%%%%%%%%%%%%%%%%%
\begin{abstract}
We define a notion of Hochschild homology of a weak assembler and use it to construct a Dennis trace map from the group completion $K$-theory of a weak assembler to its Hochschild homology. We compute the Hochschild homology for some examples, in particular for various cases of restricted polytopes. We also give an explicit description of the Dennis trace map at level $0$, i.e. from the zeroth $K$-theory group to the zeroth Hochschild homology group. Further, we show that the Dennis trace is a refinement of the trace/regulator map defined by Bohmann et al thus showing that group homology is a trace invariant. 
\end{abstract}

\maketitle

\setcounter{tocdepth}{1}
\tableofcontents

%%%%%%%%%%%%%%%%%%%%%%%%%INTRO%%%%%%%%%%%%%%%%%%%%%%%%%%%%%%%%%
\section{Introduction}
The classical Dennis trace map for a ring $R$ is a map from its algebraic $K$-theory $K(R)$ to its Hochschild homology $HH(R)$. The machinery of trace methods is built upon refinements of Dennis trace map and has been among the most successful methods towards understanding and computing algebraic $K$-theory for last four decades.

In her work \cite{zakharevich2012scissors, ZAKHAREVICH20171176}, Zakharevich defined an analogous notion of $K$-theory in contexts that are more combinatorial, for example for polytopes. Similar to classical algebraic $K$-theory, the combinatorial $K$-theory is `universal' and thus extremely helpful towards understanding scissors congruence invariants of polytopes - the content of the generalized Hilbert's third problem. To be more precise: 
\begin{definition}
    An \textit{n-simplex} in an $n$-dimensional $X$, where $X$ is  one of the Euclidean $\mathcal{E}^n$, Hyperbolic $\mathcal{H}^n$ or Spherical geometry $\mathcal{S}^n$, is the convex hull of $n+1$ points in general position in $X$. A \textit{polytope} in $X$ is a finite union of simplices in $X$.
\end{definition}
Consider the category of polytopes as follows:
\begin{definition}\label{def: category of polytopes}
   Let $\PP_G^X$ denote the category whose objects are polytopes in $X$. Here $G$ is any subgroup of the full isometry group of $X$. Given two objects $P$ and $Q$, the morphisms from $P$ to $Q$ in $\PP_G^X$ are maps $P \xrightarrow[]{g}Q$ for $g\in G$, i.e. morphisms first send $P$ to $gP$ and then include $gP$ as a set into $Q$. 
\end{definition}
Given a polytope, it is possible to cover it with smaller polytopes $\{P_i\}_{i\in I}$ whose interiors are disjoint. The scissors congruence of polytopes, then, is defined as follows:
\begin{definition}
    The \textit{scissors congruence} is an equivalence relation on the objects of $\PP_G^X$. $P, Q \in \obj\PP_G^X$ are scissors congruent, $P \sim Q$, if there exist covers $\{P_i\}_{i\in I}$ and $\{Q_i\}_{i\in I}$ of $P$ and $Q$, respectively, along with $\{g_i \in G \ \vert \ i\in I\}$ such that $g_iP_i = Q_i$ for all $i\in I$.
\end{definition}
An invariant $\mathcal{I}$ of scissors congruence is a map to an abelian group $G$
$$\mathcal{I}: \obj\PP_G^X \to G$$
such that if $P \sim Q$, then $\mathcal{I}(P) = \mathcal{I}(Q)$.
The generalized Hilbert's problem asks for a full classification of all scissors congruence invariants given $n, X, G$. The problem is completely understood for Euclidean geometry in dimensions less than $5$ but remains open in other dimensions and for other geometries in dimensions above $2$ \cite{Sydler1965, Jessen1968}. 

Classically algebraic $K$-theory space (the infinite loop space associated to the $K$-theory spectrum) of a ring $R$ is the group completion of the topological monoid $|\coprod_{[P]} B\Aut[P]|$ where $[P]$ ranges over the isomorphism classes of finitely generated projective $R$-modules. Similarly, for polytopes, the $K$-theory space is constructed as the group completion of the topological monoid $|\coprod_{[P]} B\Aut[P]|$ where $[P]$ ranges over isomorphism classes of polytopes in the category $\gG(\PP_G^X)$ which we define in \cref{Def:G(C)} \cite{KupersLemannMalkiewichMillerSroka2024ScissorsAuto}. It is an infinite loop space whose zeroth homotopy group $K_0(\PP_G^X)$, is the classical scissors congruence group and the first homotopy group $K_1(\PP_G^X)$, is the abelianization of the scissors automorphism group, i.e. $K_1(\PP_G^X)$ captures information related to ways in which a polytope can be scissors congruent to itself \cite{ZAKHAREVICH20171176, KupersLemannMalkiewichMillerSroka2024ScissorsAuto}. In recent years, there has been a lot of progress towards better understanding various combinatorial $K$-theories, however, the machinery of trace methods for combinatorial $K$-theories is still not developed.

In this article, we construct a Dennis trace map for polytopes (more generally for any weak assembler (\cref{def:Assemblers})):
\begin{theorem}[\cref{const: Dennis trace map}]
    There exists a Dennis trace map
    \[DTr: K(\PP_G^X) \to HH(\PP_G^X).\]
\end{theorem}
 Here $HH(\PP_G^X)$ is defined as the Hochschild homology of an $Ab$-enriched category we construct associated to $\PP_G^X$. This category which we denote by $\bbZ\mM(\PP_G^X)$ (\cref{def: category of scissors correspondences}) is an analogue of the category of finitely generated projective modules, $\Proj_R$, for polytopes. In classical algebraic $K$-theory, $K(R)$ is defined using the automorphisms of the category $\Proj_R$ and the codomain of Dennis trace map is defined using the endomorphisms of this category. Using essentially the fact that $R$ (seen as a single object category) and $\Proj_R$ are Morita equivalent, one can then rewrite the codomain of Dennis trace map, equivalently, as Hochschild homology of $R$, $HH(R)$. We show something similar happens for polytopes. Let $G$ be the full isometry group, and $P$ any choice of polytope, we have
 \begin{theorem}[\cref{prop: Morita equiv between ZM and ZP}]
     The inclusion of the full subcategory generated by a non-empty polytope $P$, denoted as $\bbZ P$, in $\bbZ\mM(\PP_G^X)$ gives a Morita equivalence between $\bbZ P$ and $\bbZ\mM(\PP_G^X)$.
 \end{theorem}
 The result holds more generally whenever you have a weak assembler with a Morita object $P$ (\cref{def: ZP}). In particular, this holds for any weak assembler which satisfies $EA$ axioms (\cref{def: EA weak assembler}) but also for various restricted polytope weak assemblers that do not satisfy the $EA$ axioms (\cref{ex:Restricted_polytope}, \cref{ex: Morita object in Restricted Polytopes}).

 Using Morita invariance of Hochschild homology (\cref{thm: Morita_Invariance_Hochschild_homology}) and the definition of Hochschild complex for $\bbZ P$ (\cref{lemma: Hochschild homology of weak assembler}), this allows us to write the image of the Dennis trace map as follows:
 \begin{center}
\fbox{%
  \begin{minipage}{0.9\linewidth}
    \begin{equation}
      DTr: K(\PP_G^X) \to HH(R)
    \end{equation}
  \end{minipage}%
}
\end{center}
where $R$ is the ring of endomorphisms of $P$ in $\bbZ\mM(\PP_G^X)$, i.e. $\Hom_{\bbZ\mM(\PP_G^X)}(P,P)$. 

We also do various computations of Hochschild homology for various restricted polytopes:
\begin{lemma}[\cref{ex: HH of P_n}, \cref{ex: HH of P_mm}, \cref{ex: HH of P_01}]
  Hochschild homology of various restricted polytopes $\PP^{[0,1]}_1$, $\PP^{[0,1]}_{C_2}$, $\PP^n_{\bbZ\rtimes C_2}$ and $\PP^{m,m}_{\bbZ^2\rtimes D_4}$ can be fully computed.  
\end{lemma}
Further, we identify explicitly the image of the Dennis trace map at level $0$, i.e. the image of the map from $K_0 \to HH_0$:
\begin{theorem}[\cref{thm: Dennis trace at level 0}]
For a choice of non-empty polytope $P$, let $[Q]$ be a scissors congruence class, i.e. an element of $K_0(\PP_G^X)$, and $n$ be the smallest positive integer such that there is an inclusion $Q \to P^n$ in $\bbZ\mM(\PP_G^X)$ as in \cref{def: ZP}: 
% https://q.uiver.app/#q=WzAsMyxbMCwxLCJRIl0sWzIsMSwiUF57bn0iXSxbMSwwLCJRXzFRXzJcXGNkb3RzIFFfbiJdLFsyLDBdLFsyLDEsIlxccHNpXzFcXHBzaV8yXFxjZG90cyBcXHBzaV9uIiwxLHsic3R5bGUiOnsidGFpbCI6eyJuYW1lIjoiaG9vayIsInNpZGUiOiJ0b3AifX19XV0=
\[\begin{tikzcd}
	& {Q_1Q_2\cdots Q_n} & \\
	Q && {P^{n}}
	\arrow[from=1-2, to=2-1]
	\arrow["{\psi_1\psi_2\cdots \psi_n}"{description}, hook, from=1-2, to=2-3]
\end{tikzcd}\]
with $Q_i$ mapping into the $i$th copy of $P$ via a partial cover $\psi_i$. Then the image of $[Q]$ under Dennis trace is the class of the following element in $HH_0(R)$:
% https://q.uiver.app/#q=WzAsMyxbMSwwLCJRXzEiXSxbMCwxLCJQIl0sWzIsMSwiUCJdLFswLDEsIlxccHNpXzEiLDIseyJzdHlsZSI6eyJ0YWlsIjp7Im5hbWUiOiJob29rIiwic2lkZSI6ImJvdHRvbSJ9fX1dLFswLDIsIlxccHNpXzEiLDAseyJzdHlsZSI6eyJ0YWlsIjp7Im5hbWUiOiJob29rIiwic2lkZSI6InRvcCJ9fX1dXQ==
\[
\begin{tikzcd}
	& {Q_1} & \\
	P && P
	\arrow["{\psi_1}"', hook', from=1-2, to=2-1]
	\arrow["{\psi_1}", hook, from=1-2, to=2-3]
\end{tikzcd}
+
\begin{tikzcd}
	& {Q_2} & \\
	P && P
	\arrow["{\psi_2}"', hook', from=1-2, to=2-1]
	\arrow["{\psi_2}", hook, from=1-2, to=2-3]
\end{tikzcd}
+ 
\ \cdots \
+
\begin{tikzcd}
	& {Q_n} & \\
	P && P.
	\arrow["{\psi_n}"', hook', from=1-2, to=2-1]
	\arrow["{\psi_n}", hook, from=1-2, to=2-3]
\end{tikzcd}
\]
\end{theorem}
Finally, we also show that the trace or regulator map of \cite{bohmann2024trace} from $K$-theory of $\PP^X_G$ to group homology of $G$ with coefficients in a measure $A$ on $\PP^X_G$ (\cref{def: measure}), factors through the Dennis trace map. This shows that the group homology of $G$ with coefficients in $A$ is a true trace invariant of $\PP^X_G$ and for each $A$ we can extract a trace invariant out of $HH(R)$.

\begin{theorem}[\cref{thm: regulator factors through Dennis trace}]
	The trace/regulator map of \cite{bohmann2024trace} factors through the Dennis trace map, i.e. we have a commutative diagram
    \[\begin{tikzcd}
	{K_n(\PP^X_G)} && {HH_n(R)} \\
	& {H_n(G; A)}
	\arrow["Dtr", from=1-1, to=1-3]
	\arrow["{\operatorname{tr}}"', from=1-1, to=2-2]
	\arrow["{\operatorname{tr}_\mM}", from=1-3, to=2-2]
\end{tikzcd}\]
    where $\operatorname{tr}_\mM$ map is defined using \cref{const: T_M(C)},\eqref{eqn: tr_M}.
\end{theorem}
More generally we show this for any weak $G$-assembler.

\subsection{Outline}
In \cref{sec:Dennis_trace_for_rings} we review the construction of classical Dennis trace map for rings. We then proceed to review weak assemblers, $\CC$, and construction of their $K$-theory, $K(\CC)$, in \cref{Subsec:Assemblers}. There is no new work in these sections. 

In \cref{subsec: ZMC} we construct the category of scissors correspondences $\bbZ\mM(\CC)$ and show the Morita equivalence between $\bbZ\mM(\CC)$ and $\bbZ P(\CC)$, the full subcategory generated by the Morita object $P$, in \cref{subsec: Morita Equivalence}. In \cref{subsec: Morita Equivalence}, we also compute Hochschild homology for various restricted polytope categories. In \cref{Subsec:Dennis_Trace_for_Assemblers}, we finally construct the Dennis trace map for $K$-theory of weak assemblers, studying the explicit map at level $0$. Finally in \cref{subsection:Comparision_with_BGMMZ}, we show that this map factors the trace/regulator map of \cite{bohmann2024trace}.

\subsection{Notations}\label{Notations}
\begin{enumerate}
\item $Ab$ denotes the category of Abelian groups. $sAb$ denotes the category of simplicial abelian groups.
\item $\Proj_R$ denotes the category of finitely generated projective $R$-modules.
\item\label{Not:group_as_category} At various points in the article, where it will be clear from the context, the notation $G$ will mean the one object category associated to a group $G$. Most of the time this will be accompanied by taking the nerve of this category as written as $BG$ using the notation below.
\item\label{Not:nerve} Given a category $\CC$, we use $B(\CC)$ to denote its nerve. The nerve of a group $G$ is the nerve of the associated one object category.

    \item\label{Not:Geometric_realization} We use $|\cdot|$ to denote the geometric realization of a simplicial object. Note that the geometric realization of a simplicial set is a topological space whereas of a simplicial abelian group is a topological group.
    \item\label{Not:cyclic_nerve} We denote the cyclic nerve of a category, $\CC$, by $B^{cyc}(\CC)$ \cite[Page 119]{Connes}. In particular see \cite[7.3.10]{Loday} for $B^{cyc}(G)$.
    
Usually there need not exist a simplicial map from the nerve of a category to its cyclic nerve. But in the case when the category is a groupoid which we also write $G$ here, such a map always exists and, at simplicial level $m$, is given by
\[cyc_m: (g_1, g_2, \cdots , g_m)\mapsto ((g_1 g_2 \cdots g_m)^{-1}, g_1, g_2, \cdots , g_m).\]
    \item\label{Not:cyclic_bar_complex} We denote the cyclic bar complex or the Hochschild complex of a $\bbZ$-algebra $A$ as $HH(A)$ (\cite[Page 9]{Loday}) and the Hochschild complex of an $Ab$-category (i.e. a category enriched in the category of abelian groups), $\CC$, as $HH(\CC)$ (\cite[Definition 2.1.1]{McCarthy1994}). The Hochschild complex is a simplicial abelian group and its geometric realization gives a topological group. Also, using Dold-Kan, we get a chain complex associated to this simplicial abelian group which we also call as the Hochschild complex. The homotopy groups of the simplicial abelian group and the topological group and the homology groups of the chain complex all equal each other and we call these groups Hochschild homology groups $HH_{\ast}(\CC)$.

 For an $Ab$-category, $\CC$, one can forget the enrichment to get an underlying category and take its cyclic nerve $B^{cyc}\CC$. Then there is a natural simplicial map 
 $$I: B^{cyc}\CC\to HH(\CC)$$
 given at simplicial level $m$ by
 $$I_m: (f_0,f_1,f_2,\cdots ,f_m) \mapsto f_0\otimes f_1\otimes f_2\otimes \cdots \otimes f_m.$$
\end{enumerate}

\subsection*{Acknowledgements}
The authors are deeply thankful for many insightful conversations and suggestions from Cary Malkiewich. His encouragement and excitement towards this project has been extremely supportive and foundational. We are also grateful to Maxine Calle, Teena Gerhardt, Mike Hill, Kate Ponto, Maru Sarazola and Inna Zakharevich for taking time to answer various of our questions at different points of this project. We are also grateful to Jim Davis for his feedback on a draft of the paper. Finally, we would like to thank the PIs of NSF Focused Research Group on Cut-and-Paste Methods in Algebraic K-theory for organizing various excellent schools and workshops to learn about this field. The first author deeply thankful to Mike Mandell and Ben Spitz for listening to many ramblings about this project and provide helpful feedback.

%%%%%%%%%%%%%%%%%%%%%%%%%%%%%%SECTION 2 CLASSICAL%%%%%%%%%%%%%%%%%%%%%%%%%%
\section{The Classical Dennis Trace for Rings}\label{sec:Dennis_trace_for_rings}
Consider a ring $R$. The Dennis trace map associated to $R$ is a map from algebraic $K$-theory of $R$ to Hochschild homology of $R$
\[DTr_R: K(R) \to HH(R).\]
In this section, we briefly review how to define this classical map for group completion $K$-theory, since it is crucial towards our construction of Dennis trace map for assembler $K$-theory in \cref{Subsec:Dennis_Trace_for_Assemblers}.

For $P$ a finitely generated projective $R$-module, i.e. an object of $\Proj_R$, let $\Aut(P)$ be the group of automorphisms of $P$ and $\End(P)$ be the ring of endomorphisms of $P$. We first start with the definition of group completion model of algebraic $K$-theory of the ring $R$. 
\begin{definition}
    The \textit{group completion algebraic $K$-theory} of $R$, $K(R)$, is defined as the group completion of the topological space $|\sqcup_{[P]} B\Aut P|$ where $[P]$ ranges over all isomorphism classes in $\Proj_R$ and $B\Aut P$ is as in \cref{Notations}(\ref{Not:nerve}). This group completion can be understood as the loop space of the classifying space of $|\sqcup_{[P]} B\Aut P|$, i.e. 
    $$K(R) \simeq \Omega B|\bigsqcup_{[P]} B\Aut P|$$
    and it is an infinite loop space.
\end{definition}

Now consider the following sequence of maps
\begin{equation}\label{Eq:Dennis_trace_sequence}
  B\Aut P \xrightarrow{1} B^{cyc}\Aut P \xrightarrow{2} HH(\End P) \overset{3}{\simeq} HH(R) 
\end{equation}
where the objects and maps involved are as follows:
\begin{enumerate}[(1)]
\item $B\Aut P$ and $B^{cyc}\Aut P$ denote the nerve and the cyclic nerve, respectively, of the one object category associated to the group $\Aut P$ (see \cref{Notations}(\ref{Not:nerve}),(\ref{Not:cyclic_nerve}),(\ref{Not:group_as_category})). Map 1 is then the map from the nerve to the cyclic nerve of $\Aut P$ as outlined in \cref{Notations}(\ref{Not:cyclic_nerve}). At simplicial level $m$, for $g_i \in \Aut P$, the map is 
\[(g_1, g_2, \cdots , g_m)\mapsto ((g_1 g_2 \cdots g_m)^{-1}, g_1, g_2, \cdots , g_m).\]

\item $HH(\End P)$ is the Hochschild complex associated to the $\bbZ$-algebra $\End P$ as in \cref{Notations}(\ref{Not:cyclic_bar_complex}). The natural inclusion $\Aut P \to \End P$ induces a map of cyclic nerves of these one object (unenriched) categories $$B^{cyc}\Aut P \to B^{cyc}\End P.$$ But the one object category $\End P$ is also enriched in $\bbZ$-algebras and therefore we have a map from $$B^{cyc}\End P \to HH(\End P)$$ as in \cref{Notations}(\ref{Not:cyclic_bar_complex}). Composing these two, we have map 2. Explicitly, at cyclic level $m$, for $g_i \in \Aut P$, the map is
\[(g_0 \times g_1 \times g_2 \times \cdots \times g_m) \mapsto (g_0 \otimes g_1 \otimes g_2 \otimes \cdots \otimes g_m).\]

\item $HH(R)$ is the cyclic bar complex associated to the $\bbZ$-algebra $R$ as in \cref{Notations}(\ref{Not:cyclic_bar_complex}). For $P$ a finitely generated projective $R$-module, $\End P$ and $R$ are Morita equivalent rings inducing a map $HH(\End P) \xrightarrow{\simeq} HH(R)$. More explicitly, this is the multitrace map at each cyclic level.
\end{enumerate}
We get, from \eqref{Eq:Dennis_trace_sequence},
\begin{equation}\label{eqn: trace}
    \sqcup_{[P]} B\Aut P \to \sqcup_{[P]} HH(\End P) \to HH(R).
\end{equation}
\begin{remark}\label{rmk: single object subcat}
    Alternatively, one can think of the first and second map above as going from $\sqcup_{[P]} B\Aut P$ to $HH(\Proj_R)$ and then using the fact that the single object subcategory of $\Proj_R$, $R$, is Morita equivalent to $\Proj_R$, we have the third map.
\end{remark}
 Taking geometric realization of the left hand side followed by group completion, gives algebraic $K$-theory of $R$. The right hand side is a simplicial abelian group. Therefore, its geometric realization is a topological group making group completion redundant, i.e. its group completion is homotopy equivalent to itself. Remember that we refer to the topological space and the chain complex associated to the Hochschild complex by the same notation (\cref{Notations}(\ref{Not:cyclic_bar_complex})). We therefore have the following map which we call the Dennis trace map for the ring $R$,
 \begin{equation}\label{eq:Dennis_Trace_Group_Completion}
    DTr_R: K(R) \to HH(R)
\end{equation}
from its algebraic $K$-theory to its Hochschild homology.
 
\begin{remark}\label{rem: image of K0(R)}
     Note that for a commutative ring $R$ at level $0$, the map from $K_0R \to HH_0R$ basically reads the dimension of the projective module $P$ representing a $K$-theory class by calculating the trace of the identity map from $P$ to itself. This statement is exactly true in particular for when $R$ is a field or when $P$ is a free module.
\end{remark}

%%%%%%%%%%%%%%%%%%%%%%%NEW SECTION%%%%%%%%%%%%%%%%%%%%%%%%%%%
 
\section{\texorpdfstring{$K$}{K}-theory of a Weak Assembler}\label{Subsec:Assemblers}

In this section, we review the construction of assembler $K$-theory \cite{ZAKHAREVICH20171176}. We use the more general setting of weak assemblers \cite{lemann2025scissorscongruencelineregulator} and categories with covering families \cite{bohmann2024trace} which we review below as well. We give particular examples of weak assemblers we are interested in, these examples will be recurring through later sections of this article.

An \textit{assembler} is a type of category where you can make sense of a notion of scissors congruence and define scissors congruence $K$-theory of. In \cite{bohmann2024trace}, the authors weaken the hypotheses of an assembler to define a \textit{category with covering families}. The definition of a category with covering families captures most of the nuances of an assembler, specially for the main examples the theory was created to study. In this paper we work with a \textit{weak assembler}, which is a slightly stronger notion than a category with covering families but is weak enough that most of the constructions of \cite{bohmann2024trace} which involve localization go through. 

\begin{definition}\label{def: category with families}\cite[Definition 2.1]{bohmann2024trace}
Let $\CC$ be a category. A multi-morphism in $\CC$ is an object $B\in \CC$ and a finite set (possibly empty) of morphisms indexed by $I$, all of which have the same codomain $B$:
\[\mathcal{F}_B=\{f_i: A_i \rightarrow B\}_{i\in I}.\]
A \textit{covering family structure} on $\CC$ is a collection of multi-morphisms which we call \textit{covering families} or \textit{covers} such that:
\begin{itemize}
    \item Each $\{Id_A:A\rightarrow A\}$ is a covering family for any $A\in \obj\CC$.
    \item Composition of covering families is preserved, i.e. given a covering family $\mathcal{F}_C$
    \[\{g_j:B_j \rightarrow C\}_{j\in J}\]
    and for each $j\in J$ a covering family $\mathcal{F}_{B_j}$
    \[\{f_{ij}:A_{ij}\rightarrow B_j\}_{i\in I_j},\]
    the composition $\{g_j \circ \mathcal{F}_{B_j}\}_{j\in J}$
    \[\{g_j\circ f_{ij}: A_{ij} \rightarrow C\}_{j\in J, i\in I_j}\]
    is a covering family of $C$.
\end{itemize}
A \textit{category with covering families} is a small category $\CC$ with a covering family structure and a distinguished basepoint $\varnothing\in \CC$ such that:
\begin{itemize}
    \item $\CC(\varnothing,\varnothing)=\{Id_\varnothing\}$ and $\CC(c,\varnothing)=\emptyset$ for $c\neq \varnothing$
    \item For any finite set $I$, the set $\{\varnothing\rightarrow\varnothing\}_{i\in I}$ is a covering family.
\end{itemize}
We call the category of the categories with covering families as \textbf{CatFam} where the morphisms are functors which preserve base point and the covering family structure. 
\end{definition}

Next we define a weak assembler, slightly reformulated from the cited source.
\begin{definition}\label{def:Assemblers}\cite[Definition 2.2]{lemann2025scissorscongruencelineregulator}
A \textit{weak assembler}, $\CC$, is a category with covering family such that:
\noindent\begin{itemize}
    \item \textbf{(I)} $\CC$ has a initial object $\varnothing$ such that empty family is a cover of $\varnothing$. This $\varnothing$ is the distinguished base point.
    \item \textbf{(D)} All the covering families are disjoint i.e. for any covering family $\{f_i:A_i\rightarrow B\}_{i\in I}$, the pullback $A_i \times_{B}A_j$ exists and is initial for $i\neq j$.
    \item \textbf{(R)} Any two covers of an object have a common \emph{refinement}.
    \item \textbf{(M)} Every morphism in $\CC$ is a monomorphism.
\end{itemize}
From \textbf{(I)} as $\CC(c, \varnothing)=\emptyset$ and $\varnothing$ is initial, weak assemblers have unique initial object. In \textbf{(R)}, by \emph{refinement} we mean the following: a covering family $\mathcal{F}_C$ is a refinement of the family $\mathcal{F'}_C = \{g_j: B_j \rightarrow C\}_{j\in J}$ if there exist covers $\mathcal{F}_{B_j}$ for each $j\in J$ such that $\mathcal{F}_C = \{g_j\circ \mathcal{F}_{B_j}\}_{j\in J}$. We define \textbf{wAsm} to the full subcategory of \textbf{CatFam} where the objects are the weak assemblers.
\end{definition}

\begin{remark}
    While our formulation of a weak assembler (\Cref{def:Assemblers}) slightly differs from Lemann's \cite[Definition 2.2]{lemann2025scissorscongruencelineregulator} by explicitly demanding a distinguished base point which is an initial object, this distinction is largely a matter of framing. In any category where a unique initial object $\varnothing$ exists and all morphisms are monomorphisms, the existence of a morphism $X\xrightarrow[]{} \varnothing$ implies $X \cong \varnothing$. Therefore, the initial object automatically satisfies the defining properties of a distinguished base point. This simplifies the analysis without affecting the results.
    
    All of the results here apply to weak assemblers of \cite{lemann2025scissorscongruencelineregulator} by adding a disjoint distinguished base point or an initial base point while using results from \cite{bohmann2024trace} (see \cite[Remark 2.3]{bohmann2024trace}) and forgetting the base point when using the results of \cite{KupersLemannMalkiewichMillerSroka2024ScissorsAuto}. 

\end{remark}
\begin{example}\label{ex:set_assembler}[Finite sets as a weak assembler] 
    Let \textbf{FinSet} be the category of finite sets (including the empty set $\emptyset$) and injective functions. $\emptyset$ is the initial object and $\{ f_i: A_i \rightarrow B\}_{i\in I}$ is a cover if $f_i(A_i)$ forms a partition of $B$.
\end{example}
\begin{example}\label{ex:polytope_assembler}[Polytopes as a weak assemblers \cite{bohmann2024trace}, \cref{def: category of polytopes}] 
    Let $X$ be $\mathcal{H}^n, \mathcal{S}^n$ or $\mathcal{E}^n$ and $G$ be a subgroup of isometry group of $X$. A polytope in $X$ is a finite union of $n$-simplices. 
    
    Let $\PP_G^X$ be the category with objects polytopes in $X$ (including empty polytope $\emptyset$) and morphisms $P\to Q$ consisting of inclusions $gP\subseteq Q$ for $g \in G$. The empty polytope $\emptyset$ is initial and $\{g_i: P_i\to Q\}_{i\in I}$ is a covering family if $Q= \cup g_iP_i$ where for any $i\neq j \in I$, $g_iP_i$ and $g_jP_j$ have disjoint interiors. 
    \end{example}
    
\begin{example}\label{ex:Restricted_polytope}[Restricted polytopes as weak assemblers] There are various restrictions that can be added to the category of polytopes in $\mathcal{E}^n$ to generate categories of restricted polytopes that are also weak assemblers using \cite[Lemma 6.3]{KupersLemannMalkiewichMillerSroka2024ScissorsAuto}. To add these restrictions, one makes a choice of a collection of hyperplanes, $\mathcal{L}$, in $\mathcal{E}^n$ and then defines a restricted convex polytope as a polytope in $\mathcal{E}^n$ which has all its facets on these hyperplanes and their intersections except the highest dimensional face which is in $\mathcal{E}^n$. Note that our use of notation $\mathcal{L}$ is different from \cite{KupersLemannMalkiewichMillerSroka2024ScissorsAuto} in that we are using $\mathcal{L}$ for just the generating hyperplanes and not the collection of all affine subspaces that $\mathcal{L}$ generates, which we denote by $\langle\mathcal{L}\rangle$. A general restricted polytope is then a finite union of restricted convex polytopes. A restricted isometry group $G_{\mathcal{L}}$ is a subgroup of the full isometry group of $\mathcal{E}^n$ that preserves $\langle\mathcal{L}\rangle$ as a set. One defines restricted polytope morphisms similar to $\Mor(\PP_G^{\mathcal{E}^n})$ except now we work with restricted isometry group. We denote a restricted polytope weak assembler corresponding to the set $\mathcal{L}$ of hyperplanes of $\mathcal{E}^n$ and restricted isometry subgroup $G_{\mathcal{L}}$  via $\PP_{G_{\mathcal{L}}}^{\mathcal{L}}$.

Here are some examples of categories of restricted polytopes:
\begin{enumerate}
    \item Let $\mathcal{L}$ be the set of all points in the interval $[0,1] \subset \mathcal{E}^1$. Then the restricted polytopes for $[0,1]$ are given by finite union of (disjoint) closed sub-intervals of $[0,1]$. In particular, we can consider the trivial group to be the restricted isometry subgroup, in which case the morphisms of these restricted polytopes are given by just inclusions. We denote this category of no-move polytopes by $\PP^{[0,1]}_1$.
    
    If we consider the whole symmetry group $C_2=\langle u | u^2=e \rangle$ of $[0,1]$, the morphisms are given either by inclusions corresponding to $e$ or by reflection about $1/2$ corresponding to $u$. We denote this category by $\PP^{[0,1]}_{C_2}$.
\item Let $\mathcal{L}_n$ be the set of points $$\left\{\frac{p}{n} \ \bigg\vert \ p\in\bbZ\right\}$$ in $\mathcal{E}^1$. Then the restricted polytopes for $\mathcal{L}_n$ are given by finite union of disjoint line segments with vertices in $\mathcal{L}_n$. The full restricted isometry group for $\mathcal{L}_n$ is given by $\bbZ \rtimes C_2$ where an element $q$ of $\bbZ$ corresponds to translation by $q/n$ and $C_2$ corresponds to the reflection subgroup. We denote this particular restricted polytope category by $\PP^{n}_{G_{\mathcal{L}_n}}$.
\item The above example can be generalized to `lattice polytopes' in any higher dimensional geometry. While the definition is general, we state the example only for the case $\mathcal{E}^2$ here since that will be most relevant for our computations later. Let $\mathcal{L}_{m,n}$ be the set of lines
\[ \bigcup_{p\in \bbZ}\left\{ x= \frac{p}{m} \right\}  \cup \bigcup_{q\in \bbZ}\left\{  y = \frac{q}{n} \right\}\]
in $\mathcal{E}^2$. Then the restricted polytopes for $\mathcal{L}_{m,n}$ are finite unions of rectangles that have vertices in $(1/m)\bbZ \times (1/n)\bbZ$ and sides as segments of one of the lines in $\mathcal{L}_{m,n}$. The full restricted isometry group depends on the relation between $m,n$. Translations by any $p/m$ or $q/n$ always preserve $\mathcal{L}_{m,n}$ and so does reflection. But if $m=n$ for example, then a rotation of the plane by $\pi/2$ also preserves $\mathcal{L}_{m,m}$. In fact the full restricted isometry subgroup for $\mathcal{L}_{m,m}$ will be the subgroup of $\bbR^2\rtimes O_2$ generated by such translations, rotations and reflections which is $\bbZ^2 \rtimes D_4$. Here $D_4$ is the second dihedral group. We denote the restricted category of polytopes corresponding to the case $\mathcal{L}_{m,m}$ by $\PP_{G_{\mathcal{L}_{m,m}}}^{m,m}$.
\end{enumerate}
\end{example}

Given a weak assembler $\CC$, we can define following categories associated to $\CC$: 

\begin{definition}\label{Def:W(C)}\cite[Definition 2.6]{lemann2025scissorscongruencelineregulator}
The \emph{category of covers}, $\wW(\CC)$ of a weak assembler $\CC$, has pairs $(I, \{P_i\}_{i\in I}) =:P $ as objects where $I$ is a finite set and each $P_i\in \obj\CC$, and a morphism 
\begin{equation}\label{eqn: arrow in W(C)}
    P=(I,\{P_i\})\xrightarrow{f}  Q= (J, \{Q_j\})
\end{equation}
is a map $f:I\rightarrow J$ of finite sets such that $\{f(i): P_i \rightarrow Q_j\}_{i\in f^{-1}(j)}$ is a cover for all $j\in J$. We shall refer to morphisms as $f$ which denotes both a set map from $I$ to $J$ as well as the collection of morphisms $f(i)$ in $\CC$. Composition of two such morphisms is given by a composition of the set maps and the composition of the corresponding morphisms. We shall refer to morphisms of $\wW(\CC)$ as \emph{covering map} or \emph{covers}.

The category $\wW(\CC)$ is symmetric monoidal with the monoidal product defined as $P \sqcup Q := (I \sqcup J, \{P_i\}\sqcup \{Q_j\})$. We often write the monoidal product of $P$ and $Q$ as just $PQ$.
\end{definition}

\begin{remark}\label{rem: Refinement in WC}
\cite[Proposition 2.11 (2)]{ZAKHAREVICH20171176}
Let $P=(I,\{P_i\}) \xrightarrow[]{f} Q= (J,\{Q_j\}) \xleftarrow[]{g}R=(K,\{R_k\})$ be a cospan of morphisms in $\wW(\CC)$. Then, we can use the fact that any two covers in $\CC$ have common refinement in $\CC$ to get the following commuting diagram:
% https://q.uiver.app/#q=WzAsNCxbMCwwLCJTX3tQXFxjYXAgUn0iXSxbMSwwLCJSIl0sWzAsMSwiUCJdLFsxLDEsIlEiXSxbMSwzLCJnIl0sWzIsMywiZiIsMl0sWzAsMiwiaF8xIiwyXSxbMCwxLCJoXzIiXV0=
\[\begin{tikzcd}
	{S_{P\cap R}} & R \\
	P & Q
	\arrow["{h_2}", from=1-1, to=1-2]
	\arrow["{h_1}"', from=1-1, to=2-1]
	\arrow["g", from=1-2, to=2-2]
	\arrow["f"', from=2-1, to=2-2]
\end{tikzcd}\]
where $S_{P\cap R} = (L,\{S_l\})$. By definition of finite cover,  
$$\{S_l\}_{l\in h_1^{-1}(i)} \xrightarrow[]{h_1} P_i \text{ and }\{S_l\}_{l\in h_2^{-1}(k)}\xrightarrow[]{h_2} R_k$$ are covers of pieces $P_i$ and $R_k$ respectively for each $i\in I$ and $k\in K$. We shall call $S_{P\cap R}$ the \textit{common refinement} of $P$ and $R$ over $Q$ in $\wW(\CC)$.
\end{remark}

The morphisms in $\wW(\CC)$ are very restrictive as one only has a morphism from \emph{finer pieces} into \emph{coarser pieces} and there is no map in the backwards direction. In \cite{KupersLemannMalkiewichMillerSroka2024ScissorsAuto}, the authors show that all morphisms of $\wW(\CC)$ can be localized to construct a groupoid $\gG(\CC)$ called the scissors congruence groupoid with the property $B(\wW(\CC))\simeq B(\gG(\CC))$ (see \cite[section 2.3]{lemann2025scissorscongruencelineregulator} for an alternate proof). We briefly recall this groupoid construction now.
\begin{definition}\label{Def: Category of fractions}
    The \textit{category of fractions} of a pair $(\CC,W)$, where $W$ is a wide subcategory of $\CC$, denoted by $\CC[W^{-1}]$, has same objects as of $\CC$ and the morphism are equivalence classes of spans such that the wrong way arrow is in $W$
    \[\CC[W^{-1}](a,b)= \left\{ a \xleftarrow{f} c\xrightarrow{g}b |f\in W \right\}/\sim\]
    where two spans from $a$ to $b$, $a \xleftarrow{f} c\xrightarrow{g}b$ and $a \xleftarrow{f'} c'\xrightarrow{g'}b$ are equivalent if there exists an object $x$ and morphisms $u,v$ in $\CC$ such that the following diagram commutes 
    % https://q.uiver.app/#q=WzAsNSxbMCwxLCJhIl0sWzEsMCwiYyJdLFsxLDIsIlxcZG90e2N9Il0sWzIsMSwiYiJdLFsxLDEsIngiXSxbMiwwLCJcXGRvdHtmfSJdLFsyLDMsIlxcZG90e2d9IiwyXSxbMSwwLCJmIiwyXSxbMSwzLCJnIl0sWzQsMSwidSJdLFs0LDIsInYiXV0=
\begin{equation} \label{diagram: eqivalence of span}
\begin{tikzcd}[cramped, column sep=small, row sep=small]
	& c \\
	a & x & b \\
	& c'
	\arrow["f"', from=1-2, to=2-1]
	\arrow["g", from=1-2, to=2-3]
	\arrow["u", from=2-2, to=1-2]
	\arrow["v", from=2-2, to=3-2]
	\arrow["f'", from=3-2, to=2-1]
	\arrow["g'"', from=3-2, to=2-3]
\end{tikzcd}
\end{equation}
and $fu=f'v\in W$. 
\end{definition}
\begin{remark}
 There is a localization functor $L_{W}: \CC\rightarrow\CC[W^{-1]}]$ which is identity on objects and maps morphisms as
 $$(a\xrightarrow[]{f}b) \mapsto (a \xleftarrow[]{=}a \xrightarrow[]{f}b).$$
 If $\CC$ has a symmetric monoidal structure and morphisms of $W$ are closed under tensoring with objects of $\CC$, then $\CC[W^{-1}]$ inherits a symmetric monoidal structure and $L_{W}$ is a symmetric monoidal functor \cite[Lemma 2.14]{KupersLemannMalkiewichMillerSroka2024ScissorsAuto}, \cite{DayLocalization}.   
\end{remark}
The following is a consequence of properties \textbf{R} and \textbf{M} of a weak assembler.
%%%%%%%%%%%%%%%%%%%%%%%%%%%%%%%%%%%%%%%%%%%%%%%%%%
\begin{lemma}\label{lemma: G(C) from ore}\cite[Lemma 2.15]{KupersLemannMalkiewichMillerSroka2024ScissorsAuto}
    For any weak assembler $\CC$, $\wW(\CC)$ is a wide subcategory of itself, giving a category of fractions $\wW(\CC)[\wW(\CC)^{-1}]$.
\end{lemma}
\begin{definition}\label{Def:G(C)}
The \emph{category of scissors congruences} is the category $$\gG(\CC) := \wW(\CC)[\wW(\CC)^{-1}].$$ $\gG(\CC)$ has same objects as $\wW(\CC)$ and its morphisms are equivalence classes of spans $P\xleftarrow{} R \rightarrow Q$ where each arrow is a morphism of $\wW(\CC)$. We shall denote a morphism from $P$ to $Q$ in $\gG(\CC)$ by $$P \xrightarrow{s.c.} Q.$$
\end{definition}
\begin{remark}\label{rem:equiv in GC is refinement in WC}
    In the case of $\gG(\CC)$, $x$ from \eqref{diagram: eqivalence of span} is a refinement of $c$ and $c'$ in $\wW(\CC)$ (\cref{rem: Refinement in WC}) and therefore we call equivalence of two morphisms in $\gG(\CC)$ as \textit{equivalence up to refinement in $\wW(\CC)$} but we often suppress saying in $\wW(\CC)$. Also, the composite of two morphisms is given by Ore localization of middle cospan which is the same as doing the refinement of middle cospan in $\wW(\CC)$ (\cref{rem: Refinement in WC}).
\end{remark}
\begin{remark}\label{rmk: inverse in GC}
    Note that, by construction, $\gG(\CC)$ is a groupoid. The inverse of a morphism $P \xleftarrow[] {f} R \xrightarrow[]{g} Q$ is given by $Q \xleftarrow[]{g} R \xrightarrow[]{f} P$. This can be checked by first composing the two
    % https://q.uiver.app/#q=WzAsMTAsWzAsMiwiUCJdLFsyLDIsIlEiXSxbNCwyLCJQIl0sWzEsMSwiUiJdLFszLDEsIlIiXSxbMiwwLCJSIl0sWzQsMSwiPSJdLFs1LDIsIlAiXSxbNywyLCJQIl0sWzYsMSwiUiJdLFszLDAsImYiXSxbMywxLCJnIiwyXSxbNCwxLCJnIl0sWzQsMiwiZiIsMl0sWzUsMywiPSIsMix7InN0eWxlIjp7ImJvZHkiOnsibmFtZSI6ImRvdHRlZCJ9fX1dLFs1LDQsIj0iLDIseyJzdHlsZSI6eyJib2R5Ijp7Im5hbWUiOiJkb3R0ZWQifX19XSxbOSw3LCJmIiwyXSxbOSw4LCJmIl1d
\[\begin{tikzcd}[cramped, column sep=small, row sep=small]
	&& R \\
	& R && R & {=} && R \\
	P && Q && P & P && P
	\arrow["{=}"', dotted, from=1-3, to=2-2]
	\arrow["{=}"', dotted, from=1-3, to=2-4]
	\arrow["f", from=2-2, to=3-1]
	\arrow["g"', from=2-2, to=3-3]
	\arrow["g", from=2-4, to=3-3]
	\arrow["f"', from=2-4, to=3-5]
	\arrow["f"', from=2-7, to=3-6]
	\arrow["f", from=2-7, to=3-8]
\end{tikzcd}\]
and then checking that the composite is in the same class as the identity morphism of $P$ in $\gG(\CC)$:
% https://q.uiver.app/#q=WzAsNSxbMCwxLCJQIl0sWzEsMCwiUiJdLFsyLDEsIlAiXSxbMSwyLCJQIl0sWzEsMSwiUiJdLFsxLDAsImYiLDJdLFsxLDIsImYiXSxbMywwLCI9Il0sWzMsMiwiPSIsMl0sWzQsMSwiPSJdLFs0LDMsImYiXV0=
\[\begin{tikzcd}[cramped]
	& R \\
	P & R & P \\
	& P
	\arrow["f"', from=1-2, to=2-1]
	\arrow["f", from=1-2, to=2-3]
	\arrow["{=}", from=2-2, to=1-2]
	\arrow["f", from=2-2, to=3-2]
	\arrow["{=}", from=3-2, to=2-1]
	\arrow["{=}"', from=3-2, to=2-3]
\end{tikzcd}\]
Moreover, we note that the localization functor 
\begin{equation}
    \begin{aligned}
        L_{W} : \wW(\CC) & \to \gG(\CC) \label{eq:W(C)_to_G(C)}\\
        P \xrightarrow[]{f}Q & \mapsto P\xleftarrow[]{=}P \xrightarrow[]{f}Q
    \end{aligned}
\end{equation}
is symmetric monoidal.
\end{remark}
For a weak assembler, one can produce a $K$-theory spectrum \cite{ZAKHAREVICH20171176}. The associated infinite loop space of this spectrum can be understood as a group completion of the groupoid $\gG(\CC)$ with respect to its symmetric monoidal operation.
\begin{definition}\label{def:Assemblers_K_theory}\cite[Theorem 4.1, $K$-theory as group completion]{KupersLemannMalkiewichMillerSroka2024ScissorsAuto}
For a weak assembler $\CC$, define the \textit{$K$-theory space} of $\CC$, $K(\CC)$, as the group completion of $\gG(\CC)$:
$$K(\CC) := \Omega B|B\gG(\CC)|.$$
Equivalently, $K(\CC):= \Omega B|B\wW(\CC)|$.
\end{definition}
\begin{remark}\label{rmk: K0}
    $K_0$ group is generated by scissors congruence classes of objects $[Q]$ in $\CC$ via Grothendieck group completion.
\end{remark}

%%%%%%%%%%%%%%%%%%%NEW SECTION%%%%%%%%%%%%%%%%%%%%%%%%%%%%%%%%%%
\section{The Category of Scissors Correspondences}\label{subsec: ZMC}
The category of finitely generated projective modules, $\Proj_R$, is crucial towards the construction of Dennis trace map for rings. In this section we define an analogue of it for a weak assembler $\CC$. The goal is for this category to be enriched over $Ab$ and have general endomorphisms between any two objects that are non zero and not automorphisms. 

\begin{definition}\label{def: partial cover}
Let $P, Q \in \obj\wW(\CC)$. A \textit{partial cover} from $P$ to $Q$, denoted by $$\overline{f}_{P'}:P \xhookrightarrow{} Q,$$ is a choice of object $P' \in \obj\wW(\CC)$ and a cover $f:P \sqcup P' \to Q$ in $\wW(\CC)$.
\end{definition}
\begin{remark}\label{rem: cover is partial cover}
    Note that a cover $f: P\to Q$ always canonically gives a partial cover for the choice of $P' = \emptyset$, the resulting partial cover is $\overline{f}_{\emptyset}: P\emptyset \cong P \xhookrightarrow{} Q$.
\end{remark}
\begin{definition}\label{def: category of partial covers}
 The \textit{category of partial covers}, $\wW^{par}(\CC)$, has objects same as $\obj \wW(\CC)$ and a morphism from $P$ to $Q$ is a partial cover $\overline{f}_{P'}: P \xhookrightarrow{}Q$. Given partial covers $\overline{f}_{P'}: P \xhookrightarrow{} Q$ and $\overline{g}_{Q'}: Q \xhookrightarrow{} R$, we get the following composite in $\wW(\CC)$
 % https://q.uiver.app/#q=WzAsMyxbMCwwLCJQUCdRJyJdLFsxLDAsIlFRJyJdLFsyLDAsIlIiXSxbMCwxLCJmIGlkX3tRJ30iXSxbMSwyLCJnIl1d
\[\begin{tikzcd}[cramped]
	{PP'Q'} & {QQ'} & R
	\arrow["{f id_{Q'}}", from=1-1, to=1-2]
	\arrow["g", from=1-2, to=1-3],
\end{tikzcd}\]
giving us a partial cover $$(\overline{g\circ f id_{Q'}})_{P'Q'}: P \xhookrightarrow{} R.$$ This is the composite of the two partial covers we started with in $\wW^{par}(\CC)$. The identity partial cover from $P$ to itself is $$(\overline{id_P})_{\emptyset}: P\emptyset \cong P \xhookrightarrow{id_P} P.$$ This is a symmetric monoidal category with monoidal operation given by concatenation and there is a faithful symmetric monoidal functor from $\wW(\CC)$ to $\wW^{par}(\CC)$ which is identity on objects and on morphisms sends $f$ to $\overline{f}_{\emptyset}$. We sometimes take a shortcut and say a morphism of $\wW^{par}(\CC)$ is in $\wW(\CC)$ if this morphism is in the image of this faithful functor. 
\end{definition}
Similar to the construction of $\gG(\CC)$ from $\wW(\CC)$, we would like to define a corresponding category of spans for $\wW^{par}(\CC)$ where all the arrows of $\wW^{par}(\CC)$ are inverted. It is tempting to do an analogous category of fractions argument, unfortunately that does not work and we have to define the construction and check all the details explicitly. See \cref{warn: MC not category of fractions} for more details on this. We call this category $\mM(\CC)$ and work now on defining it. First we need a remark that will help define composition in $\mM(\CC)$. 
\begin{remark}\label{rem:completing cospan of partial covers}
Given a cospan of partial covers $$P \xhookrightarrow{\overline{f}_{P'}} Q \xhookleftarrow{\overline{g}_{R'}} R,$$ we get the left commutative diagram below, of morphisms in $\wW(\CC)$, by looking at the corresponding cospan of morphisms in $\wW(\CC)$ and then taking their refinement as in \cref{rem: Refinement in WC}. This gives a commutative diagram of partial covers in $\wW^{par}(\CC)$ on the right below:
% https://q.uiver.app/#q=WzAsOCxbMCwxLCJQIFAnIl0sWzEsMiwiUSJdLFsyLDEsIlIgUiciXSxbMSwwLCJTX3tQXFxjYXAgUn0gXFwgU197UFxcY2FwIFInfSBcXFxcIFNfe1AnIFxcY2FwIFJ9XFwgU197UCcgXFxjYXAgUid9Il0sWzMsMSwiUCJdLFs1LDEsIlIiXSxbNCwyLCJRIl0sWzQsMCwiU197UFxcY2FwIFJ9Il0sWzMsMCwiaCIsMl0sWzMsMiwiayJdLFsyLDEsImciXSxbMCwxLCJmIiwyXSxbNyw0LCJcXG92ZXJsaW5le2h9X3tTX3tQXFxjYXAgUid9fSIsMix7InN0eWxlIjp7InRhaWwiOnsibmFtZSI6Imhvb2siLCJzaWRlIjoiYm90dG9tIn19fV0sWzcsNSwiXFxvdmVybGluZXtrfV97U197UCdcXGNhcCBSfX0iLDAseyJzdHlsZSI6eyJ0YWlsIjp7Im5hbWUiOiJob29rIiwic2lkZSI6InRvcCJ9fX1dLFs0LDYsIlxcb3ZlcmxpbmV7Zn1fe1AnfSIsMix7InN0eWxlIjp7InRhaWwiOnsibmFtZSI6Imhvb2siLCJzaWRlIjoidG9wIn19fV0sWzUsNiwiXFxvdmVybGluZXtnfV97Uid9IiwwLHsic3R5bGUiOnsidGFpbCI6eyJuYW1lIjoiaG9vayIsInNpZGUiOiJib3R0b20ifX19XSxbMiw0LCIiLDEseyJsZXZlbCI6Mn1dXQ==
\[\begin{tikzcd}[cramped, column sep=small, row sep=small]
	& \begin{array}{c} S_{P\cap R} \ S_{P\cap R'} \\ S_{P' \cap R}\ S_{P' \cap R'} \end{array} &&& {S_{P\cap R}} & \\
	{P P'} && {R R'} & P && R \\
	& Q &&& Q
	\arrow["h"', from=1-2, to=2-1]
	\arrow["k", from=1-2, to=2-3]
	\arrow["{\overline{h}_{S_{P\cap R'}}}"', hook', from=1-5, to=2-4]
	\arrow["{\overline{k}_{S_{P'\cap R}}}", hook, from=1-5, to=2-6]
	\arrow["f"', from=2-1, to=3-2]
	\arrow[Rightarrow, from=2-3, to=2-4]
	\arrow["g", from=2-3, to=3-2]
	\arrow["{\overline{f}_{P'}}"', hook, from=2-4, to=3-5]
	\arrow["{\overline{g}_{R'}}", hook', from=2-6, to=3-5]
\end{tikzcd}\]
\end{remark}
We now define a new category denoted as $\mM(\CC)$ where we invert all the partial covers and consider their spans. This category will be then modified to construct the category of scissors correspondences.
\begin{definition}\label{def: unenriched scissors correspondences}
    The \emph{category $\mM(\CC)$} has objects same as $\obj\wW(\CC)$ and a morphism $f: P\to Q$, $P,Q \in \obj\mM(\CC)$, is an equivalence class of spans of partial covers $P \xhookleftarrow{} R \xhookrightarrow{} Q$, $R \in \obj\mM(\CC)$, up to refinement in $\wW(\CC)$. Note that as in \cref{rem:equiv in GC is refinement in WC}, by up to refinement in $\wW(\CC)$, we mean the spans $P \xhookleftarrow{} R \xhookrightarrow{} Q$ and $P \xhookleftarrow{} R' \xhookrightarrow{} Q$ are equivalent if there exists $S \in \obj\wW(\CC)$ along with a commuting diagram as follows in $\wW^{par}(\CC)$
    % https://q.uiver.app/#q=WzAsNSxbMCwxLCJQIl0sWzEsMCwiUiJdLFsxLDIsIlInIl0sWzIsMSwiUSJdLFsxLDEsIlMiXSxbMSwwLCIiLDAseyJzdHlsZSI6eyJ0YWlsIjp7Im5hbWUiOiJob29rIiwic2lkZSI6ImJvdHRvbSJ9fX1dLFsxLDMsIiIsMix7InN0eWxlIjp7InRhaWwiOnsibmFtZSI6Imhvb2siLCJzaWRlIjoidG9wIn19fV0sWzIsMCwiIiwyLHsic3R5bGUiOnsidGFpbCI6eyJuYW1lIjoiaG9vayIsInNpZGUiOiJ0b3AifX19XSxbMiwzLCIiLDAseyJzdHlsZSI6eyJ0YWlsIjp7Im5hbWUiOiJob29rIiwic2lkZSI6ImJvdHRvbSJ9fX1dLFs0LDIsIlxcaW4gXFxtYXRoY2Fse1d9KFxcbWF0aGNhbHtDfSkiXSxbNCwxLCJcXGluIFxcbWF0aGNhbHtXfShcXG1hdGhjYWx7Q30pIiwyXV0=
\[\begin{tikzcd}[cramped]
	& R & \\
	P & S & Q \\
	& {R'}
	\arrow[hook', from=1-2, to=2-1]
	\arrow[hook, from=1-2, to=2-3]
	\arrow["{\in \mathcal{W}(\mathcal{C})}"', from=2-2, to=1-2]
	\arrow["{\in \mathcal{W}(\mathcal{C})}", from=2-2, to=3-2]
	\arrow[hook, from=3-2, to=2-1]
	\arrow[hook', from=3-2, to=2-3]
\end{tikzcd}\]
Similar to \cref{rem:equiv in GC is refinement in WC} we simply call this as \textit{equivalence up to refinement}. The composition is defined by completing the cospan of the partial covers as in \cref{rem:completing cospan of partial covers} and is well defined by \cref{rem: well defined} below. The identity morphism of a object $P$ in $\mM(\CC)$ is given by the class
$$P \xhookleftarrow{\overline{id_P}_{\emptyset}} P \xhookrightarrow{\overline{id_P}_{\emptyset}} P$$ 
and we denote is as $id: P\to P$.
The category $\mM(\CC)$ has a symmetric monoidal structure (by concatenation). Most of the times, when it is clear from context, we will just concatenate objects of $\mM(\CC)$ to denote their product. But, when we need to highlight the product, we use the symbol $\odot$ to denote it. 
    \end{definition}

\begin{remark}\label{rem: well defined}
To check that composition is well defined in $\mM(\CC)$, consider a different representative from equivalence class of two morphisms:
% https://q.uiver.app/#q=WzAsMTEsWzAsMSwiUCJdLFsxLDAsIlIiXSxbMiwxLCJRIl0sWzQsMSwiUyJdLFszLDAsIlQiXSxbMywyLCJUJyJdLFsxLDIsIlInIl0sWzEsMSwiXFx0ZXh0Y29sb3J7Ymx1ZX17UicnfSJdLFszLDEsIlxcdGV4dGNvbG9ye2JsdWV9e1QnJ30iXSxbMiwwLCJcXHRleHRjb2xvcntyZWR9e1V9Il0sWzIsMiwiXFx0ZXh0Y29sb3J7cmVkfXtVJ30iXSxbMSwwLCIiLDIseyJzdHlsZSI6eyJ0YWlsIjp7Im5hbWUiOiJob29rIiwic2lkZSI6ImJvdHRvbSJ9fX1dLFsxLDIsIiIsMCx7InN0eWxlIjp7InRhaWwiOnsibmFtZSI6Imhvb2siLCJzaWRlIjoidG9wIn19fV0sWzYsMiwiIiwwLHsic3R5bGUiOnsidGFpbCI6eyJuYW1lIjoiaG9vayIsInNpZGUiOiJib3R0b20ifX19XSxbNiwwLCIiLDIseyJzdHlsZSI6eyJ0YWlsIjp7Im5hbWUiOiJob29rIiwic2lkZSI6InRvcCJ9fX1dLFs1LDMsIiIsMCx7InN0eWxlIjp7InRhaWwiOnsibmFtZSI6Imhvb2siLCJzaWRlIjoiYm90dG9tIn19fV0sWzQsMywiIiwyLHsic3R5bGUiOnsidGFpbCI6eyJuYW1lIjoiaG9vayIsInNpZGUiOiJ0b3AifX19XSxbNywxXSxbNyw2XSxbOCw0XSxbOCw1XSxbNCwyLCIiLDAseyJzdHlsZSI6eyJ0YWlsIjp7Im5hbWUiOiJob29rIiwic2lkZSI6ImJvdHRvbSJ9fX1dLFs1LDIsIiIsMix7InN0eWxlIjp7InRhaWwiOnsibmFtZSI6Imhvb2siLCJzaWRlIjoidG9wIn19fV0sWzEwLDYsIiIsMCx7InN0eWxlIjp7InRhaWwiOnsibmFtZSI6Imhvb2siLCJzaWRlIjoidG9wIn19fV0sWzEwLDUsIiIsMCx7InN0eWxlIjp7InRhaWwiOnsibmFtZSI6Imhvb2siLCJzaWRlIjoiYm90dG9tIn19fV0sWzksMSwiIiwwLHsic3R5bGUiOnsidGFpbCI6eyJuYW1lIjoiaG9vayIsInNpZGUiOiJib3R0b20ifX19XSxbOSw0LCIiLDAseyJzdHlsZSI6eyJ0YWlsIjp7Im5hbWUiOiJob29rIiwic2lkZSI6InRvcCJ9fX1dXQ==
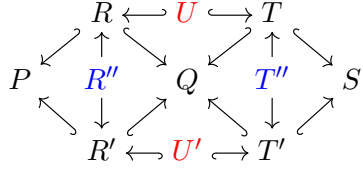
\[\begin{tikzcd}[cramped, column sep=small, row sep=small]
	& R & {\textcolor{red}{U}} & T & \\
	P & {\textcolor{blue}{R''}} & Q & {\textcolor{blue}{T''}} & S \\
	& {R'} & {\textcolor{red}{U'}} & {T'}
	\arrow[hook', from=1-2, to=2-1]
	\arrow[hook, from=1-2, to=2-3]
	\arrow[hook', from=1-3, to=1-2]
	\arrow[hook, from=1-3, to=1-4]
	\arrow[hook', from=1-4, to=2-3]
	\arrow[hook, from=1-4, to=2-5]
	\arrow[from=2-2, to=1-2]
	\arrow[from=2-2, to=3-2]
	\arrow[from=2-4, to=1-4]
	\arrow[from=2-4, to=3-4]
	\arrow[hook, from=3-2, to=2-1]
	\arrow[hook', from=3-2, to=2-3]
	\arrow[hook, from=3-3, to=3-2]
	\arrow[hook', from=3-3, to=3-4]
	\arrow[hook, from=3-4, to=2-3]
	\arrow[hook', from=3-4, to=2-5]
\end{tikzcd}\]
where $R''$ and $T''$ witness the equivalence and $U$ and $U'$ are the composition as defined in \cref{rem:completing cospan of partial covers}. The cospan of partial covers $R''\xhookrightarrow{}  Q \xhookleftarrow{} T''$ can be completed by a partial cover $U''$ to give 
% https://q.uiver.app/#q=WzAsNCxbMCwxLCJSJyciXSxbMiwxLCJUJyciXSxbMSwxLCJRIl0sWzEsMCwiVScnIl0sWzAsMiwiIiwwLHsic3R5bGUiOnsidGFpbCI6eyJuYW1lIjoiaG9vayIsInNpZGUiOiJ0b3AifX19XSxbMSwyLCIiLDIseyJzdHlsZSI6eyJ0YWlsIjp7Im5hbWUiOiJob29rIiwic2lkZSI6ImJvdHRvbSJ9fX1dLFszLDAsIiIsMCx7InN0eWxlIjp7InRhaWwiOnsibmFtZSI6Imhvb2siLCJzaWRlIjoiYm90dG9tIn19fV0sWzMsMSwiIiwyLHsic3R5bGUiOnsidGFpbCI6eyJuYW1lIjoiaG9vayIsInNpZGUiOiJ0b3AifX19XV0=
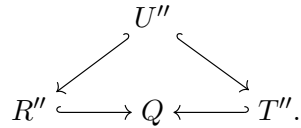
\[\begin{tikzcd}[cramped]
	& {U''} & \\
	{R''} & Q & {T''.}
	\arrow[hook', from=1-2, to=2-1]
	\arrow[hook, from=1-2, to=2-3]
	\arrow[hook, from=2-1, to=2-2]
	\arrow[hook', from=2-3, to=2-2]
\end{tikzcd}\]
As $U''$ partially refines both $R''$ and $T''$ which are both finer covers of $R$ and $T$ (in $\wW(\CC)$) and completes the cospan $R'' \xhookrightarrow{} Q \xleftarrow{} T''$, $U''$ is a finer cover of $U$. Similarly, $U''$ is also a finer cover of $U'$. Hence one gets the required commuting diagram establishing the well definedness of composition in $\mM(\CC)$.
\end{remark}

\begin{notation}
    When writing a morphism of $\mM(\CC)$ in its full zigzag form, if a leg of the zig-zag is a cover rather than a partial cover, we might write that leg as an element of $\Mor(\wW(\CC))$ and use just the symbol $\to$ instead of $\hookrightarrow$ for it. 
 \end{notation} 
 \begin{remark}
     We call a morphism $i: P \to Q$ in $\mM(\CC)$ an \textit{inclusion} if it can be written as a zigzag $$i: P \xleftarrow{\in \wW(\CC)} P' \xhookrightarrow{} Q$$ where the first arrow is a cover. Similarly, a morphism $\pi: P \to Q$ is called a \textit{projection} if it can be written as a zigzag $$\pi: P \xhookleftarrow{} Q' \xrightarrow{\in \wW(\CC)} Q$$ where the second arrow is a cover.
     The identity map $$id_{R^{\odot n}}: \underbrace{R R \dots R}_{n \text{ times}} \to \underbrace{R R \dots R}_{n \text{ times}}$$ gives $n$ inclusions of $R$ into $R^{\odot n}$, which we also write as just the concatenation $R^n$, depending on which `coordinate' you choose $R$ to include into. We denote inclusion into the $k^{th}$ coordinate by $i_k$. Similarly there are $n$ projections from $R^n$ to $R$ and the projection onto the $k^{th}$ coordinate is denoted by $\pi_k$.
 \end{remark}
 \begin{remark}\label{rem: projection is retract}
 Note that a projection map is a retract. Given a projection $\pi: P \xhookleftarrow{} Q' \xrightarrow{} Q$, we have an inclusion $i: Q \xleftarrow{} Q' \xhookrightarrow{} P$ and $\pi \circ i$ gives $Q \xleftarrow{} Q' \xrightarrow{} Q$ which up to refinement is the identity map from $Q$ to itself.
 \end{remark}
\begin{remark}
    Note that there are symmetric monoidal functors $$\wW(\CC) \to \wW^{par}(\CC)\to \mM(\CC)$$ where the last functor sends an object to itself and a morphism $\overline{f}_{P'}: P \xhookrightarrow{} Q$ to $$P \xhookleftarrow{\overline{id}_{\emptyset}} P \xhookrightarrow{\overline{f}_{P'}} Q.$$
    Also, there exist symmetric monoidal functors
    \begin{equation}\label{eq:WC to GC to MC}
    \wW(\CC) \to \gG(\CC)\to \mM(\CC)
    \end{equation}
    where the last functor sends an object to itself and a morphism $P \xleftarrow[]{f}Q \xrightarrow[]{g}R$ to 
    $$P \xhookleftarrow[]{\overline{f}_{\emptyset}}Q \xhookrightarrow[]{\overline{g}_{\emptyset}}R.$$
    \end{remark}
    \begin{remark}\label{rem: isos in MC}
     Note that an isomorphism in $\mM(\CC)$ is just a morphism that is in the image of the functor $\gG(\CC) \to \mM(\CC)$ from \eqref{eq:WC to GC to MC}.
 \end{remark}
\begin{warning}\label{warn: MC not category of fractions}
 $\wW^{par}(\CC)$ is a wide subcategory of itself and therefore one can define a category of fractions $\wW^{par}(\CC)[\wW^{par}(\CC)^{-1}]$. The objects here are same as $\obj \mM(\CC)$ and morphisms are an equivalence class of spans of partial covers. However, this equivalence class of spans is different (and smaller) than the equivalence class of spans that make a morphism in $\mM(\CC)$. In fact, given two spans from $P$ to $R$, they are always equivalent in $\wW^{par}(\CC)[\wW^{par}(\CC)^{-1}]$ via the commutative diagram
 % https://q.uiver.app/#q=WzAsNSxbMCwxLCJQIl0sWzEsMCwiUSJdLFsyLDEsIlIiXSxbMSwyLCJRJyJdLFsxLDEsIlxccGhpIl0sWzEsMCwiIiwwLHsic3R5bGUiOnsidGFpbCI6eyJuYW1lIjoiaG9vayIsInNpZGUiOiJib3R0b20ifX19XSxbMSwyLCIiLDIseyJzdHlsZSI6eyJ0YWlsIjp7Im5hbWUiOiJob29rIiwic2lkZSI6InRvcCJ9fX1dLFszLDAsIiIsMix7InN0eWxlIjp7InRhaWwiOnsibmFtZSI6Imhvb2siLCJzaWRlIjoidG9wIn19fV0sWzMsMiwiIiwwLHsic3R5bGUiOnsidGFpbCI6eyJuYW1lIjoiaG9vayIsInNpZGUiOiJib3R0b20ifX19XSxbNCwxLCIiLDAseyJzdHlsZSI6eyJ0YWlsIjp7Im5hbWUiOiJob29rIiwic2lkZSI6ImJvdHRvbSJ9fX1dLFs0LDMsIiIsMCx7InN0eWxlIjp7InRhaWwiOnsibmFtZSI6Imhvb2siLCJzaWRlIjoidG9wIn19fV1d
\[\begin{tikzcd}[cramped, column sep=small, row sep=small]
	& Q & \\
	P & \emptyset & R. \\
	& {Q'}
	\arrow[hook', from=1-2, to=2-1]
	\arrow[hook, from=1-2, to=2-3]
	\arrow[hook', from=2-2, to=1-2]
	\arrow[hook, from=2-2, to=3-2]
	\arrow[hook, from=3-2, to=2-1]
	\arrow[hook', from=3-2, to=2-3]
\end{tikzcd}\]
Therefore, all spans from $P$ to $R$ are equivalent to the zero span $P\xhookleftarrow{} \emptyset\xhookrightarrow{} R$ in $\wW^{par}(\CC)[\wW^{par}(\CC)^{-1}]$.
\end{warning}
We now construct the category of scissors correspondences.
\begin{construction}\label{def: category of scissors correspondences}
\begin{enumerate} The category of scissors correspondences is constructed via the following series of steps:
    \item\label{item: 1} First, enrich the morphism sets of $\mM(\CC)$ freely over abelian groups, i.e., consider $\bbZ\langle\Mor_{\mM(\CC)}(P,Q)\rangle$ for $P,Q \in \obj(\mM(\CC))$ and then reduce this abelian group by identifying
    % https://q.uiver.app/#q=WzAsNSxbMSwwLCJcXHBoaSJdLFswLDEsIlAiXSxbMiwxLCJRIl0sWzMsMSwiPSJdLFs0LDEsIjAuIl0sWzAsMSwiIiwwLHsic3R5bGUiOnsidGFpbCI6eyJuYW1lIjoiaG9vayIsInNpZGUiOiJib3R0b20ifX19XSxbMCwyLCIiLDIseyJzdHlsZSI6eyJ0YWlsIjp7Im5hbWUiOiJob29rIiwic2lkZSI6InRvcCJ9fX1dXQ==
\[\begin{tikzcd}[cramped, column sep=small, row sep=small]
	& \emptyset &&& \\
	P && Q
	\arrow[hook', from=1-2, to=2-1]
	\arrow[hook, from=1-2, to=2-3]
\end{tikzcd}
=  \ \ \ \ \ \ \ \ \ \ 0.
\]
    Note that these reduced groups are still free abelian groups but now they are pointed. For the rest of this construction we write $\Hom$ (without any subscript label) for these reduced free abelian groups.
    \item For $P, Q, R, T_1, T_2 \in \obj(\mM(\CC))$, we now have following two kinds of group homomorphisms. The first one is
    \begin{equation*}
        F: \Hom(PQ,R) \rightarrow \Hom(P,R) \oplus \Hom(Q,R)
    \end{equation*}
    given on a generator on the left by taking its precomposite with inclusion maps $$P \xleftarrow{id} P \xhookrightarrow{(\overline{id_{PQ}})_{Q}} PQ \text{  and  } Q \xleftarrow{id} Q \xhookrightarrow{(\overline{id_{PQ}})_{P}} PQ$$ to get:
\begin{equation}\label{eq: split}
    % https://q.uiver.app/#q=WzAsMyxbMCwxLCJQUSJdLFsxLDAsIlQiXSxbMiwxLCJSIl0sWzEsMCwiIiwwLHsic3R5bGUiOnsidGFpbCI6eyJuYW1lIjoiaG9vayIsInNpZGUiOiJib3R0b20ifX19XSxbMSwyLCIiLDIseyJzdHlsZSI6eyJ0YWlsIjp7Im5hbWUiOiJob29rIiwic2lkZSI6InRvcCJ9fX1dXQ==
\begin{tikzcd}[cramped, column sep=small, row sep=small]
	& T & \\
	PQ && R
	\arrow[hook', from=1-2, to=2-1]
	\arrow[hook, from=1-2, to=2-3]
\end{tikzcd}
\longmapsto
% https://q.uiver.app/#q=WzAsMyxbMCwxLCJQIl0sWzIsMSwiUiJdLFsxLDAsIlNfe1BcXGNhcCBUfSJdLFsyLDAsIiIsMix7InN0eWxlIjp7InRhaWwiOnsibmFtZSI6Imhvb2siLCJzaWRlIjoiYm90dG9tIn19fV0sWzIsMSwiIiwwLHsic3R5bGUiOnsidGFpbCI6eyJuYW1lIjoiaG9vayIsInNpZGUiOiJ0b3AifX19XV0=
\begin{tikzcd}[cramped, column sep=small, row sep=small]
	& {S_{P\cap T}} & \\
	P && R
	\arrow[hook', from=1-2, to=2-1]
	\arrow[hook, from=1-2, to=2-3]
\end{tikzcd}
+
\begin{tikzcd}[cramped, column sep=small, row sep=small]
	& {S_{Q\cap T}} & \\
	Q && R
	\arrow[hook', from=1-2, to=2-1]
	\arrow[hook, from=1-2, to=2-3]
\end{tikzcd}
\end{equation}
Note that $S_{P\cap T} \odot S_{Q\cap T} \cong T$ in $\mM(\CC)$ and the map $T \hookrightarrow PQ$ gives $S_{P\cap T} \hookrightarrow P$ and $S_{Q\cap T} \hookrightarrow Q$. 
The second group homomorphism is
 \begin{equation*}
    G: \Hom(P,R) \oplus \Hom(Q,R) \rightarrow \Hom(PQ,RR) \rightarrow \Hom(PQ,R) 
    \end{equation*} 
    given on a set of generators by first taking their symmetric monoidal product giving an element in $\Hom(PQ,RR)$:
\begin{equation}\label{eq: product}
% https://q.uiver.app/#q=WzAsMyxbMCwxLCJQIl0sWzIsMSwiUiJdLFsxLDAsIlRfMSJdLFsyLDAsIiIsMCx7InN0eWxlIjp7InRhaWwiOnsibmFtZSI6Imhvb2siLCJzaWRlIjoiYm90dG9tIn19fV0sWzIsMSwiIiwyLHsic3R5bGUiOnsidGFpbCI6eyJuYW1lIjoiaG9vayIsInNpZGUiOiJ0b3AifX19XV0=
\begin{tikzcd}[cramped, column sep=small, row sep=small]
	& {T_1} & \\
	P && R
	\arrow[hook', from=1-2, to=2-1]
	\arrow[hook, from=1-2, to=2-3]
\end{tikzcd}
+
% https://q.uiver.app/#q=WzAsMyxbMCwxLCJRIl0sWzEsMCwiVF8yIl0sWzIsMSwiUiJdLFsxLDAsIiIsMix7InN0eWxlIjp7InRhaWwiOnsibmFtZSI6Imhvb2siLCJzaWRlIjoiYm90dG9tIn19fV0sWzEsMiwiIiwwLHsic3R5bGUiOnsidGFpbCI6eyJuYW1lIjoiaG9vayIsInNpZGUiOiJ0b3AifX19XV0=
\begin{tikzcd}[cramped, column sep=small, row sep=small]
	& {T_2} & \\
	Q && R
	\arrow[hook', from=1-2, to=2-1]
	\arrow[hook, from=1-2, to=2-3]
\end{tikzcd}
\longmapsto
% https://q.uiver.app/#q=WzAsMyxbMCwxLCJQUSJdLFsxLDAsIlRfMVRfMiJdLFsyLDEsIlJSLiJdLFsxLDAsIiIsMCx7InN0eWxlIjp7InRhaWwiOnsibmFtZSI6Imhvb2siLCJzaWRlIjoiYm90dG9tIn19fV0sWzEsMiwiIiwyLHsic3R5bGUiOnsidGFpbCI6eyJuYW1lIjoiaG9vayIsInNpZGUiOiJ0b3AifX19XV0=
\begin{tikzcd}[cramped, column sep=small, row sep=small]
	& {T_1T_2} & \\
	PQ && {RR.}
	\arrow[hook', from=1-2, to=2-1]
	\arrow[hook, from=1-2, to=2-3]
\end{tikzcd}
\end{equation}
Note that here $T_1$ is mapping into $P$ and first coordinate of $RR$ and $T_2$ is mapping into $Q$ and second coordinate of $RR$. And then taking the sum of the postcomposites of this element with projection into the first coordinate and projection into the second coordinate of $RR$:
\[
\begin{adjustbox}{max width=\textwidth}
$\displaystyle
% https://q.uiver.app/#q=WzAsMyxbMSwwLCJUXzFUXzIiXSxbMCwxLCJQUSJdLFsyLDEsIlJSIl0sWzAsMSwiIiwwLHsic3R5bGUiOnsidGFpbCI6eyJuYW1lIjoiaG9vayIsInNpZGUiOiJib3R0b20ifX19XSxbMCwyLCIiLDIseyJzdHlsZSI6eyJ0YWlsIjp7Im5hbWUiOiJob29rIiwic2lkZSI6InRvcCJ9fX1dXQ==
%\begin{tikzcd}[cramped, column sep=small, row sep=tiny]
	%& {T_1T_2} & \\
	%PQ && RR
	%\arrow[hook', from=1-2, to=2-1]
	%\arrow[hook, from=1-2, to=2-3]
%\end{tikzcd}
%\!\longmapsto\!
% https://q.uiver.app/#q=WzAsNixbMCwxLCJQUSJdLFsxLDAsIlRfMVRfMiJdLFsyLDEsIlJSIl0sWzMsMCwiUiJdLFs0LDEsIlIiXSxbMywxLCJcXHVuZGVyYnJhY2V7fV97XFxwaV8xfSJdLFszLDIsIiIsMix7InN0eWxlIjp7InRhaWwiOnsibmFtZSI6Imhvb2siLCJzaWRlIjoiYm90dG9tIn19fV0sWzMsNCwiIiwwLHsic3R5bGUiOnsidGFpbCI6eyJuYW1lIjoiaG9vayIsInNpZGUiOiJ0b3AifX19XSxbMSwwLCIiLDAseyJzdHlsZSI6eyJ0YWlsIjp7Im5hbWUiOiJob29rIiwic2lkZSI6ImJvdHRvbSJ9fX1dLFsxLDIsIiIsMSx7InN0eWxlIjp7InRhaWwiOnsibmFtZSI6Imhvb2siLCJzaWRlIjoidG9wIn19fV1d
\begin{tikzcd}[cramped, column sep=small, row sep=small]
	& {T_1T_2} && R & \\
	PQ && RR & {\underbrace{}_{\pi_1}} & R
	\arrow[hook', from=1-2, to=2-1]
	\arrow[hook, from=1-2, to=2-3]
	\arrow[hook', from=1-4, to=2-3]
	\arrow[hook, from=1-4, to=2-5]
\end{tikzcd}
\ + \
\begin{tikzcd}[cramped, column sep=small, row sep=small]
	& {T_1T_2} && R & \\
	PQ && RR & {\underbrace{}_{\pi_2}} & R
	\arrow[hook', from=1-2, to=2-1]
	\arrow[hook, from=1-2, to=2-3]
	\arrow[hook', from=1-4, to=2-3]
	\arrow[hook, from=1-4, to=2-5]
\end{tikzcd}
$
\end{adjustbox}
\]
Looking at the composites $FG$ and $GF$, we see that
\begin{align*}
  FG(P \xhookleftarrow{} T_1 \xhookrightarrow{} R, Q \xhookleftarrow{} T_2 \xhookrightarrow{} R) &\cong (P \xhookleftarrow{} T_1 \xhookrightarrow{} R, 0) + (0, Q \xhookleftarrow{} T_2 \xhookrightarrow{} R)\\ &\cong (P \xhookleftarrow{} T_1 \xhookrightarrow{} R, Q \xhookleftarrow{} T_2 \xhookrightarrow{} R)  
\end{align*}
using the note below \eqref{eq: product} and the identification from \cref{def: category of scissors correspondences}(\ref{item: 1}). So, we have $FG \cong Id$. Whereas, 
$GF(PQ \xhookleftarrow{} T \xhookrightarrow{} R)$ is
% https://q.uiver.app/#q=WzAsNixbMCwxLCJQUSJdLFsxLDAsIlNfe1BcXGNhcCBUfVNfe1FcXGNhcCBUfSJdLFsyLDEsIlJSIl0sWzMsMCwiUiJdLFs0LDEsIlIiXSxbMywxLCJcXHVuZGVyYnJhY2V7fV97XFxwaV8xfSJdLFszLDIsIiIsMix7InN0eWxlIjp7InRhaWwiOnsibmFtZSI6Imhvb2siLCJzaWRlIjoiYm90dG9tIn19fV0sWzMsNCwiIiwwLHsic3R5bGUiOnsidGFpbCI6eyJuYW1lIjoiaG9vayIsInNpZGUiOiJ0b3AifX19XSxbMSwwLCIiLDAseyJzdHlsZSI6eyJ0YWlsIjp7Im5hbWUiOiJob29rIiwic2lkZSI6ImJvdHRvbSJ9fX1dLFsxLDIsIiIsMSx7InN0eWxlIjp7InRhaWwiOnsibmFtZSI6Imhvb2siLCJzaWRlIjoidG9wIn19fV1d
\[\begin{tikzcd}[cramped, column sep=small, row sep=small]
	& {S_{P\cap T}S_{Q\cap T}} && R & \\
	PQ && RR & {\underbrace{}_{\pi_1}} & R
	\arrow[hook', from=1-2, to=2-1]
	\arrow[hook, from=1-2, to=2-3]
	\arrow[hook', from=1-4, to=2-3]
	\arrow[hook, from=1-4, to=2-5]
\end{tikzcd}
+
% https://q.uiver.app/#q=WzAsNixbMCwxLCJQUSJdLFsxLDAsIlNfe1BcXGNhcCBUfVNfe1FcXGNhcCBUfSJdLFsyLDEsIlJSIl0sWzMsMCwiUiJdLFs0LDEsIlIiXSxbMywxLCJcXHVuZGVyYnJhY2V7fV97XFxwaV8yfSJdLFszLDIsIiIsMix7InN0eWxlIjp7InRhaWwiOnsibmFtZSI6Imhvb2siLCJzaWRlIjoiYm90dG9tIn19fV0sWzMsNCwiIiwwLHsic3R5bGUiOnsidGFpbCI6eyJuYW1lIjoiaG9vayIsInNpZGUiOiJ0b3AifX19XSxbMSwwLCIiLDAseyJzdHlsZSI6eyJ0YWlsIjp7Im5hbWUiOiJob29rIiwic2lkZSI6ImJvdHRvbSJ9fX1dLFsxLDIsIiIsMSx7InN0eWxlIjp7InRhaWwiOnsibmFtZSI6Imhvb2siLCJzaWRlIjoidG9wIn19fV1d
\begin{tikzcd}[cramped, column sep=small, row sep=small]
	& {S_{P\cap T}S_{Q\cap T}} && R & \\
	PQ && RR & {\underbrace{}_{\pi_2}} & R
	\arrow[hook', from=1-2, to=2-1]
	\arrow[hook, from=1-2, to=2-3]
	\arrow[hook', from=1-4, to=2-3]
	\arrow[hook, from=1-4, to=2-5]
\end{tikzcd}\]
which is isomorphic to
% https://q.uiver.app/#q=WzAsMyxbMCwxLCJQUSJdLFsxLDAsIlNfe1BcXGNhcCBUfSJdLFsyLDEsIlIiXSxbMSwwLCIiLDAseyJzdHlsZSI6eyJ0YWlsIjp7Im5hbWUiOiJob29rIiwic2lkZSI6ImJvdHRvbSJ9fX1dLFsxLDIsIiIsMSx7InN0eWxlIjp7InRhaWwiOnsibmFtZSI6Imhvb2siLCJzaWRlIjoidG9wIn19fV1d
\[\begin{tikzcd}[cramped, column sep=small, row sep=small]
	& {S_{P\cap T}} & \\
	PQ && R
	\arrow[hook', from=1-2, to=2-1]
	\arrow[hook, from=1-2, to=2-3]
\end{tikzcd}
+
% https://q.uiver.app/#q=WzAsMyxbMCwxLCJQUSJdLFsxLDAsIlNfe1BcXGNhcCBUfSJdLFsyLDEsIlIiXSxbMSwwLCIiLDAseyJzdHlsZSI6eyJ0YWlsIjp7Im5hbWUiOiJob29rIiwic2lkZSI6ImJvdHRvbSJ9fX1dLFsxLDIsIiIsMSx7InN0eWxlIjp7InRhaWwiOnsibmFtZSI6Imhvb2siLCJzaWRlIjoidG9wIn19fV1d
\begin{tikzcd}[cramped, column sep=small, row sep=small]
	& {S_{Q\cap T}} & \\
	PQ && R
	\arrow[hook', from=1-2, to=2-1]
	\arrow[hook, from=1-2, to=2-3]
\end{tikzcd}\]
but not to $PQ \xhookleftarrow{} T \xhookrightarrow{} R$ in $\Hom(PQ,R)$.
\item We enforce $F$ and $G$ to be inverses of each other, i.e., $GF \cong Id$, by imposing the following relation on $\Hom$ groups, where $S_{P\cap T}$ and $S_{Q\cap T}$ satisfy the note below \eqref{eq: split},
\[
% https://q.uiver.app/#q=WzAsMyxbMCwxLCJQUSJdLFsxLDAsIlNfe1BcXGNhcCBUfSJdLFsyLDEsIlIiXSxbMSwwLCIiLDAseyJzdHlsZSI6eyJ0YWlsIjp7Im5hbWUiOiJob29rIiwic2lkZSI6ImJvdHRvbSJ9fX1dLFsxLDIsIiIsMSx7InN0eWxlIjp7InRhaWwiOnsibmFtZSI6Imhvb2siLCJzaWRlIjoidG9wIn19fV1d
\begin{tikzcd}[cramped, column sep=small, row sep=small]
	& {T} & \\
	PQ && R
	\arrow[hook', from=1-2, to=2-1]
	\arrow[hook, from=1-2, to=2-3]
\end{tikzcd}
=
\begin{tikzcd}[cramped, column sep=small, row sep=small]
	& {S_{P\cap T}} & \\
	PQ && R
	\arrow[hook', from=1-2, to=2-1]
	\arrow[hook, from=1-2, to=2-3]
\end{tikzcd}
+
\begin{tikzcd}[cramped, column sep=small, row sep=small]
	& {S_{Q\cap T}} & \\
	PQ && R
	\arrow[hook', from=1-2, to=2-1]
	\arrow[hook, from=1-2, to=2-3]
\end{tikzcd}
\]
\end{enumerate}
We denote the category with same objects as $\mM(\CC)$, but with these modified $\Hom$ groups, by $\bbZ\mM(\CC)$ and call it the \textit{category of scissors correspondences}. Note that $\bbZ\mM(\CC)$ is an $Ab$-enriched category.
\end{construction}

\begin{remark}\label{rem: Hom of actual object}
  Note that for $P, Q \in \obj(\CC) \subset \obj(\mM(\CC))$, $\Hom_{\bbZ\mM(\CC)}(P, Q)$ is the reduced free abelian group $\Hom(P, Q)$ from \cref{def: category of scissors correspondences}(\ref{item: 1}) that you get by freely enriching and then reducing since no extra relation is added in this case.
\end{remark}
\begin{remark}\label{rem: Hom of concatenations}
    Remember that general objects $P$ and $Q$ in $\mM(\CC)$ are concatenations of objects $P_1, P_2, \cdots, P_n$ and $Q_1, Q_2, \cdots, Q_m$ of $\CC$. Therefore, from \cref{def: category of scissors correspondences}, we have $$\Hom_{\bbZ\mM(\CC)}(P,Q) \cong \oplus_{1\leq i\leq n,1\leq j\leq m}\Hom_{\bbZ\mM(\CC)}(P_i,Q_j).$$
\end{remark}
\begin{remark}
The category of scissors correspondences inherits symmetric monoidal structure from $\mM(\CC)$ and we use the same notation, $\odot$, for this symmetric monoidal product. The monoidal product concatenates objects and on morphisms it does the following 
$$\sum_{i=1}^{k}n_{i}(P \xhookleftarrow{} T_i \xhookrightarrow{} Q) \odot \sum_{j=1}^{l}m_{j}(O \xhookleftarrow{} U_j \xhookrightarrow{} R) = \sum_{i=1}^{k}\sum_{j=1}^{l}n_{i}m_{j}((P \xhookleftarrow{} T_i \xhookrightarrow{} Q) \odot (O \xhookleftarrow{} U_j \xhookrightarrow{} R)).$$
Moreover, the functor
\begin{equation}\label{eq:MC to ZMC}
 \mM(\CC) \to \bbZ\mM(\CC)   
\end{equation}
that is identity on objects and sends a morphism in $\mM(\CC)$ to its equivalence class in $\Hom$ group of $\bbZ\mM(\CC)$ is symmetric monoidal.
\end{remark}
\begin{example}\label{ex: ZM for FinSets}
For $\CC:= \textbf{FinSet}$, the category of scissors correspondences $\bbZ\mM(\textbf{FinSet})$ has objects $(I,\{X_i\}_{i\in I})$ where $X_i \in \obj\textbf{FinSet}$ for each $i$. Given two objects $(I,\{X_i\}_{i\in I})$ and $(J,\{Y_j\}_{j\in J})$, 
\begin{align*}
\Hom_{\bbZ\mM(\textbf{FinSet})}((I,\{X_i\}_{i\in I}), (J,\{Y_j\}_{j\in J})) & \overset{\cref{rem: Hom of concatenations}}{\cong} \oplus_{i,j} \Hom_{\bbZ\mM(\textbf{FinSet})}(X_i, Y_j)\\
& \overset{\cref{rem: Hom of actual object}}{\cong} \oplus_{i,j}\Hom(X_i, Y_j). 
\end{align*}
Note that $\Hom(X_i, Y_j)$ is the free abelian group generated by equivalence classes of spans of injections $X_i \xhookleftarrow{} Z \xhookrightarrow{} Y_j$ where $Z$ is a non-empty finite set of cardinality less than $min\{|X_i|,|Y_j|\}$. 
In particular, if we compute $\Hom_{\bbZ\mM(\textbf{FinSet})}(\{1\},\{1\}) \cong \Hom(\{1\},\{1\})$, this will be generated as a reduced free abelian group by zigzags $\{1\}\xhookleftarrow[]{}X\xhookrightarrow[]{}\{1\}$ where $X$ is non empty and injects into the set $\{1\}$. The only possibility of such a zigzag is $$\{1\}\xleftarrow[]{id}\{1\}\xrightarrow[]{id}\{1\}.$$ Therefore, $\Hom_{\bbZ\mM(\textbf{FinSet})}(\{1\},\{1\}) \cong \bbZ$ as free abelian groups.
\end{example}
\begin{example} For the category $\PP_G^X$, the category of scissors correspondences $\bbZ\mM(\PP_G^X)$, has objects $(I,\{P_i\}_{i\in I})$ where $P_i \in \obj\PP_G^X$ for each $i$. Given two objects $(I,\{P_i\}_{i\in I})$ and $(J,\{Q_j\}_{j\in J})$,
        \begin{align*}
			\Hom_{\bbZ\mM(\PP_G^X)}((I,\{P_i\}_{i\in I}),(J,\{Q_j\}_{j\in J})) & \overset{\cref{rem: Hom of concatenations}}{\cong} \oplus_{i,j}\Hom_{\bbZ\mM(\PP_G^X)}(P_i, Q_j)\\
& \overset{\cref{rem: Hom of actual object}}{\cong} \oplus_{i,j}\Hom(P_i, Q_j)
		\end{align*}
        For the object $P:= [0,1]$, $\Hom([0,1],[0,1])$ is generated by all morphisms $[0,1] \xhookleftarrow{} R\xhookrightarrow{} [0,1]$ which is a huge group.
\end{example}
\begin{example}\label{ex: ZM for Restricted Polytopes}
Consider the category of restricted polytopes $\PP^{\mathcal{L}}_{G_{\mathcal{L}}}$ for various $\mathcal{L}$ as in \cref{ex:Restricted_polytope}(1),(2),(3). Here we study the category of scissors correspondences for each of these examples:
    \begin{enumerate}
        \item For $\CC:= \PP^{[0,1]}_1$, the category of scissors correspondences $\bbZ\mM(\PP^{[0,1]}_1)$ has objects $(I,\{P_i\}_{i\in I})$ where $P_i \in \obj\PP^{[0,1]}_1$ for each $i$. Given two objects $(I,\{P_i\}_{i\in I})$ and $(J,\{Q_j\}_{j\in J})$,
        \begin{align*}
			\Hom_{\bbZ\mM(\PP^{[0,1]}_1)}((I,\{P_i\}_{i\in I}),(J,\{Q_j\}_{j\in J})) & \overset{\cref{rem: Hom of concatenations}}{\cong} \oplus_{i,j}\Hom_{\bbZ\mM(\PP^{[0,1]}_1)}(P_i, Q_j)\\
& \overset{\cref{rem: Hom of actual object}}{\cong} \oplus_{i,j}\Hom(P_i, Q_j).
		\end{align*}
		$\Hom(P_i,Q_j)$ is the free abelian group generated by equivalence class of spans of inclusions $P_i \xhookleftarrow{} R \xhookrightarrow{} Q_j$ and therefore such an $R\in \obj\PP^{[0,1]}_1$ is a non-empty subset of $P_i \cap Q_j$. Further note that two such morphisms given by $R$ and $R'$ are equivalent if and only if $R$ and $R'$ are equal as sets and the morphisms are same set inclusion maps for both $R$ and $R'$.

		In particular, consider the polytope $P = [0,1]$. Then $$\Hom_{\bbZ\mM(\PP^{[0,1]}_1)}([0,1],[0,1]) \cong \Hom([0,1],[0,1])$$ is the free abelian group generated by finite union of closed intervals $R \subset [0,1]$. Hence if $M$ is the set of all finite disjoint union of closed intervals of $[0,1]$ then $\Hom_{\bbZ\mM(\PP_{[0,1]})}([0,1],[0,1]) \cong \bbZ[M]$. 

        For $\CC:= \PP^{[0,1]}_{C_2}$, we have the same objects but now $\Hom(P_i,Q_j)$ is the free abelian group generated by equivalence class of spans $$P_i \xhookleftarrow{g} R \xhookrightarrow{h} Q_j$$
        where $g,h \in C_2 =\{e, u\}$. Up to equivalence, this can be written as $$P_i \xhookleftarrow{e} gR \xhookrightarrow{hg^{-1}} Q_j$$ via
        % https://q.uiver.app/#q=WzAsNSxbMCwxLCJQX2kiXSxbMiwxLCJRX2oiXSxbMSwwLCJSIl0sWzEsMSwiUiJdLFsxLDIsImdSIl0sWzIsMCwiZyIsMix7InN0eWxlIjp7InRhaWwiOnsibmFtZSI6Imhvb2siLCJzaWRlIjoiYm90dG9tIn19fV0sWzIsMSwiaCIsMCx7InN0eWxlIjp7InRhaWwiOnsibmFtZSI6Imhvb2siLCJzaWRlIjoidG9wIn19fV0sWzQsMCwiZSIsMCx7InN0eWxlIjp7InRhaWwiOnsibmFtZSI6Imhvb2siLCJzaWRlIjoidG9wIn19fV0sWzQsMSwiaGdeey0xfSIsMix7InN0eWxlIjp7InRhaWwiOnsibmFtZSI6Imhvb2siLCJzaWRlIjoiYm90dG9tIn19fV0sWzMsMiwiZSIsMl0sWzMsNCwiZyIsMl1d
        \[\begin{tikzcd}[cramped]
        	& R & \\
        	{P_i} & R & {Q_j} \\
        	& gR
        	\arrow["g"', hook', from=1-2, to=2-1]
        	\arrow["h", hook, from=1-2, to=2-3]
        	\arrow["e"', from=2-2, to=1-2]
        	\arrow["g"', from=2-2, to=3-2]
        	\arrow["e", hook, from=3-2, to=2-1]
        	\arrow["{hg^{-1}}"', hook', from=3-2, to=2-3]
        \end{tikzcd}\]
        Hence an element of $\Hom(P_i,Q_j)$ is uniquely determined by such a pair $(D, \phi)$ where $D=gR$ and $\phi = hg^{-1}$. In particular, for $P = [0,1]$, $\Hom([0,1],[0,1])$ is the free abelian group generated by set of all such morphisms. The set of such morphisms for $[0,1]$ is $M \times \{e\} \cup M\times \{u\}$. Thus we have, $$\Hom_{\bbZ\mM(\PP^{[0,1]}_{C_2})}([0,1],[0,1]) \cong \Hom([0,1],[0,1]) \cong \bbZ[M\times C_2]$$ where $C_2$ here is just considered as the set $\{e,u\}$.
         
         \item Similarly for $\CC:= \PP^n_{G_{\mathcal{L}_n}}$, given two objects $(I,\{P_i\}_{i\in I})$ and $(J,\{Q_j\}_{j\in J})$ in $\bbZ\mM(\PP^n_{G_{\mathcal{L}_n}})$
        $$\Hom_{\bbZ\mM(\PP^n_{G_{\mathcal{L}_n}})}((I,\{P_i\}_{i\in I}),(J,\{Q_j\}_{j\in J})) \cong \oplus_{i,j}\Hom(P_i, Q_j).$$ Note that $G_{\mathcal{L}_n}$ is any subgroup of restricted isometry group. In particular, we will be interested in the cases where we work with just the translations subgroup $\bbZ$ or with the full restricted isometry group $\bbZ \rtimes C_2$ as discussed in \cref{ex:Restricted_polytope}(2).

		Consider the object $P:= [0,1/n] \in \obj(\bbZ\mM(\PP^n_{G_{\mathcal{L}_n}}))$. We demonstrate computation for $\Hom_{\bbZ\mM(\PP^n_{\bbZ})}([0,1/n],[0,1/n])$ and $\Hom_{\bbZ\mM(\PP^n_{\bbZ\rtimes C_2})}([0,1/n],[0,1/n])$.

		$\Hom_{\bbZ\mM(\PP^n_{\bbZ})}([0,1/n],[0,1/n])$ is generated as a free abelian group by morphisms 
		\begin{equation}\label{eqn: generator}
			[0,1/n]\xhookleftarrow{}R\xhookrightarrow{}[0,1/n]
		\end{equation}
		where $R$ is a non empty object in $\bbZ\mM(\PP^n_{\bbZ})$. But this means $R$ needs to have size at most $1/n$ and given the restriction $\mathcal{L}_n$, we have that $R$ is a single $1/n$-length interval that under this morphism, translates to $[0,1/n]$. Hence the morphisms look like $[0,1/n]\xleftarrow{T}R\xrightarrow{T}[0,1/n]$ where $T$ is the translation of $R$ to $[0,1/n]$. But any such morphism, up to equivalence, can be written as $[0,1/n]\xleftarrow{id}[0,1/n]\xrightarrow{id}[0,1/n]$ via
        % https://q.uiver.app/#q=WzAsNSxbMSwwLCJSIl0sWzAsMSwiWzAsMS9uXSJdLFsyLDEsIlswLDEvbl0iXSxbMSwyLCJbMCwxL25dIl0sWzEsMSwiUiJdLFswLDEsIlQiLDJdLFswLDIsIlQiXSxbMywxLCJpZCJdLFszLDIsImlkIiwyXSxbNCwwLCJpZCJdLFs0LDMsIlQiLDJdXQ==
\[\begin{tikzcd}[cramped]
	& R & \\
	{[0,1/n]} & R & {[0,1/n]} \\
	& {[0,1/n]}
	\arrow["T"', from=1-2, to=2-1]
	\arrow["T", from=1-2, to=2-3]
	\arrow["id", from=2-2, to=1-2]
	\arrow["T"', from=2-2, to=3-2]
	\arrow["id", from=3-2, to=2-1]
	\arrow["id"', from=3-2, to=2-3]
\end{tikzcd}\]
So $\Hom_{\bbZ\mM(\PP^n_{\bbZ})}([0,1/n],[0,1/n])$ is generated by $[0,1/n]\xleftarrow{id}[0,1/n]\xrightarrow{id}[0,1/n]$ and therefore as free abelian group
\begin{equation}\label{eqn: Hom for dim 1 lattice translations}
	\Hom_{\bbZ\mM(\PP^n_{\bbZ})}([0,1/n],[0,1/n]) \cong \bbZ.
\end{equation}		

$\Hom_{\bbZ\mM(\PP^n_{\bbZ\rtimes C_2})}([0,1/n],[0,1/n])$ is also generated as a free abelian group by morphisms \eqref{eqn: generator} and we get the same restrictions on length of $R$ except now $R$ can map into $[0,1/n]$ using some composite of translations and reflections. Consider a general morphism of this type
$$\left[0,\frac{1}{n}\right]\xleftarrow{\theta}\left[\frac{p}{n},\frac{p+1}{n}\right]\xrightarrow{\theta'}\left[0,\frac{1}{n}\right]$$
then we have an equivalence
\[\begin{tikzcd}[cramped]
	& \left[\frac{p}{n},\frac{p+1}{n}\right] & \\
	{\left[0,\frac{1}{n}\right]} & \left[\frac{p}{n},\frac{p+1}{n}\right] & {\left[0,\frac{1}{n}\right]} \\
	& {\left[0,\frac{1}{n}\right]}
	\arrow["\theta"', from=1-2, to=2-1]
	\arrow["\theta'", from=1-2, to=2-3]
	\arrow["id", from=2-2, to=1-2]
	\arrow["T_{-p}"', from=2-2, to=3-2]
	\arrow["\theta\circ T_{p}", from=3-2, to=2-1]
	\arrow["\theta'\circ T_{p}"', from=3-2, to=2-3]
\end{tikzcd}\]
where $T_p$ denotes translation by $p/n$.
So it suffices to only consider non equivalent morphisms $[0,1/n]\xleftarrow{\theta} [0,1/n]\xrightarrow{\theta'} [0,1/n]$. But further we have the equivalence
\[\begin{tikzcd}[cramped]
	& {[0,1/n]} & \\
	{[0,1/n]} & {[0,1/n]} & {[0,1/n]} \\
	& {[0,1/n]}
	\arrow["\theta"', from=1-2, to=2-1]
	\arrow["\theta'", from=1-2, to=2-3]
	\arrow["id", from=2-2, to=1-2]
	\arrow["\theta'"', from=2-2, to=3-2]
	\arrow["\theta\circ\theta'^{-1}", from=3-2, to=2-1]
	\arrow["id"', from=3-2, to=2-3]
\end{tikzcd}\]
showing we can restrict to morphisms of the type $[0,1/n]\xleftarrow{\theta} [0,1/n]\xrightarrow{id} [0,1/n]$. But any morphism from $[0,1/n]$ to itself up to isomorphism of isometry group elements is either identity map or a reflection followed by translation by $1/n$, $T_{1}\circ r$. Therefore as free abelian group 
\begin{equation}\label{eqn: Hom for dim 1 lattice translations+reflections}
	\Hom_{\bbZ\mM(\PP^n_{\bbZ\rtimes C_2})}([0,1/n],[0,1/n]) \cong \bbZ^2
\end{equation} 
generated by $[0,1/n]\xleftarrow{id} [0,1/n]\xrightarrow{id} [0,1/n]$ and $[0,1/n]\xleftarrow{T_{1}\circ r} [0,1/n]\xrightarrow{id} [0,1/n]$.
        \item Similar arguments hold for more general lattice polytope categories as defined in \cref{ex:Restricted_polytope}(3). For the object $I_m\times I_m$, where $I_m$ is $[0,1/m]$ in the category $\PP^{m,m}_{G_{\mathcal{L}_{m,m}}}$, we can compute $\Hom$ groups when $G_{\mathcal{L}_{m,m}}$ is $\bbZ$ (translations) or the full restricted isometry group $\bbZ \rtimes D_4$.
        
		$\Hom_{\PP^{m,m}_{\bbZ}}(I_m\times I_m, I_m\times I_m)$ is generated by 
\begin{equation*}
	I_m\times I_m \xleftarrow{id} I_m\times I_m \xrightarrow{id} I_m\times I_m
\end{equation*}
		and is, therefore, isomorphic to $\bbZ$.

		Whereas $\Hom_{\PP^{m,m}_{\bbZ \rtimes D_4}}(I_m\times I_m, I_m\times I_m)$ is generated by elements
		\begin{equation*}
	I_m\times I_m \xleftarrow{\theta} I_m\times I_m \xrightarrow{id} I_m\times I_m
\end{equation*} 
where $\theta$ ranges over non isomorphic isometries of the square to itself, i.e. $\theta$ ranges over the elements of $D_4$. Thus, as free abelian groups
\begin{equation}
	\Hom_{\PP^{m,m}_{\bbZ \rtimes D_4}}(I_m\times I_m, I_m\times I_m) \cong \bbZ^8.
\end{equation}
Note that the generators of this $D_4$ are $\rho$ (denoting rotation by $\pi/2$) and $T_{1,0}\circ r$ (denoting reflection along $y$-axis followed by translation by $(1/m,0)$).
\end{enumerate}
\end{example}

%%%%%%%%%%%%%%%%%%%%%%%%%%%%NEW SUBSECTION%%%%%%%%%%%%%%%%%%%%%%%%%%%%%%
\section{A Morita Equivalence}\label{subsec: Morita Equivalence}

In this section, we define a notion of Morita object in $\bbZ\mM(\CC)$ along with a single object subcategory containing this Morita object which we call $\bbZ P(\CC)$ (\cref{def: ZP}). We show that these two categories are Morita equivalent (\cref{prop: Morita equiv between ZM and ZP}) and therefore their Hochschild complexes are equivalent (\cref{thm: Morita_Invariance_Hochschild_homology}, \cref{corr: HH of ZM and ZP}). Further the Hochschild complex of $\bbZ P(\CC)$ can be rewritten as Hochschild homology of a ring (\cref{lemma: Hochschild homology of weak assembler}), which we use to do some Hochschild homology computations in \cref{ex: HH of P_n}, \cref{ex: HH of P_mm} and \cref{ex: HH of P_01}. For our discussions on Morita equivalence of $Ab$-categories and Morita invariance of Hochschild homology we use \cite{CaryMorita}, \cite{BM-Morita} as our main reference.

\begin{definition}\label{def: ZP}
Let $P$ be any object in $\bbZ\mM(\CC)$ such that it is also an object in $\CC$ and satisfies the following condition:
for every $Q \in \obj \bbZ\mM(\CC)$, there exists some (not necessarily unique) $n \geq 1$, and an inclusion (in the sense of \cref{def: unenriched scissors correspondences}) $Q \to P^n$. We call an object that satisfies such a condition, a \textit{Morita object}.
For such a $P$, consider the full subcategory of $\bbZ\mM(\CC)$ containing the single object $P$. We denote this category by $\bbZ P(\CC)$.
\end{definition}
Morita objects are not guaranteed to exist (see \cref{ex: Morita object in Polytopes} below). Before we give some examples of Morita objects, we define a type of weak assembler and then show that this type always has Morita objects.
\begin{definition}\label{def: EA weak assembler}\cite[Definition 3.2]{KupersLemannMalkiewichMillerSroka2024ScissorsAuto}
    Let $\operatorname{R}$ be an ordered abelian group and $R_{\geq 0}$ be the subset of elements $r$ that satisfy $r \geq 0$. A \textit{volume function} is a function $$\operatorname{vol}:\obj\CC\to \operatorname{R}_{\geq 0}$$ such that it is preserved by covering families i.e. given $\{P_i \to P\}_{i\in I}$ in $\wW(C)$, $\operatorname{vol}(P)=\sum_{i\in I}\operatorname{vol}(P_i)$. This can be extended to all objects of $\wW(\CC)$ by setting $\operatorname{vol}({P_i}):= \sum_i \operatorname{vol}(P_i)$. A \emph{weak EA-assembler} is a weak assembler with a volume function such that
    \noindent\begin{itemize}
        \item \textbf{(E)}: If $\operatorname{vol}(P)< \operatorname{vol}(Q)$ for $P,Q \in \wW(\CC)$ then
        there exists an inclusion in $\mM(\CC)$
        $$P \xleftarrow[]{\in \wW(\CC)} P' \xhookrightarrow{} Q. $$
        \item \textbf{(A)}: For any $P,Q \in \obj(\wW(\CC))$, there exists a cover ${Q_i}$ of $Q$ such that $2 \operatorname{vol}(Q_i) \leq \operatorname{vol}(P) $ for all $i$.
    \end{itemize} 
\end{definition}
\begin{remark}\label{rem: volumes}
    Note that we have $0\leq \vol(Q_i)$ and therefore $\vol(Q_i)\leq 2 \vol(Q_i)$. Therefore, from \textbf{(A)} we have, $\vol(Q_i)\leq\vol(P)$. In particular if $0< \vol(Q_i)$, then $\vol(Q_i)<\vol(P)$.
\end{remark}
\begin{lemma}\label{lemma: weak EA assembler has Morita}
   Every object in a weak EA-assembler with non-zero volume is a Morita object.  
\end{lemma}
\begin{proof}
 Let $P$ be an object with $\operatorname{vol}(P)>0$ and let $Q$ be any other object. Then by (A), there is a cover $\{Q_i\}_{i\in I}$ of $Q$ such that $2 \operatorname{vol}(Q_i) \leq \operatorname{vol}(P)$. If $\vol(Q_i) = 0$, then $\vol(Q_i)<\vol(P)$ and from \textbf{(E)}, we have an inclusion $Q_i \to P$ in $\mM(\CC)$. Otherwise, $\vol(Q_i)>0$, in which case from \cref{rem: volumes} $\vol(Q_i)<\vol(P)$, and therefore from \textbf{(E)}, we have again have an inclusion $Q_i \to P$.

Concatenating all these inclusions gives us 
$$Q \xleftarrow{} \odot_{i\in I}Q_i \xhookrightarrow{} P^n$$
where $n = \vert I\vert$ proving our claim. 
\end{proof}
\begin{remark}
    While EA condition always implies the existence of Morita objects, a weak assembler can have a Morita object even if it is not EA. Check \cref{ex: Morita object in FinSet}, \cref{ex: Morita object in Restricted Polytopes} below.
\end{remark}
\begin{example}\label{ex: Morita object in FinSet}
\textbf{FinSet} is NOT a weak EA-assembler with respect to cardinality function. Regardless, the category has Morita objects. In particular $\{1\}$ is a Morita object with one choice of inclusion for a given set $X$ being $X \xhookleftarrow{} X' \xhookrightarrow{} \{1\}^{n}$ where $n = |X|$ and $X' = \odot_{i=1}^n\{x_i \vert x_i\in X\}$.
\end{example}
\begin{example}\label{ex: Morita object in Polytopes}
From \cite[Proposition 5.4]{KupersLemannMalkiewichMillerSroka2024ScissorsAuto} $\PP_G^{\mathcal{E}^n}$ is a weak EA-assembler with the usual volume function whenever $G$ contains a group of translations that is dense in $\bbR^n$. In particular, $\PP_G^{\mathcal{E}^n}$ is weak EA-assembler for $G = \bbR^n$ (the translation isometries) and $G = \bbR^n \rtimes O_n$ (the full isometry group). For example, $[0,1]$ is a Morita object in $\PP_G^{\mathcal{E}^1}$ for these choices of $G$.

From \cite[Proposition 5.12]{KupersLemannMalkiewichMillerSroka2024ScissorsAuto} $\PP_G^{\mathcal{H}^n}$ and $\PP_G^{\mathcal{S}^n}$ are EA-assemblers with hyperbolic and spherical volume respectively, provided that $G$ contains a dense subgroup of orientation preserving isometries so they also have Morita objects. 

On the other hand, if $G$ is chosen to be trivial, the category of no-move polytopes $\PP^{X}_1$ for any $X= \mathcal{E}^n,\mathcal{S}^n$ or $ \mathcal{H}^n$, is not a weak EA-assembler with respect to the respective volume function. The morphisms in this category are just set inclusions. It is easy to see that this category has no Morita object: for any choice of object $P$, there exists an object $Q$ with volume less that $P$ and greater than zero such that $Q \not\subset P$.
\end{example}
\begin{example}\label{ex: Morita object in Restricted Polytopes}
Restricted polytopes are not necessarily always weak EA-assemblers with respect to the volume function. In particular if the volume function takes discrete values and there are multiple objects of minimum non-zero volume, then a pair of such objects will never satisfy \cref{def: EA weak assembler}\textbf{(A)}. \cite[Section 6]{KupersLemannMalkiewichMillerSroka2024ScissorsAuto} has examples of some restricted polytope categories that are EA. However, various categories of restricted polytopes that are not EA still have Morita objects:
\begin{enumerate}
    \item The category $\PP^{[0,1]}_1$ is not EA for same reason as $\PP^{\mathcal{E}^n}_1$. But unlike $\PP^{\mathcal{E}^n}_1$, in $\PP^{[0,1]}_1$, $[0,1]$ is a Morita object because every other object of $\PP^{[0,1]}_1$ includes into $[0,1]$. In fact this is the only Morita object in $\PP^{[0,1]}_1$ since $[0,1]$ cannot include into any other object of $\PP^{[0,1]}_1$. Using similar reasoning, $[0,1]$ is a Morita object of $\PP^{[0,1]}_{C_2}$.
    \item The category $\PP^{n}_{G_{\mathcal{L}_n}}$ is not EA. But, in both the cases, for $G = \bbZ$ and for $G= \bbZ\rtimes C_2$, $[0,1/n]$ is a Morita object. Anything that has same length or less includes into $[0,1/n]$ by translation for example. Anything that is bigger can be covered by intervals of length $1/n$ and each of them can be included into $[0,1/n]$.
    \item The category $\PP^{m,m}_{G_{\mathcal{L}_{m,m}}}$, by similar arguments, is not EA but for both $G = \bbZ^2$ and $G= \bbZ\rtimes D_4$, $I_m\times I_m$ is a Morita object.
\end{enumerate}
\end{example}
\begin{remark}\label{rem: inclusion of morita object is DK}
    The inclusion functor $i: \bbZ P(\CC) \to \bbZ\mM(\CC)$ is a Dwyer-Kan embedding (full and faithful additive functor) since $\bbZ P(\CC)$ is a full subcategory of $\bbZ\mM(\CC)$.
\end{remark}
Next, we define Morita equivalence and prove Morita equivalence of $\bbZ P(\CC)$ and $\bbZ\mM(\CC)$.
\begin{definition}\label{def: Morita Equivalence}
Two $Ab$-enriched categories $\CC$ and $\DD$ are \textit{Morita equivalent} if there exists a functor $F: \CC \to \DD$ which is a Dwyer-Kan embedding and is surjective up to thick closure. The latter means that every object $d \in \obj \DD$ belongs to an enlargement of the set $S:= \{ Fc: c\in \obj\CC\}$ by following these sequence of steps performed iteratively finitely many times:
\begin{itemize}
    \item Take a retract of any object in $S$, and add it to $S$.
    \item Take a pair of composable morphisms $d \to d' \to d''$ in the underlying category $\DD_0$, such that the induced map of modules 
    \begin{equation*}
    \DD(d'',-) \to \DD(d',-) \to \DD(d,-), \ \ \ \DD(-,d) \to \DD(-,d') \to \DD(-,d'')    
    \end{equation*}
 are split short exact sequences, i.e. they are split short exact sequences of abelian groups for every input in $-$. Then, if two of $d, d', d''$ are in $S$, add the third one to $S$.       
\end{itemize}
\end{definition}

\begin{proposition}\label{prop: Morita equiv between ZM and ZP}
    Let $\CC$ be a weak assembler and $P \in \obj \bbZ\mM(\CC)$ be a Morita object as in \cref{def: ZP}, then $i: \bbZ P(\CC) \to \bbZ \mM(\CC)$ gives a Morita equivalence between $\bbZ P(\CC)$ and $\bbZ \mM(\CC)$.
\end{proposition}

\begin{proof}
    Since $i$ is a Dwyer-Kan embedding from \cref{rem: inclusion of morita object is DK}, we just need to show $i$ is surjective up to thick closure. 
    
    Since $P$ is a Morita object, given $Q \in \obj \bbZ\mM(\CC)$, there exists an $n$ and an inclusion from $Q$ into $P^n$ $$Q \xleftarrow{} Q' \xhookrightarrow{} P^n.$$ Composing this with the retraction map $P^n \xhookleftarrow{} Q' \xrightarrow{} Q$ gives the identity map $Q \xrightarrow{} Q$ from \cref{rem: projection is retract}. Thus, every $Q$ is a retract of some $P^n$ and it is enough to show that for each $n \geq 1$, $P^n$ is in the thick closure of $P$ in $\bbZ\mM(\CC)$.

    We show this by induction. We know $P$ is in the thick closure of itself. Now assume $P^n$ is in the thick closure of $P$, we want to show so is $P^{n+1}$. Consider the following composite of morphisms in $\bbZ\mM(\CC)$, $P^n \to P^{n+1} \to P$ where the first map is inclusion into the first $n$ copies of $P$ in $P^{n+1}$ and the second map is projection onto the last copy of $P$ in $P^{n+1}$. Now consider the induced map of modules
    $$\bbZ\mM(\CC)(P,-) \to \bbZ\mM(\CC)(P^{n+1},-) \to \bbZ\mM(\CC)(P^n,-).$$
    From \cref{def: category of scissors correspondences}, we know that this is isomorphic to
    $$\bbZ\mM(\CC)(P,-) \to \oplus_{n+1}\bbZ\mM(\CC)(P,-) \to \oplus_{n}\bbZ\mM(\CC)(P,-)$$
    where the first map sends $\bbZ\mM(\CC)(P,-)$ isomorphically to the last summand in $\oplus_{n+1}\bbZ\mM(\CC)(P,-)$ and the second map projects $\oplus_{n+1}\bbZ\mM(\CC)(P,-)$ to its first $n$-coordinates. Similar statements hold for the map of modules
    $$\bbZ\mM(\CC)(-,P^n) \to \bbZ\mM(\CC)(-,P^{n+1}) \to \bbZ\mM(\CC)(-,P).$$ Thus both of these sequences of modules are split short exact sequences and since $P, P^n$ are in thick closure of $P$ so is $P^{n+1}$.
\end{proof}

Now, we want to show Morita invariance of the Hochschild complex, i.e. we want to show that a Morita equivalence $F: \CC \to \DD$ of two $Ab$-enriched categories induces an equivalence of their Hochschild complexes. We define these terminologies first and then show the result. This is a well-known fact and we are mostly taking this argument from \cite{CaryMorita}, \cite{BM-Morita}. However, we could not find the argument written out explicitly for the case of $Ab$-enriched categories using the notion of Morita equivalence of \cref{def: Morita Equivalence} so we are giving some concise arguments below. For more details, we will refer the readers to \cite{CaryMorita}, \cite{BM-Morita} which contains the full argument in the setting of spectrally enriched categories. 

First, for $\CC$, $\DD$ $Ab$-enriched categories, we define the notion of $\CC$-modules and $\CC,\DD$-bimodules in simplicial abelian groups $sAb$. 
\begin{definition}\label{def: C-modules}
    $M$ is a \textit{left $\CC$-module} in $sAb$ if there exists a set map from $\obj\CC$ to $sAb$ along with multiplication maps $$\CC(c',c) \otimes M(c)\to M(c')$$ for every $c,c' \in \obj\CC$ which are associative and unital. Similarly, one can define right $\CC$-modules in $sAb$. Two left $\CC$-modules $M$ and $N$ are isomorphic if for every object $c$, $M(c) \cong N(c)$ via maps that respect the module structure. They are simplicial homotopy equivalent if for every $c$, $M(c)$ and $N(c)$ are simplicial homotopy equivalent via maps that respect the module structure and weakly equivalent if for every $c$, $M(c)$ and $N(c)$ are weakly equivalent.
\end{definition}
\begin{definition}\label{def: C,D-bimodules}
    $M$ is a \textit{$\CC,\DD$-bimodule} in $sAb$ if there exists a set map from $\obj\CC \times \obj\DD$ to $sAb$ along with multiplication maps $$\CC(c',c) \otimes M(c,d) \to M(c',d) \ \text{and} \ M(c,d) \otimes \DD(d,d') \to M(c,d')$$ for every $c,c' \in \obj\CC$ and $d,d' \in \obj\DD$ which are associative and unital. Additionally, the two maps $$\CC(c',c) \otimes M(c,d) \otimes \DD(d,d') \to M(c',d')$$ must agree with each other. Two $\CC,\DD$-bimodules $M$ and $N$ are isomorphic if for every pair $(c,d)$, $M(c,d) \cong N(c,d)$ via maps that respect the bimodule structure. They are simplicial homotopy equivalent if for every $(c,d)$, $M(c,d)$ and $N(c,d)$ are simplicial homotopy equivalent via maps that respect the bimodule structure and weakly equivalent if for every $(c,d)$, $M(c,d)$ and $N(c,d)$ are weakly equivalent.
\end{definition}
\begin{remark}
    Note that $\CC$ can be thought of as a $\CC,\CC$-bimodule by associating to a pair $(c,d)$, the constant simplicial abelian group $\CC(c,d)$.
\end{remark}
\begin{definition}\label{def: cyclic bar complex or Hochschild complex}
    For an $Ab$-enriched category $\CC$ and a $\CC,\CC$-bimodule $M$, consider the cyclic simplicial abelian group which at cyclic level $n$ is given by $$\bigoplus_{c_0, c_1, \cdots, c_n \in \obj\CC}\CC(c_0,c_1)\otimes \CC(c_1, c_2)\otimes\cdots \otimes \CC(c_{n-1}, c_n)\otimes M(c_n,c_0)$$
    with the face maps and degeneracy maps (in the cyclic direction) given by composing a term in the tensor product with the next one (the second last and last face maps use the left and right module structure of the bimodule $M$) and inserting identity morphism respectively. We define the \textit{Hochschild complex of $\CC$ with coefficients in a $\CC,\CC$-bimodule $M$}, denoted as $HH(\CC;M)$, to be the cyclic abelian group given by the diagonal of this cyclic simplicial abelian group. In other words, the $n$th level of the cyclic complex $HH(\CC;M)$ is given by
    $$HH(\CC;M)_n = \bigoplus_{c_0, c_1, \cdots, c_n \in \obj\CC}\CC(c_0,c_1)\otimes \CC(c_1, c_2)\otimes\cdots \otimes \CC(c_{n-1}, c_n)\otimes M(c_n,c_0)_n$$
    where by $M(c_n,c_0)_n$ we mean the $n$th level of the simplicial abelian group $M(c_n,c_0)$.
    
    When $M$ is the $\CC,\CC$-bimodule $\CC$ itself, we sometimes omit the coefficient and call the complex the \textit{Hochschild complex of $\CC$}, $HH(\CC)$, as this agrees with the classical definition of Hochschild complex of $\CC$ as in \cref{Notations}(\ref{Not:cyclic_bar_complex}). 
    
    Using Dold-Kan, we equivalently get a chain complex which we also call the Hochschild complex of $\CC$ with coefficients in $M$. The homotopy of this simplicial abelian group or, equivalently, the homology of the associated chain complex is defined as the \textit{Hochschild homology of $\CC$ with coefficients in $M$} and denoted by $HH_{\ast}(\CC;M)$. We denote the $n$th Hochschild homology group as $HH_n(\CC;M)$. Note the distinction in notation from the $n$th level of the Hochschild complex, $HH(\CC;M)_n$, above.
\end{definition}
\begin{definition}
  Given a left $\CC$-module $N$ and a right $\CC$-module $M$, the \textit{double sided bar complex} $B(M;\CC;N)$ is a simplicial abelian group with level $n$ given by $$\oplus_{\{c_0,c_1,\cdots ,c_n\}}M(c_0)_n\otimes \CC(c_0,c_1)\otimes \CC(c_1,c_2)\otimes \cdots \otimes \CC(c_{n-1},c_n)\otimes N(c_n)_n$$ and the $i^{th}$ face and degeneracy maps given by multiplying $i^{th}$ term with the next one or inserting unit at the $i^{th}$ spot. Here by $M(c_0)_n$ we mean the $n$th level of the simplicial abelian group $M(c_0)$ and similarly for $N(c_n)_n$. 
\end{definition}
\begin{remark}\label{rem: double-sided Bar complex as bimodule}
    If $M$ and $N$ in the definition above have extra module structures on the other side, for example if $M$ is a left $\DD$-module and $N$ is a right $\EE$-module, then we can define a new $\DD,\EE$-bimodule via $(d,e) \mapsto B(M(d,-);\CC;N(-,e))$.
\end{remark}
\begin{remark}\label{rem:pullback modules}
   Let $F: \CC \to \DD$ be a functor of $Ab$-categories. Then given a left $\DD$-module $M$, we can take pullback along $F$ to get a left $\CC$-module ${}_FM$, where ${}_FM(c):= M(Fc)$ and the multiplication maps $\CC(c',c)\otimes {}_FM(c)\to {}_FM(c')$ are given by $\DD(Fc',Fc)\otimes M(Fc)\to M(Fc')$. Similarly, we can pullback along right a right $\DD$-module, and pullback on either or both sides a $\DD,\DD$-bimodule.
\end{remark}
\begin{theorem}\label{thm: Morita_Invariance_Hochschild_homology}
Let $\CC$ and $\DD$ be two $Ab$-enriched categories such that there exists $F: \CC \to \DD$ giving a Morita equivalence between them. Then their Hochschild complexes are equivalent as simplicial abelian groups, i.e. $HH(\CC) \simeq HH(\DD)$.
\end{theorem}
\begin{proof}
    This is a standard fact and we omit a proof in full details here. We refer the reader to \cite[Section 3]{CaryMorita} for details.
    The idea is that \cite[Lemma 3.3 and Lemma 3.5]{CaryMorita} hold, for convenience we state them again below as \cref{lemma:Key_Morita_lemma} and \cref{lemma:Rotator_Isomorphism} for $Ab$-enriched categories. Then we can directly conclude the result using the following commuting diagram
    \[
\begin{tikzcd}
    HH(\CC;\CC) \ar[d,"\cong"] & HH(\CC;B(\CC;\CC;\CC)) \ar[l,"\cref{lemma:Bar_Lemma}","\simeq"'] \ar[r,"\cong","\cref{lemma:Rotator_Isomorphism}"'] \ar[d] & HH(\CC;B(\CC;\CC;\CC)) \ar[d] \ar[r,"\simeq","\cref{lemma:Bar_Lemma}"'] & HH(\CC;\CC) \ar[d]\\
    HH(\CC;{}_F\DD_F) & HH(\CC;B({}_F\DD;\DD;\DD_F)) \ar[l,"\cref{lemma:Bar_Lemma}","\simeq"'] \ar[r,"\cong","\cref{lemma:Rotator_Isomorphism}"'] & HH(\DD;B(\DD_F;\CC;{}_F\DD)) \ar[r,"\simeq","\cref{lemma:Key_Morita_lemma}"'] & HH(\DD;\DD)
\end{tikzcd}
\]
where all the vertical maps are induced by $F$. The first vertical map is an isomorphism due to $F$ being Dwyer-Kan embedding and the middle two vertical maps are induced by $F$ using \cref{rem:pullback modules}.
This gives us the required equivalence between $HH(\CC)$ and $HH(\DD)$.
\end{proof}
\begin{lemma}\label{lemma:Bar_Lemma}\cite[Lemma 1.4, Bar Lemma]{CaryMorita}
    For any $\CC,\DD$-bimodule $M$, there are canonical simplicial homotopy equivalences between $B(\CC(c,-);\CC(-,-);M(-,d))$ and $M(c,d)$. Similarly for any $\DD, \CC$-bimodule $N$, there are canonical simplicial homotopy equivalences between $B(N(d,-);\CC(-,-);\CC(-,c))$ and $N(d,c)$. One side maps in each of these homotopy equivalences respect the module maps, thus inducing weak equivalence of modules. Thus, we can say that the $\CC,\DD$-bimodules $B(\CC;\CC;M)$ and $M$ are equivalent and the $\DD, \CC$-bimodules $B(N;\CC;\CC)$ and $N$ are equivalent.
\end{lemma}
\begin{lemma}\label{lemma:Key_Morita_lemma}\cite[Lemma 3.3]{CaryMorita}
    If $F: \CC \to \DD$ is a Morita equivalence, then the following maps are equivalences of $\CC,\CC$-bimodules and $\DD,\DD$-bimodules respectively
    \begin{align*}
        \CC &\to {}_F\DD_F \\
        B(\DD_F;\CC;{}_F\DD) &\to B(\DD;\DD;\DD)\simeq \DD
    \end{align*}
    where the first and second maps are induced by $F$ and the last equivalence is from Bar lemma, \cref{lemma:Bar_Lemma}.
\end{lemma}
\begin{lemma}\label{lemma:Rotator_Isomorphism}\cite[Lemma 3.5, Dennis-Waldhausen-Morita argument]{CaryMorita}
    For each $\CC, \DD$-bimodule $M$ and $\DD, \CC$-bimodule $N$, there is a natural rotator isomorphism 
    $$HH(\CC;B(M;\DD;N))\cong HH(\DD;B(N;\CC;M)).$$
\end{lemma}
From \cref{prop: Morita equiv between ZM and ZP} and \cref{thm: Morita_Invariance_Hochschild_homology}, we get
\begin{corollary}\label{corr: HH of ZM and ZP}
   $HH(\bbZ\mM(\CC))\simeq HH(\bbZ P(\CC))$. In particular, $HH_{\ast}\bbZ\mM(\CC) \cong HH_{\ast}\bbZ P(\CC)$.
\end{corollary}
\begin{lemma}\label{lemma: Hochschild homology of weak assembler}
The Hochschild complex of the one object category $\bbZ P(\CC)$ is isomorphic to the Hochschild complex of the ring of endomorphisms of $P$ in $\bbZ\mM(\CC)$ (equivalently, the ring of endomorphisms of $P$ in $\bbZ P(\CC)$), $\Hom_{\bbZ\mM(\CC)}(P,P)$, where the ring operation is composition. In particular, we have 
$$HH_{\ast}(\bbZ P(\CC)) \cong HH_{\ast}(\Hom_{\bbZ\mM(\CC)}(P,P)).$$    
\end{lemma}
\begin{proof}
The Hochschild complex of $\bbZ P(\CC)$ by definition (\cref{Notations}(\ref{Not:cyclic_bar_complex})) is 
   % https://q.uiver.app/#q=WzAsNCxbMCwwLCJcXEhvbV97XFxiYlogUChcXENDKX0oUCxQKSJdLFsxLDAsIlxcSG9tX3tcXGJiWiBQKFxcQ0MpfShQLFApXFxvdGltZXMgXFxIb21fe1xcYmJaIFAoXFxDQyl9KFAsUCkiXSxbMiwwLCJcXEhvbV97XFxiYlogUChcXENDKX0oUCxQKV57XFxvdGltZXMgM30iXSxbMywwLCJcXGNkb3RzIl0sWzAsMSwiIiwwLHsib2Zmc2V0IjotMiwic3R5bGUiOnsidGFpbCI6eyJuYW1lIjoiYXJyb3doZWFkIn0sImhlYWQiOnsibmFtZSI6Im5vbmUifX19XSxbMSwwLCIiLDEseyJzdHlsZSI6eyJ0YWlsIjp7Im5hbWUiOiJhcnJvd2hlYWQifSwiaGVhZCI6eyJuYW1lIjoibm9uZSJ9fX1dLFswLDEsIiIsMCx7Im9mZnNldCI6Miwic3R5bGUiOnsidGFpbCI6eyJuYW1lIjoiYXJyb3doZWFkIn0sImhlYWQiOnsibmFtZSI6Im5vbmUifX19XSxbMSwyLCIiLDAseyJzdHlsZSI6eyJ0YWlsIjp7Im5hbWUiOiJhcnJvd2hlYWQifSwiaGVhZCI6eyJuYW1lIjoibm9uZSJ9fX1dLFsyLDEsIiIsMCx7Im9mZnNldCI6Miwic3R5bGUiOnsidGFpbCI6eyJuYW1lIjoiYXJyb3doZWFkIn0sImhlYWQiOnsibmFtZSI6Im5vbmUifX19XSxbMSwyLCIiLDAseyJvZmZzZXQiOi00LCJzdHlsZSI6eyJ0YWlsIjp7Im5hbWUiOiJhcnJvd2hlYWQifSwiaGVhZCI6eyJuYW1lIjoibm9uZSJ9fX1dLFsxLDIsIiIsMCx7Im9mZnNldCI6Mn1dLFsyLDEsIiIsMCx7Im9mZnNldCI6LTR9XSxbMiwzXSxbMywyLCIiLDAseyJvZmZzZXQiOi0xfV0sWzIsMywiIiwwLHsib2Zmc2V0IjotMn1dLFszLDIsIiIsMCx7Im9mZnNldCI6MX1dLFsyLDMsIiIsMCx7Im9mZnNldCI6Mn1dLFszLDIsIiIsMCx7Im9mZnNldCI6LTN9XSxbMywyLCIiLDAseyJvZmZzZXQiOjN9XV0=
\[\begin{tikzcd}
	{\Hom_{\bbZ P(\CC)}(P,P)} & {\Hom_{\bbZ P(\CC)}(P,P)\otimes \Hom_{\bbZ P(\CC)}(P,P)} & {\Hom_{\bbZ P(\CC)}(P,P)^{\otimes 3}} & \cdots
	\arrow[shift left=2, tail reversed, no head, from=1-1, to=1-2]
	\arrow[shift right=2, tail reversed, no head, from=1-1, to=1-2]
	\arrow[tail reversed, no head, from=1-2, to=1-1]
	\arrow[tail reversed, no head, from=1-2, to=1-3]
	\arrow[shift left=4, tail reversed, no head, from=1-2, to=1-3]
	\arrow[shift right=2, from=1-2, to=1-3]
	\arrow[shift right=2, tail reversed, no head, from=1-3, to=1-2]
	\arrow[shift left=4, from=1-3, to=1-2]
	\arrow[from=1-3, to=1-4]
	\arrow[shift left=2, from=1-3, to=1-4]
	\arrow[shift right=2, from=1-3, to=1-4]
	\arrow[shift left, from=1-4, to=1-3]
	\arrow[shift right, from=1-4, to=1-3]
	\arrow[shift left=3, from=1-4, to=1-3]
	\arrow[shift right=3, from=1-4, to=1-3]
\end{tikzcd}\]
but this is also the Hochschild complex of the ring $\Hom_{\bbZ P(\CC)}(P,P) \cong \Hom_{\bbZ\mM(\CC)}(P,P)$.
\end{proof}
\begin{example}
Consider the weak assembler $\textbf{FinSet}$. Consider $P$ to be the finite set $\{1\}$. From \cref{ex: Morita object in FinSet}, \cref{prop: Morita equiv between ZM and ZP} and \cref{corr: HH of ZM and ZP}, we have $HH_{\ast}\bbZ\mM(\textbf{FinSet}) \cong HH_{\ast}\bbZ\{1\}$. From \cref{ex: ZM for FinSets} and \cref{lemma: Hochschild homology of weak assembler}, we have $HH_{\ast}\bbZ\{1\} \cong HH_{\ast}(\Hom_{\bbZ\mM(\textbf{FinSet})}(\{1\},\{1\}))$ with $\Hom_{\bbZ\mM(\textbf{FinSet})}(\{1\},\{1\})$ generated by the identity map $$\{1\}\xleftarrow{id}\{1\}\xrightarrow{id}\{1\}$$
as free abelian group. In fact this is the identity element in the ring of endomorphisms of $\{1\}$ and as a ring $\Hom_{\bbZ\mM(\textbf{FinSet})}(\{1\},\{1\}) \cong \bbZ$ giving us
$$HH_{\ast}\bbZ\mM(\textbf{FinSet})\cong \bbZ.$$
\end{example}
\begin{example}\label{ex: HH of P_n}
Consider the category $\PP_{\bbZ}^{n}$ (with only the translation restricted isometry group). Consider $P$ to be $[0,1/n]$ which is a Morita object as discussed in \cref{ex: Morita object in Restricted Polytopes}(2). From \cref{prop: Morita equiv between ZM and ZP}, \cref{corr: HH of ZM and ZP} and \cref{lemma: Hochschild homology of weak assembler}, we have $HH_{\ast}\bbZ\mM(\PP_{\bbZ}^{n}) \cong HH_{\ast}\bbZ [0,1/n] \cong HH_{\ast}(\Hom_{\bbZ\mM(\PP_{\bbZ}^{n})}([0,1/n],[0,1/n]))$. Using \cref{ex: ZM for Restricted Polytopes}(2), we get $\Hom_{\bbZ\mM(\CC)}([0,1/n],[0,1/n])$ is generated by the morphism $$[0,1/n] \xhookleftarrow{id} [0,1/n] \xhookrightarrow{id} [0,1/n]$$ and is isomorphic, not just as a free abelian group but as a ring, to $\bbZ$. Thus, we have
    $$HH_{\ast}\bbZ\mM(\CC)\cong \bbZ.$$
    On the other hand, if we consider the same Morita object in the category $\PP_{\bbZ\rtimes C_2}^{n}$, the calculations are different. In this case, we have $\Hom_{\bbZ\mM(\PP_{\bbZ\rtimes C_2}^{n})}([0,1/n],[0,1/n])$ generated as a free abelian group by morphisms $$e_1:= [0,1/n] \xhookleftarrow{id} [0,1/n] \xhookrightarrow{id} [0,1/n] \text{ and } e_2:= [0,1/n] \xhookleftarrow{id} [0,1/n] \xhookrightarrow{T_{1/n}\circ r} [0,1/n]$$ from \cref{ex: ZM for Restricted Polytopes}(2). To understand the ring structure, note that
    \begin{align*}
        e_1^2 &= e_1\\
        e_1e_2 &= e_2\\
        e_2e_1 &= e_2\\
        e_2^2 &= e_1
    \end{align*}
 which extends to multiplication of any two terms of the form $\alpha e_1 + \beta e_2$. Thus $e_1$ is the identity element of the ring and $e_2$ is an involution. As rings we have the following isomorphism $$\Hom_{\bbZ\mM(\PP_{\bbZ\rtimes C_2}^{n})}([0,1/n],[0,1/n]) \cong \bbZ [C_2] \cong \bbZ [x]/(1-x^2)$$
 where $\bbZ[C_2]$ is the group ring associated the cyclic group of order $2$. Thus, using \cref{lemma: HH of group rings} below, we have
$$HH_{\ast}\bbZ\mM(\PP^n_{\bbZ\rtimes C_2})\cong HH_{\ast}(\bbZ[C_2]) \cong H_{\ast}(C_2)\oplus H_{\ast}(C_2) \cong \begin{cases}
    \bbZ\oplus\bbZ & \text{if } \ast = 0 \\
    \bbZ/2\bbZ\oplus\bbZ/2\bbZ  & \text{if } \ast = 2n+1, n\in \bbZ_{\geq 0}\\
    0 & \text{if } \ast = 2n, n\in \bbZ_{> 0}
\end{cases}.$$
\end{example}
\begin{example}\label{ex: HH of P_mm}
Similarly, consider the category $\PP^{m,m}_{\bbZ^2}$ with the Morita object $I_m\times I_m$ (\cref{ex: Morita object in Restricted Polytopes}(3)). As in the last example, and using \cref{ex: ZM for Restricted Polytopes}(3) we have, $\Hom_{\bbZ\mM(\PP^{m,m}_{\bbZ^2})}(I_m\times I_m,I_m\times I_m) \cong \bbZ$ as a ring, giving us
$$HH_{\ast}(\bbZ\mM(\PP^{m,m}_{\bbZ^2})) \cong \bbZ.$$

On the other hand, for the category $\PP^{m,m}_{\bbZ^2\rtimes D_4}$ with the same Morita object, we have 
$$\Hom_{\bbZ\mM(\PP^{m,m}_{\bbZ^2\rtimes D_4})}(I_m\times I_m,I_m\times I_m)$$
is generated as a free abelian group by 
\begin{align*}
	& e_1 := I_m\times I_m \xleftarrow{id} I_m\times I_m \xrightarrow{id} I_m\times I_m, \  & e_2 := I_m\times I_m \xleftarrow{\rho} I_m\times I_m \xrightarrow{id} I_m\times I_m,\\
	& e_3 := I_m\times I_m \xleftarrow{\rho^2} I_m\times I_m \xrightarrow{id} I_m\times I_m, \  & e_4 := I_m\times I_m \xleftarrow{\rho^3} I_m\times I_m \xrightarrow{id} I_m\times I_m,\\
	& e_5 := I_m\times I_m \xleftarrow{T_{1,0}\circ r} I_m\times I_m \xrightarrow{id} I_m\times I_m, \  & e_6 := I_m\times I_m \xleftarrow{\rho\circ T_{1,0}\circ r} I_m\times I_m \xrightarrow{id} I_m\times I_m,\\
	& e_7 := I_m\times I_m \xleftarrow{\rho^2\circ T_{1,0}\circ r} I_m\times I_m \xrightarrow{id} I_m\times I_m, \  & e_8 := I_m\times I_m \xleftarrow{\rho^3\circ T_{1,0}\circ r} I_m\times I_m \xrightarrow{id} I_m\times I_m,
\end{align*}
using \cref{ex: ZM for Restricted Polytopes}(3) and similar equivalences as in the last example. Here, $\rho$ denotes rotation by $\pi/2$, $r$ denotes the reflection along $y$-axis and $T_{1,0}$ denotes the translation by $(1/m,0)$. We claim that we have an isomorphism of rings
$$f: \Hom_{\PP^{m,m}_{\bbZ \rtimes D_4}}(I_m\times I_m, I_m\times I_m) \overset{\cong}{\longrightarrow} \bbZ[D_4]$$
via the map that sends $e_i$ to the corresponding element of $D_4$ that marks the first map in its zigzag. Note that this map has a natural inverse. The group homomorphism is clear for both $f$ and $f^{-1}$, and to show the claim it is enough to show that the relations that the generators of $D_4$ satisfy are satisfied by the corresponding $e_i$s. One checks 
\begin{align*}
	e_2^4 &= e_1\\
	e_5^2 &= e_1\\
	e_5e_2e_5 &= e_2^3
\end{align*}
proving the claim. Then to use the \cref{lemma: HH of group rings} below, we note that $D_4$ has five conjugacy classes with centralizers $D_4, D_4, C_4, C_2\times C_2, C_2\times C_2$.
Using their group homologies we get
$$HH_{\ast}\bbZ\mM(\PP^{m,m}_{\bbZ^2\rtimes D_4})\cong HH_{\ast}(\bbZ[D_4]) \cong 2H_{\ast}(D_4)\oplus 2H_{\ast}(C_2\times C_2)\oplus H_{\ast}(C_4)$$
where
\[H_q(D_4;\bbZ) \cong \begin{cases}
    \bbZ & \text{if } q = 0 \\
    (\bbZ/2\bbZ)^{(q+3)/2}  & \text{if } q \equiv 1 (\text{mod } 4)\\
    (\bbZ/2\bbZ)^{(q+1)/2} \oplus \bbZ/4\bbZ  & \text{if } q \equiv 3 (\text{mod } 4)\\
    (\bbZ/2\bbZ)^{q/2} & \text{if } q\in 2\bbZ_{> 0}
\end{cases},\]
\[H_q(C_4;\bbZ) \cong \begin{cases}
\bbZ & \text{if } q = 0 \\
   0 & \text{if } q \text{ even} \\
    \bbZ/4\bbZ  & \text{if } q \text{ odd} 
\end{cases},
\]
\[H_q(C_2\times C_2;\bbZ) \cong \begin{cases}
\bbZ & \text{if } q = 0 \\
   (\bbZ/2\bbZ)^{q/2} & \text{if } q \text{ even} \\
    (\bbZ/2\bbZ)^{(q+3)/2}  & \text{if } q \text{ odd} 
\end{cases}.
\]
\end{example}
\begin{remark}
    This can be generalized with no extra techniques to `square polytopes' in dimension $n$, $\PP_{\bbZ^n\rtimes G_n}^{m,m,\cdots ,m}$, where $G_n$ is the group of isometries of $I_m^{\times n}$. One can compute the Hochschild homology completely explicitly as long as the group homology of the centralizer groups of each conjugacy class of $G_n$ is calculable.
\end{remark}
\begin{example}\label{ex: HH of P_01}
  Consider the category $\PP^{[0,1]}_1$ with Morita object $P = [0,1]$. Recall from \cref{ex: ZM for Restricted Polytopes}(1), $\Hom_{\PP^{[0,1]}_1}([0,1],[0,1])$ is generated as a free abelian group by the set of finite disjoint union of closed intervals $R\subset [0,1]$ which we denote as $M$. An element $R$ of $M$ corresponds to the span
   $$[0,1] \xhookleftarrow{\supset} R\xhookrightarrow{\subset} [0,1].$$ The composite (and hence the monoidal product) of two elements $R, R'$ in $M$ is given by their intersection. Hence, $M$ is a commutative monoid with $[0,1]$ being the identity element. This gives us $$\Hom_{\PP^{[0,1]}_1}([0,1],[0,1]) \cong \bbZ[M]$$ as commutative rings which immediately tells us $HH_0(\bbZ \PP^{[0,1]}_1)\cong \bbZ[M]$. Next we show that all the higher Hochschild homology of $\bbZ[M]$ are trivial: 
  
  Let $$S(F_1, F_2 \cdots , F_n)=\left\{ \bigcap_{i\in I}F_i| I \subset \{1,2, \cdots, n\}\right\}$$ be the finite (at most $2^n$ elements) submonoid generated by the closed sub intervals $F_i \subset [0,1]$. Then $\bbZ[M] = \operatorname{colim}_{S\subset M} \bbZ[S]$. $S$ is in fact a poset with partial order given by subset relation. By \cite[Theorem 1, page 605]{Solomon1967}, $\bbZ[S]$ is isomorphic as a ring to direct sum of copies of $\bbZ$. As $HH_n(\bbZ) \cong 0$ for $n>0$, $HH_n(\bbZ[S]) = 0$. Since $\bbZ[M]$ is colimit of such rings, $HH_n(\bbZ[M])=0$ for $n>0$. Thus for $n>0$,
  $$HH_n(\bbZ\PP^{[0,1]}_1)\cong HH_n(\bbZ[M]) \cong 0.$$
  
The category $\PP^{[0,1]}_{C_2}$ also has $P=[0,1]$ as the Morita object. In $\Hom_{\PP^{[0,1]}_{C_2}}([0,1],[0,1])$ we can check by taking composition of morphisms that the ring multiplication is
        $$(D_1,\phi_1) \odot (D_2, \phi_2) = (D_1 \cap \phi_1 (D_2), \phi_1 \phi_2 )$$
        showing that $\Hom_{\bbZ\mM(\PP^{[0,1]}_{C_2})}([0,1],[0,1]) \cong \bbZ[M]_e \oplus \bbZ[M]_u $ graded by the group $C_2$ (subscripts denote grading) as rings. Lets denote this ring by $\RR$. This is strongly graded by $C_2$ so we shall use \cref{lemma: HH of graded rings} to compute the Hochschild homology. We have a decomposition
    $$HH_* (\RR) \cong HH_*(\RR)_{\langle e \rangle} \oplus HH_*(\RR)_{\langle u \rangle}$$
    and spectral sequences,
    $$ H_p(C_2, HH_q(\bbZ[M]_e, \bbZ[M]_z)) \implies HH_{p+q} (\RR)_{\langle z \rangle} $$
    for each $z\in C_2$. Using essentially the same argument as above (\cite[Theorem 1]{Solomon1967} and $HH_*$ commutes with filtered colimits), $$HH_q(\bbZ[M], \bbZ[M]_z) = 0 $$
    for $q>0$ and $z\in C_2$. As $\bbZ[M]$ is a commutative ring,
    $$HH_0(\bbZ[M]_e, \bbZ[M]_e) \cong \bbZ[M] \text{ and }HH_0(\bbZ[M]_e, \bbZ[M]_u) = \bbZ[M]_u/ [\bbZ[M]_e, \bbZ[M]_u]. $$
    We first note that $[\bbZ[M]_e, \bbZ[M]_u]$ is freely generated by elements of the form $(D_1,e)\odot  (D_2, u) - (D_2, u) \odot (D_1,e) = (D_1\cap D_2, u) - (D_2 \cap u(D_1), u)$, hence as a module $$HH_0(\bbZ[M]_e, \bbZ[M]_u) = \bbZ[M]/ \langle X\cap Y  -X \cap u(Y) \rangle_{X,Y \in M}. $$
    It is easy to see that 
    $$X \cap u(X) \sim X\cap u^2(X) \sim  X\text{ and } X\cap u(X)= Y\cap u(Y) \iff X\sim Y.$$
    Hence the map $\bbZ[M]/ \langle X\cap Y  -X \cap u(Y) \rangle \xrightarrow[]{} \bbZ[M^{C_2}]$ given by $[X]\mapsto [X\cap u(X)]$ is an isomorphism. The spectral sequence is concentrated in row $q=0$ and therefore collapses to 
    $$H_n(\RR)_{\langle e \rangle } \cong H_n(C_2; \bbZ[M]) \text{ and }  H_n (\RR)_{\langle u \rangle} \cong H_n(C_2; \bbZ[M^{C_2}]). $$
    The action of $C_2$ on $\bbZ[M]$ by reflecting the generators by about $1/2$, and trivial action on $\bbZ[M^{C_2}]$.  From \cite[Theorem 6.2.2]{weibel1994introduction}, the group homology of the cyclic group is given by
    $$H_n(C_2; \bbZ[M])\cong 
    \begin{cases}
        \frac{\bbZ[M]}{(u-1)\bbZ[M]} & n=0\\
        \frac{\bbZ[M]^{C_2}}{(u+1)\bbZ[M]} & n= \text{odd}\\
        \frac{\{X\in \bbZ[M]| (u+1)X=0\}}{(u-1)\bbZ[M]} & n= \text{even}
    \end{cases}
    =
    \begin{cases}
        \bbZ[M/{C_2}] & n=0\\
        (\bbZ/2)[M^{C_2}] & n= \text{odd}\\
        0 & n= \text{even}
    \end{cases}
    $$
    $$H_n(C_2; \bbZ[M^{C_2}]) \cong \begin{cases}
        \frac{\bbZ[M^{C_2}]}{(u-1)\bbZ[M^{C_2}]} & n=0\\
        \frac{\bbZ[M^{C_2}]^{C_2}}{(u+1)\bbZ[M^{C_2}]} & n= \text{odd}\\
        \frac{\{X\in \bbZ[M^{C_2}]| (u+1)X=0\}}{(u-1)\bbZ[M^{C_2}]} & n= \text{even}
    \end{cases}
    =
    \begin{cases}
        \bbZ[M^{C_2}] & n=0\\
        (\bbZ/2)[M^{C_2}] & n= \text{odd}\\
        0 & n= \text{even}
    \end{cases}
    $$
    Putting this together gives 
    $$HH_n(\RR) \cong 
    \begin{cases}
        \bbZ[M/{C_2}]\oplus \bbZ[M^{C_2}] & n=0\\
        (\bbZ/2)[M^{C_2}]\oplus (\bbZ/2)[M^{C_2}] & n= \text{odd}\\
        0 & n = \text{even}
    \end{cases}.$$
\end{example}    
\begin{example}
 Consider the weak assembler $\PP^{\mathcal{E}^1}_{G}$ for $G:= \bbR$ or $G:= \bbR\rtimes C_2$. $P:= [0,1]$ is a Morita object from \cref{ex: Morita object in Polytopes}. From \cref{prop: Morita equiv between ZM and ZP} and \cref{corr: HH of ZM and ZP}, we have $HH_{\ast}\bbZ\mM(\PP_G^{\mathcal{E}^1}) \cong HH_{\ast}\bbZ[0,1]$. From \cref{lemma: Hochschild homology of weak assembler}, $HH_{\ast}\bbZ[0,1] \cong HH_{\ast}(\Hom_{\bbZ\mM(\PP_G^{\mathcal{E}^1})}([0,1],[0,1]))$. This is a really complicated ring and its hard to say what its Hochschild homology is. We know that similar to previous examples, we can reduce the free abelian group generators to morphisms of the type $[0,1] \xhookleftarrow{inc}R \xhookrightarrow{}[0,1]$ where the first map is just the set inclusion. The ring multiplication is composition and we know its a non-commutative ring. For example the following two morphisms
\[
f:= % https://q.uiver.app/#q=WzAsMyxbMSwwLCJbMCwxLzJdIl0sWzAsMSwiWzAsMV0iXSxbMiwxLCJbMCwxXSJdLFswLDEsImlkIiwyLHsic3R5bGUiOnsidGFpbCI6eyJuYW1lIjoiaG9vayIsInNpZGUiOiJib3R0b20ifX19XSxbMCwyLCJUX3sxLzJ9IiwwLHsic3R5bGUiOnsidGFpbCI6eyJuYW1lIjoiaG9vayIsInNpZGUiOiJ0b3AifX19XV0=
\begin{tikzcd}[cramped, column sep = small, row sep = small]
	& {[0,1/2]} & \\
	{[0,1]} && {[0,1]}
	\arrow["id"', hook', from=1-2, to=2-1]
	\arrow["{T_{1/2}}", hook, from=1-2, to=2-3]
\end{tikzcd},
\ g:=
% https://q.uiver.app/#q=WzAsMyxbMSwwLCJbMCwxLzJdIl0sWzAsMSwiWzAsMV0iXSxbMiwxLCJbMCwxXSJdLFswLDEsImlkIiwyLHsic3R5bGUiOnsidGFpbCI6eyJuYW1lIjoiaG9vayIsInNpZGUiOiJib3R0b20ifX19XSxbMCwyLCJUX3sxLzN9IiwwLHsic3R5bGUiOnsidGFpbCI6eyJuYW1lIjoiaG9vayIsInNpZGUiOiJ0b3AifX19XV0=
\begin{tikzcd}[cramped, column sep = small, row sep = small]
	& {[0,1/2]} & \\
	{[0,1]} && {[0,1]}
	\arrow["id"', hook', from=1-2, to=2-1]
	\arrow["{T_{1/3}}", hook, from=1-2, to=2-3]
\end{tikzcd}
\]
give
\[
fg = % https://q.uiver.app/#q=WzAsMyxbMSwwLCJcXGVtcHR5c2V0Il0sWzAsMSwiWzAsMV0iXSxbMiwxLCJbMCwxXSJdLFswLDEsIiIsMix7InN0eWxlIjp7InRhaWwiOnsibmFtZSI6Imhvb2siLCJzaWRlIjoiYm90dG9tIn19fV0sWzAsMiwiIiwwLHsic3R5bGUiOnsidGFpbCI6eyJuYW1lIjoiaG9vayIsInNpZGUiOiJ0b3AifX19XV0=
\begin{tikzcd}[cramped, column sep = small, row sep = small]
	& \emptyset & \\
	{[0,1]} && {[0,1]}
	\arrow[hook', from=1-2, to=2-1]
	\arrow[hook, from=1-2, to=2-3]
\end{tikzcd},
\text{ }
gf=% https://q.uiver.app/#q=WzAsMyxbMSwwLCJbMS8zLCAxLzJdIl0sWzAsMSwiWzAsMV0iXSxbMiwxLCJbMCwxXSJdLFswLDEsIlRfey0xLzN9IiwyLHsic3R5bGUiOnsidGFpbCI6eyJuYW1lIjoiaG9vayIsInNpZGUiOiJib3R0b20ifX19XSxbMCwyLCJUX3sxLzJ9IiwwLHsic3R5bGUiOnsidGFpbCI6eyJuYW1lIjoiaG9vayIsInNpZGUiOiJ0b3AifX19XV0=
\begin{tikzcd}[cramped, column sep = small, row sep = small]
	& {[1/3, 1/2]} & \\
	{[0,1]} && {[0,1]}
	\arrow["{T_{-1/3}}"', hook', from=1-2, to=2-1]
	\arrow["{T_{1/2}}", hook, from=1-2, to=2-3]
\end{tikzcd}
\]
and therefore we do have a non trivial commutator for calculating even $HH_0$. Lot of classes like $gf$ above will get killed when we mod out by the commutator for $HH_0$ computation. However, so far we are unable to say more about the structure of this ring and its Hochschild homology.
\end{example}
\begin{lemma}\label{lemma: HH of group rings}\cite[Theorem 7.4.6]{Loday}
The Hochschild homology of the group ring $\bbZ[G]$ is given by
$$ HH_{\ast} (\bbZ[G]) \cong \bigoplus_{\langle z \rangle \in \operatorname{Conj}(G)} H_{\ast}(G_z; \bbZ)$$
where $\operatorname{Conj(G)}$ is the set of conjugacy classes of the group $G$, and $G_z$ is the centralizer of $z\in G$.
\end{lemma}

\begin{lemma}\label{lemma: HH of graded rings}\cite{Lorenz1992}
    Let $S$ be an algebra strongly graded by group $G$ i.e. $S = \bigoplus_{g\in G} S_g$ such that $S_gS_h = S_{gh}$ for all $g, h \in G$. Then there is a canonical decomposition with components labeled by each conjugacy class of $G$, 
    $$HH_*(S) \cong \bigoplus_{\langle z \rangle \in \operatorname{Conj}(G)} HH_*(S)_{\langle z \rangle} $$
    and each component has a description in terms of the following spectral sequence,
    $$E^2_{p,q} = H_p(G_z, HH_q(R, S_z)) \implies HH_{p+q} (S)_{\langle z \rangle}$$
    where $HH_q(R, S_z)$ is the Hochschild homology of the identity component $R:= S_e$ of $S$ with coefficients in the bimodule $S_z$, and $H_p(G_z, \cdot )$ is the group homology of the centralizer $C_z$ of $z\in G$.
\end{lemma}
%%%%%%%%%%%%%%%%%%%%%%%%%%%%%%%%%%%%%%%%%%%%%%%%%%%%%%%
%%%%%%%%%%%NEW SECTION%%%%%%%%%%%%%%%%%%%%%%%%%%%%%%%%%
%%%%%%%%%%%%%%%%%%%%%%%%%%%%%%%%%%%%%%%%%%%%%%%%%%%%%%%
\section{The Dennis Trace Map for Weak Assemblers}\label{Subsec:Dennis_Trace_for_Assemblers}

In this section, we construct the Dennis trace map (\cref{const: Dennis trace map}) for a weak assembler (\cref{def:Assemblers}) with a Morita object (\cref{def: ZP}), $\CC$. The construction follows the construction of Dennis trace map for rings outlined in \cref{sec:Dennis_trace_for_rings}. The analogue to the first object in \eqref{Eq:Dennis_trace_sequence}, for us, will be $B\gG(\CC)$ from \cref{def:Assemblers_K_theory} and the analogue to the second object will be the Hochschild complex of $\bbZ\mM(\CC)$. The analogue of the last object will be Hochschild complex of $\bbZ P(\CC)$. We also explicitly study the Dennis trace map at level $0$ (\cref{thm: Dennis trace at level 0}). 

\begin{construction}\label{const: Dennis trace map}
The Dennis trace map will be a composite of a few maps, each of which we first construct individually:
\begin{enumerate}
    \item\label{const: Dennis trace map 1} For a groupoid, $G$, there is a natural map from $BG$, the nerve of $G$, to $B^{cyc}G$, the cyclic nerve of $G$ as in \cref{Notations}(\ref{Not:cyclic_nerve}). We use it to map $$cyc: B\gG(\CC) \to B^{cyc}\gG(\CC).$$ This will be our first map towards Dennis trace map.
    \item\label{const: Dennis trace map 2} The composite of the map from $\gG(\CC) \to \mM(\CC)$ in \eqref{eq:WC to GC to MC} and the map from $\mM(\CC) \to \bbZ\mM(\CC)$ in \eqref{eq:MC to ZMC} induces a map
    \begin{equation*}
     B^{cyc}\gG(\CC) \to B^{cyc}\bbZ\mM(\CC).   
    \end{equation*}
    Composing this with the map $$B^{cyc}\bbZ\mM(\CC) \to HH(\bbZ\mM(\CC))$$ from \cref{Notations}(\ref{Not:cyclic_bar_complex}) gives
    $$I: B^{cyc}\gG(\CC) \to HH(\bbZ\mM(\CC)).$$ This will be our second map.
    \item\label{const: Dennis trace map 3} Next we have the equivalence $HH(\bbZ\mM(\CC))\simeq HH(\bbZ P(\CC))$ from \cref{corr: HH of ZM and ZP}. This will be the final part of our Dennis trace construction.   
\end{enumerate}
 The Dennis trace map is defined using the composite of (\ref{const: Dennis trace map 1}), (\ref{const: Dennis trace map 2}) and (\ref{const: Dennis trace map 3}) above
 \begin{equation}\label{eq:pre dennis trace}
     B\gG(\CC) \to HH(\bbZ\mM(\CC)) \simeq HH(\bbZ P(\CC))
 \end{equation}
 by taking the group completion of the associated map of topological spaces
 \begin{align*}
    & \Omega B|B\gG(\CC)| \to \Omega B|HH(\bbZ\mM(\CC))| \simeq \Omega B|HH(\bbZ P(\CC))|\\
    \overset{\cref{def:Assemblers_K_theory}}{\Longrightarrow}& DTr: K(\CC) \to |HH(\bbZ\mM(\CC))| \xrightarrow[]{\simeq} |HH(\bbZ P(\CC))|
 \end{align*}
 in the homotopy category of spaces. For the last implication we use that $\Omega B|HH(\bbZ\mM(\CC))|$ and $\Omega B|HH(\bbZ P(\CC))|$ are group completions of $|HH(\bbZ\mM(\CC))|$ and $|HH(\bbZ P(\CC))|$ which are already group objects. We call this map the Dennis trace map and denote it as $DTr$. 
\end{construction}

To compute where a $K$-theory class goes under the Dennis trace map, we use \eqref{eq:pre dennis trace} and the zigzag between $HH(\bbZ\mM(\CC))$ and $HH(\bbZ P(\CC))$ from \cref{thm: Morita_Invariance_Hochschild_homology}.  
\begin{theorem}\label{thm: Dennis trace at level 0}
Let $[Q]$ be a generating class of $K_0(\CC)$ as in \cref{rmk: K0} and $n$ be the smallest positive integer such that there is an inclusion $Q \to P^n$ as in \cref{def: ZP}, aka there exists a morphism class in $\bbZ\mM(\CC)$
% https://q.uiver.app/#q=WzAsMyxbMCwxLCJRIl0sWzIsMSwiUF57bn0iXSxbMSwwLCJRXzFRXzJcXGNkb3RzIFFfbiJdLFsyLDBdLFsyLDEsIlxccHNpXzFcXHBzaV8yXFxjZG90cyBcXHBzaV9uIiwxLHsic3R5bGUiOnsidGFpbCI6eyJuYW1lIjoiaG9vayIsInNpZGUiOiJ0b3AifX19XV0=
\[\begin{tikzcd}[cramped, column sep=small]
	& {Q_1Q_2\cdots Q_n} & \\
	Q && {P^{n}}
	\arrow[from=1-2, to=2-1]
	\arrow["{\psi_1\psi_2\cdots \psi_n}"{description}, hook, from=1-2, to=2-3]
\end{tikzcd}\]
with $Q_i$ mapping into the $i$th copy of $P$ via a morphism $\psi_i$ in $\Mor(\wW^{par}\CC)$. Then the image of $[Q] \in K_0(\CC)$ under Dennis trace is the class of following element in $HH_0(\bbZ P(\CC))$:
% https://q.uiver.app/#q=WzAsMyxbMSwwLCJRXzEiXSxbMCwxLCJQIl0sWzIsMSwiUCJdLFswLDEsIlxccHNpXzEiLDIseyJzdHlsZSI6eyJ0YWlsIjp7Im5hbWUiOiJob29rIiwic2lkZSI6ImJvdHRvbSJ9fX1dLFswLDIsIlxccHNpXzEiLDAseyJzdHlsZSI6eyJ0YWlsIjp7Im5hbWUiOiJob29rIiwic2lkZSI6InRvcCJ9fX1dXQ==
\[
\begin{tikzcd}[cramped, row sep=small, column sep=small]
	& {Q_1} & \\
	P && P
	\arrow["{\psi_1}"', hook', from=1-2, to=2-1]
	\arrow["{\psi_1}", hook, from=1-2, to=2-3]
\end{tikzcd}
+
\begin{tikzcd}[cramped, row sep=small, column sep=small]
	& {Q_2} & \\
	P && P
	\arrow["{\psi_2}"', hook', from=1-2, to=2-1]
	\arrow["{\psi_2}", hook, from=1-2, to=2-3]
\end{tikzcd}
+
\ \cdots \
+
\begin{tikzcd}[cramped, row sep=small, column sep=small]
	& {Q_n} & \\
	P && P.
	\arrow["{\psi_n}"', hook', from=1-2, to=2-1]
	\arrow["{\psi_n}", hook, from=1-2, to=2-3]
\end{tikzcd}
\]
\end{theorem}
\begin{proof}
Since $B\gG(\CC) \simeq \sqcup_{[Q]}B\Aut(Q)$ where $[Q]$ varies over isomorphism classes of objects in $\gG(\CC)$, a $K_0(\CC)$ class is given by an isomorphism class $[Q]$ (there are formal inverses in $K_0$ group but it is enough to find the image under Dennis trace of these generators). Under the map \eqref{eq:pre dennis trace}, $Q$ in $B\gG(\CC)_0$ maps to $id: Q\xrightarrow[]{} Q$ in $B^{cyc}\gG(\CC)_0$ and therefore further to 
\begin{equation}\label{eq:K0 class}
    id: Q\xrightarrow[]{} Q
\end{equation}
in $\Hom_{\bbZ\mM(\CC)}(Q,Q)$ in $HH(\bbZ\mM(\CC))_0$. Next we use \cref{thm: Morita_Invariance_Hochschild_homology}, to find an element of $HH(\bbZ P(\CC))_0$ that corresponds to this element of $HH(\bbZ\mM(\CC))_0$. Writing out the zig-zag for the categories $\bbZ P(\CC)$ and $\bbZ\mM(\CC)$ at simplicial level $0$, gives the following commuting diagram where for ease we just write $\Hom$ without clarifying the category we are taking $\Hom$ in. All the $\Hom$ are in fact $\Hom_{\bbZ\mM(\CC)}$ in this commuting diagram:
 \[
\begin{tikzcd}[cramped]
    \Hom(P,P) \ar[d,"="] & \Hom(P,P)\otimes\Hom(P,P) \ar[l,"multiply"'] \ar[r,"switch"] \ar[d] & \Hom(P,P)\otimes\Hom(P,P) \ar[d] \ar[r,"multiply"] & \Hom(P,P) \ar[d]\\
    \Hom(P,P) & \bigoplus\limits_{R}(\Hom(P,R)\otimes\Hom(R,P)) \ar[l,"multiply"'] \ar[r,"switch"] & \bigoplus\limits_{R} (\Hom(R,P)\otimes \Hom(P,R))
    \ar[r,"multiply"] & \bigoplus\limits_{R}\Hom(R,R)
\end{tikzcd}
\]  
and in the direct sums, $R$ varies over all objects of $\bbZ\mM(\CC)$. Note that all vertical maps except the first are inclusion into the direct summands. Consider the object 
% https://q.uiver.app/#q=WzAsMyxbMCwxLCJQIl0sWzEsMCwiUV9pIl0sWzIsMSwiUV9pIl0sWzEsMCwiXFxwc2lfaSIsMix7InN0eWxlIjp7InRhaWwiOnsibmFtZSI6Imhvb2siLCJzaWRlIjoiYm90dG9tIn19fV0sWzEsMiwiXFxvdmVybGluZXtpZH1fe1xccGhpfSJdXQ==
\[\sum\limits_{i=1}^n
\left(\begin{tikzcd}[cramped, row sep=small, column sep=small]
	& {Q_i} & \\
	P && {Q_i}
	\arrow["{\psi_i}"', hook', from=1-2, to=2-1]
	\arrow["{id}", from=1-2, to=2-3]
\end{tikzcd}
\otimes
% https://q.uiver.app/#q=WzAsMyxbMCwxLCJRX2kiXSxbMSwwLCJRX2kiXSxbMiwxLCJQIl0sWzEsMiwiXFxwc2lfaSIsMCx7InN0eWxlIjp7InRhaWwiOnsibmFtZSI6Imhvb2siLCJzaWRlIjoidG9wIn19fV0sWzEsMCwiXFxvdmVybGluZXtpZH1fe1xccGhpfSIsMl1d
\begin{tikzcd}[cramped, row sep=small, column sep=small]
	& {Q_i} & \\
	{Q_i} && P
	\arrow["{id}"', from=1-2, to=2-1]
	\arrow["{\psi_i}", hook, from=1-2, to=2-3]
\end{tikzcd}\right)
\]
in $\oplus_{R}(\Hom(P,R)\otimes\Hom(R,P))$. Under the map to the left of $\oplus_{R}(\Hom(P,R)\otimes\Hom(R,P))$ in the diagram, this object goes to 
% https://q.uiver.app/#q=WzAsMyxbMCwxLCJQIl0sWzEsMCwiUV9pIl0sWzIsMSwiUCJdLFsxLDAsIlxccHNpX2kiLDIseyJzdHlsZSI6eyJ0YWlsIjp7Im5hbWUiOiJob29rIiwic2lkZSI6ImJvdHRvbSJ9fX1dLFsxLDIsIlxccHNpX2kiLDAseyJzdHlsZSI6eyJ0YWlsIjp7Im5hbWUiOiJob29rIiwic2lkZSI6InRvcCJ9fX1dXQ==
\begin{equation}\label{eqn:target}
\sum\limits_{i=1}^n
\begin{tikzcd}[cramped, row sep=small, column sep=small]
	& {Q_i} & \\
	P && P
	\arrow["{\psi_i}"', hook', from=1-2, to=2-1]
	\arrow["{\psi_i}", hook, from=1-2, to=2-3]
\end{tikzcd}
\end{equation}
in $\Hom(P,P)$. Whereas, under the map to the right this object goes to first
% https://q.uiver.app/#q=WzAsMyxbMCwxLCJRX2kiXSxbMSwwLCJRX2kiXSxbMiwxLCJQIl0sWzEsMiwiXFxwc2lfaSIsMCx7InN0eWxlIjp7InRhaWwiOnsibmFtZSI6Imhvb2siLCJzaWRlIjoidG9wIn19fV0sWzEsMCwiXFxvdmVybGluZXtpZH1fe1xccGhpfSIsMl1d
\[\sum\limits_{i=1}^n
\left(\begin{tikzcd}[cramped, row sep=small, column sep=small]
	& {Q_i} & \\
	{Q_i} && P
	\arrow["{id}"', from=1-2, to=2-1]
	\arrow["{\psi_i}", hook, from=1-2, to=2-3]
\end{tikzcd}
\otimes
\begin{tikzcd}[cramped, row sep=small, column sep=small]
	& {Q_i} & \\
	P && {Q_i}
	\arrow["{\psi_i}"', hook', from=1-2, to=2-1]
	\arrow["{id}", from=1-2, to=2-3]
\end{tikzcd}\right)
\]
and then to
% https://q.uiver.app/#q=WzAsMyxbMSwwLCJRX2kiXSxbMCwxLCJRX2kiXSxbMiwxLCJRX2kiXSxbMCwxLCJpZCIsMl0sWzAsMiwiaWQiXV0=
\begin{equation}\label{eqn:domain}
\sum\limits_{i=1}^n
\begin{tikzcd}[cramped, row sep=small, column sep=small]
	& {Q_i} & \\
	{Q_i} && {Q_i}
	\arrow["id"', from=1-2, to=2-1]
	\arrow["id", from=1-2, to=2-3]
\end{tikzcd}
\end{equation}
in $\oplus_{R}\Hom(R,R)$. So far we have shown, under the zigzag, object \eqref{eqn:domain} in $HH(\bbZ\mM(\CC))_0$ corresponds to \eqref{eqn:target} in $HH(\bbZ P(\CC))_0$. 

Now note that we have the following two objects in $HH(\bbZ\mM(\CC))_1$, where $f$ is any isomorphism between $Q_1Q_2\cdots Q_n$ and $Q$:
\begin{equation}
% https://q.uiver.app/#q=WzAsMyxbMCwxLCJRXzFRXzJcXGNkb3RzIFFfbiJdLFsyLDEsIlEiXSxbMSwwLCJRXzFRXzJcXGNkb3RzIFFfbiJdLFsyLDAsImlkIiwyXSxbMiwxLCJmIl1d
\begin{tikzcd}[cramped, row sep=small, column sep=small]
	& {Q_1Q_2\cdots Q_n} & \\
	{Q_1Q_2\cdots Q_n} && Q
	\arrow["id"', from=1-2, to=2-1]
	\arrow["f", from=1-2, to=2-3]
\end{tikzcd}
\ \otimes \
% https://q.uiver.app/#q=WzAsMyxbMCwxLCJRIl0sWzIsMSwiUV8xUV8yXFxjZG90cyBRX24iXSxbMSwwLCJRXzFRXzJcXGNkb3RzIFFfbiJdLFsyLDAsImYiLDJdLFsyLDEsImlkIl1d
\begin{tikzcd}[cramped, row sep=small, column sep=small]
	& {Q_1Q_2\cdots Q_n} & \\
	Q && {Q_1Q_2\cdots Q_n}
	\arrow["f"', from=1-2, to=2-1]
	\arrow["id", from=1-2, to=2-3]
\end{tikzcd}
\end{equation}
and
\begin{equation}
\sum\limits_{i=1}^n
    % https://q.uiver.app/#q=WzAsMyxbMSwwLCJRX2kiXSxbMCwxLCJRX2kiXSxbMiwxLCJRXzFRXzJcXGNkb3RzIFFfbiJdLFswLDEsImlkIiwyXSxbMCwyLCIiLDAseyJzdHlsZSI6eyJ0YWlsIjp7Im5hbWUiOiJob29rIiwic2lkZSI6InRvcCJ9fX1dXQ==
\left(\begin{tikzcd}[cramped, row sep=small, column sep=small]
	& {Q_i} & \\
	{Q_i} && {Q_1Q_2\cdots Q_n}
	\arrow["id"', from=1-2, to=2-1]
	\arrow[hook, from=1-2, to=2-3]
\end{tikzcd}
\ \otimes \
% https://q.uiver.app/#q=WzAsMyxbMSwwLCJRX2kiXSxbMCwxLCJRXzFRXzJcXGNkb3RzIFFfbiJdLFsyLDEsIlFfaSJdLFswLDIsImlkIl0sWzAsMSwiIiwyLHsic3R5bGUiOnsidGFpbCI6eyJuYW1lIjoiaG9vayIsInNpZGUiOiJib3R0b20ifX19XV0=
\begin{tikzcd}[cramped, row sep=small, column sep=small]
	& {Q_i} & \\
	{Q_1Q_2\cdots Q_n} && {Q_i}
	\arrow[hook', from=1-2, to=2-1]
	\arrow["id", from=1-2, to=2-3]
\end{tikzcd}\right)
\end{equation}
Taking $d_0$ and $d_1$ of these objects and then using the relation $d_0 = d_1$ in $HH_{0}(\bbZ\mM(\CC))$ gives the following equality of $HH_0$ classes:
\[
\begin{aligned}
\begin{tikzcd}[cramped, row sep=small, column sep=small]
  & {Q_1Q_2\cdots Q_n} & \\
  {Q_1Q_2\cdots Q_n} && {Q_1Q_2\cdots Q_n}
  \arrow["id"', from=1-2, to=2-1]
  \arrow["id", from=1-2, to=2-3]
\end{tikzcd}
&=
\begin{tikzcd}[cramped, row sep=small, column sep=small]
  & {Q_1Q_2\cdots Q_n} & \\
  Q && Q
  \arrow["f"', from=1-2, to=2-1]
  \arrow["f", from=1-2, to=2-3]
\end{tikzcd}
\\[1ex]
&=
\begin{tikzcd}[cramped, row sep=small, column sep=small]
  & Q & \\
  Q && Q
  \arrow["id"', from=1-2, to=2-1]
  \arrow["id", from=1-2, to=2-3]
\end{tikzcd}
\end{aligned}
\]
where the second equality is due to the two morphisms being equivalent in $\mM(\CC)$ via
% https://q.uiver.app/#q=WzAsNSxbMSwwLCJRXzFRXzJcXGNkb3RzIFFfbiJdLFswLDEsIlEiXSxbMiwxLCJRIl0sWzEsMiwiUSJdLFsxLDEsIlFfMVFfMlxcY2RvdHMgUV9uIl0sWzAsMSwiZiIsMl0sWzAsMiwiZiJdLFszLDEsImlkIl0sWzMsMiwiaWQiLDJdLFswLDQsImlkIiwyXSxbNCwzLCJmIiwyXV0=
\[\begin{tikzcd}[cramped]
	& {Q_1Q_2\cdots Q_n} & \\
	Q & {Q_1Q_2\cdots Q_n} & Q\\
	& Q
	\arrow["f"', from=1-2, to=2-1]
	\arrow["id"', from=1-2, to=2-2]
	\arrow["f", from=1-2, to=2-3]
	\arrow["f"', from=2-2, to=3-2]
	\arrow["id", from=3-2, to=2-1]
	\arrow["id"', from=3-2, to=2-3]
\end{tikzcd}\]
and
\[
\begin{aligned}
\sum_{i=1}^n
\begin{tikzcd}[cramped, row sep=small, column sep=small]
  & {Q_i} & \\
  {Q_i} && {Q_i}
  \arrow["id"', from=1-2, to=2-1]
  \arrow["id", from=1-2, to=2-3]
\end{tikzcd}
&=
\sum_{i=1}^n
\begin{tikzcd}[cramped, row sep=small, column sep=small]
  & {Q_i} & \\
  {Q_1Q_2\cdots Q_n} && {Q_1Q_2\cdots Q_n}
  \arrow[hook', from=1-2, to=2-1]
  \arrow[hook, from=1-2, to=2-3]
\end{tikzcd}
\\[1ex]
&=
\begin{tikzcd}[cramped, row sep=small, column sep=small]
  & {Q_1Q_2\cdots Q_n} & \\
  {Q_1Q_2\cdots Q_n} && {Q_1Q_2\cdots Q_n}
  \arrow["id"', from=1-2, to=2-1]
  \arrow["id", from=1-2, to=2-3]
\end{tikzcd}
\end{aligned}
\]
where the last equality is due to \cref{def: category of scissors correspondences}(3). Thus, in $HH_{0}\bbZ\mM(\CC)$, we have class \eqref{eq:K0 class} equals class \eqref{eqn:domain}. Hence, object $Q \xleftarrow[]{id}Q\xrightarrow[]{id}Q$ maps to object \eqref{eqn:target} under the map $HH_0(\bbZ\mM(\CC)) \to HH_0(\bbZ P(\CC))$ and therefore, the Dennis trace map takes a generating class $[Q]$ to \eqref{eqn:target}.
\end{proof}
\begin{example}
    For the weak assembler $\textbf{FinSet}$, the Dennis trace map at level $0$ is  
    $$DTr_0: K_0(\textbf{FinSet}) \to HH_0(\bbZ) \cong \bbZ.$$
    Here, $K_{0}$ is also isomorphic to $\bbZ$, where a set $X$ is in class $n$ if its cardinality is $n$. If we choose a class representative for $n$ to be $Q:= \{1,2,\cdots ,n\}$, then under the Dennis trace map this goes to $$n(\{1\} \xleftarrow[]{id}\{1\}\xrightarrow{id}\{1\})$$ (using \cref{ex: Morita object in FinSet}), i.e. it maps to $n\in \bbZ \cong HH_0(\bbZ)$. So, $DTr_0$ is identity map for $\textbf{FinSet}$.
\end{example}
\begin{example}
Similar to \cref{rem: image of K0(R)}, in $\PP_G^X$, the image of a $K_0$ polytope class $[Q]$ is given by its `dimension' aka the embedding of its pieces into the Morita object.
\end{example}
%%%%%%%%%%%%%%%%%%%%%%%%%%%%%%%%%%%%%%%%%%%%%%%%%%%%%%%%%%%%%%%%%
%%%%%%%%%%%%%%%%%%%%%%%%%%%Subsection 3.5%%%%%%%%%%%%%%%%%%%%%%%
%%%%%%%%%%%%%%%%%%%%%%%%%%%%%%%%%%%%%%%%%%%%%%%%%%%%%%%%%%%%%%%
\section{Dennis Trace as Refinement of Trace/Regulator}\label{subsection:Comparision_with_BGMMZ}

In \cite{bohmann2024trace}, the authors construct a map from assembler $K$-theory to group homology which they call the \emph{regulator} or \emph{trace} map. Here we show that their map factors through the Dennis trace map $DTr$, thus showing that the group homology is in fact a trace invariant. To make this comparison, we first define weak $G$-assembler and a measure $A$ on it following \cite[Definition 4.1, 4.3, 6.1]{bohmann2024trace}. We then construct a simplicial map $T_{\wW}$ from $HH(\bbZ\mM(\CC_{hG}))$ to the double sided bar complex associated to group homology of $G$ with coefficients in $A$, $B(\bbZ,\bbZ[G],A)$ (\cref{const: T_M(C)}, \cref{prop: well-defined}, \cref{prop: simplicial}) and use it show that the regulator map factors through the Dennis trace in \cref{thm: regulator factors through Dennis trace}.
\begin{definition}\label{def:G weak assembler}
    A \emph{weak $G$-assembler} is a functor $F: G \rightarrow \textbf{wAsm}$. 
\end{definition}
We shall interchangeably call the image of the functor $F$ say $\CC=F(\ast)$ as the weak $G$-assembler. $\CC$ inherits an action of the group $G$ on objects as well the covering families i.e. if $\mathcal{F}_Q=\{P_i \to Q\}$ is a cover then $g\cdot \mathcal{F}_Q:=\{gP_i\to gQ\}$ is also a cover, and $(g h)\cdot  \mathcal{F}_Q = g\cdot (h \cdot \mathcal{F}_Q)$.
%\rbnote[]{do I give one example here?}

\begin{definition}\label{def:homotopy orbit}
    For a weak $G$-assembler $\CC$, the \emph{homotopy orbit} $\CC_{hG}$ %\rbnote[]{check if this is a weak assembler and describe the covering family structure}
    is the category whose objects are same as $obj(\CC)$ and the morphisms are pairs $(f,g):P \xrightarrow{} Q$ such that $f: gP \xrightarrow{} Q $ is a morphism in $\CC$ and $g\in G$. Composition of two such morphisms is given by $(f_1,g_1)\circ (f_0,g_0) = (f_1 \circ g_1(f_0), g_1g_0)$: 
    \[P\xrightarrow[]{(f_0,g_0)}Q \xrightarrow[]{(f_1,g_1)}R := g_1g_0P\xrightarrow[]{g_1f_0}g_1Q \xrightarrow[]{f_1}R\in \CC\]
\end{definition}

Since $\varnothing\in \CC$ is the unique initial object and the distinguished base point and it does not have any non-trivial automorphisms, $g\cdot \varnothing = \varnothing$ for all $g\in G$. Any morphism $\varnothing \to B$ in $\CC_{hG}$ is defined to be the corresponding unique morphism in $\CC$. This ensures the group action is pointed. Under the hypothesis that a collection $$\{P_i\xrightarrow[]{(f_i,g_i)} Q\}_{i\in I}$$ is a cover in $\CC_{hG}$ if $\{f_i:g_i P_i\xrightarrow[]{} Q\}_{i\in I}$ is a cover in $\CC$, $\CC_{hG}$ becomes a weak assembler. 

\begin{lemma}
    $\CC_{hG}$ is a weak assembler.
\end{lemma}

\begin{proof}
It is easy to see that this is a category with covering family. We check the axioms for weak assembler. \textbf{(I)} holds as the initial object $\varnothing$ acts as the distinguished base point in $\CC_{hG}$. To check \textbf{(D)}, consider for any covering family $$\{P_i\xrightarrow[]{(f_i,g_i)} Q\}_{i\in I}$$ of $\CC_{hG}$ and any $i \neq j$, the following commutative diagram in $\CC_{hG}$ and its corresponding diagram in $\CC$ respectively,
% https://q.uiver.app/#q=WzAsOCxbMCwwLCJSIl0sWzEsMCwiUF9pIl0sWzAsMSwiUF9qIl0sWzEsMSwiUSJdLFszLDAsImdfaWhfaVIiXSxbNCwwLCJnX2lQX2kiXSxbMywxLCJnX2pQX2oiXSxbNCwxLCJRIl0sWzAsMSwiKGEsaF9pKSJdLFsxLDMsIihmX2ksZ19pKSJdLFswLDIsIihiLGhfaikiLDJdLFsyLDMsIihmX2osZ19qKSIsMl0sWzQsNSwiZ19paF9pIl0sWzQsNiwiZ19qaF9qIiwyXSxbNiw3LCJmX2oiXSxbNSw3LCJmX2kiXSxbMCwzLCJcXGluIFxcQ0Nfe2h7R319IiwxLHsibGFiZWxfcG9zaXRpb24iOjYwLCJzaG9ydGVuIjp7InNvdXJjZSI6NjB9LCJzdHlsZSI6eyJib2R5Ijp7Im5hbWUiOiJub25lIn0sImhlYWQiOnsibmFtZSI6Im5vbmUifX19XSxbNCw3LCJcXGluIFxcQ0MiLDEseyJzdHlsZSI6eyJib2R5Ijp7Im5hbWUiOiJub25lIn0sImhlYWQiOnsibmFtZSI6Im5vbmUifX19XV0=
\[\begin{tikzcd}
	R & {P_i} && {g_ih_iR} & {g_iP_i} \\
	{P_j} & Q && {g_jP_j} & Q
	\arrow["{(a,h_i)}", from=1-1, to=1-2]
	\arrow["{(b,h_j)}"', from=1-1, to=2-1]
	\arrow["{\in \CC_{h{G}}}"{description, pos=0.6}, draw=none, from=1-1, to=2-2]
	\arrow["{(f_i,g_i)}", from=1-2, to=2-2]
	\arrow["{g_ih_i}", from=1-4, to=1-5]
	\arrow["{g_jh_j}"', from=1-4, to=2-4]
	\arrow["{\in \CC}"{description}, draw=none, from=1-4, to=2-5]
	\arrow["{f_i}", from=1-5, to=2-5]
	\arrow["{(f_j,g_j)}"', from=2-1, to=2-2]
	\arrow["{f_j}", from=2-4, to=2-5]
\end{tikzcd}\]
The diagram commutes in $\CC_{hG}$ iff the corresponding diagram in $\CC$ commutes and $g_ih_i = g_j h_j$. By definition of covers in $\CC_{hG}$ and universal property of pullback, we have a map $g_ih_iR \to g_iP_i \times_{Q} g_jP_j \cong \varnothing$. Because any object mapping into $\varnothing$ must be isomorphic to $\varnothing$ in $\CC$, it follows that $g_ih_iR \cong \varnothing$. Since $G$ acts via automorphisms, $R \cong (g_ih_i)^{-1}\varnothing = \varnothing$ showing that the pullback of covering family exists and is $\varnothing$ in $\CC_{hG}$.

To verify axiom \textbf{(R)} for $\CC_{hG}$, let $$\{P_i \xrightarrow{(a^i, g^i)} Q\}_{i \in I} \text{ and } \{R_j \xrightarrow{(b^j, h^j)} Q\}_{j \in J}$$ be two covering families of $Q$ in $\CC_{hG}$. Then $\{g^i P_i \xrightarrow{a^i} Q\}_{i \in I}$ and $\{h^j R_j \xrightarrow{b^j} Q\}_{j \in J}$ are covering families of $Q$ in $\CC$. Let $S_{ij} = g^i P_i \times_Q h^j R_j$ be the common refinement pieces in $\CC$, equipped with covering maps $$S_{ij} \xrightarrow{u_{ij}} g^i P_i \text{ and } S_{ij} \xrightarrow[]{v_{ij}} h^j R_j.$$ Then $$\{S_{ij} \xrightarrow{(a^i \circ u_{ij}, e)} Q\}_{(i,j) \in I \times J}$$ is the required refinement in $\CC_{hG}$.

To check \textbf{(M)}, suppose $(a,g)\circ (b_1,h_1)= (a,g) \circ (b_2,h_2)$. Expanding this gives $(a\circ gb_1, gh_1) = (a\circ gb_2, gh_2)$. Equality of second coordinate dictates $h_1=h_2$ and equality of first coordinate gives $a\circ gb_1 = a\circ gb_2 \in \CC$. As all morphisms in $\CC$ are monomorphism, $ gb_1 = gb_2$, which yields $b_1= b_2$. Therefore $(b_1,h_1)=(b_2,h_2)$.
\end{proof}

\begin{definition} \label{def: measure}
    A \emph{measure} on a weak $G$-assembler $\CC$ with values in $A \in \bbZ[G]\text{-}\operatorname{Mod}$ is a $G$-equivariant function $\mu: obj (\CC) \rightarrow A$ which is additive on covering families: $\mu(Q)=\sum\mu(P_i)$ for every covering family $\{P_i \rightarrow Q\}_{i \in I}$ in $\CC$. 
\end{definition}

Any such measure $\mu$ on weak $G$-assembler $\CC$ gives rise to an explicit trace map/regulator map (see \cite[Lemma 6.8, Definition 6.11]{bohmann2024trace}) from K-groups $K_{n}(\CC_{hG})$ to group homology $H_n(G;A)$. We recall the construction briefly here. 

\begin{notation}
    We fix a notation for morphisms between covers in $\wW(\CC_{hG})$. A map of covering families $$P=\{P_i\}_{i\in I }\xrightarrow{(\alpha, f,g)} \{Q_j\}_{j\in J}=Q$$ is a set map $\alpha: I \to J$ and for each $i \in I $, a morphism $P_i \xrightarrow{f_i,g_i} Q_{\alpha(i)}$ in $\CC_{hG}$, such that for each $j \in J $, the collection $$\{ P_i \xrightarrow{f_i,g_i} Q_j\}_{i\in \alpha^{-1}(j)}$$ forms a cover in $\CC_{hG}$. Note here that $f$ and $g$ are placeholders for the names of inclusion maps and the group elements. Henceforth unless otherwise mentioned we use $I, J, K$ for indexing $P, Q, R$ respectively. A subscript like $P_i$ would generally mean a object in the $P=\{P_i\}_{i\in I}$ and a superscript like $P^i$ would refer to an object of $\wW(\CC_{hG})$ (or $\gG(\CC_{hG})$ or $\mM(\CC_{hG})$).
\end{notation}

\begin{notation}
    For any map $P\xhookrightarrow{\overline{f}_{P'}} Q$ of partial covers, we henceforth suppress the $P'$ and by the earlier notation focus on $f|_{P}=(\alpha, a, g)$ which track where each piece is moving. 
\end{notation}

\begin{construction}\label{const: regulator map}\cite[Lemma 6.8]{bohmann2024trace}
    We recall the construction of the regulator map $K(\CC_{hG})\to H(G;A)$. Let $B_{\bullet}^{\otimes}$ be the two sided bar complex with respect to $\otimes$. Then the simplicial map
    $$\operatorname{T_\wW}: B\wW(\CC_{hG}) \xrightarrow[]{} B(\bbZ, \bbZ[G], A)$$
    is defined by sending a $q$-tuple of covers $P^0 \xleftarrow{\beta^1,\,b^1,\,h^1} P^1 \xleftarrow{\beta^2,\,b^2,\,h^2} \cdots \xleftarrow{\beta^q,\,b^q,\,h^q} P^q$ to 

    $$\sum_{i\in I_q} h^1_{i_1} \otimes h^2_{i_2} \otimes \cdots \otimes h^{q-1}_{i_{q-1}} \otimes h^q_{i_q} \otimes \mu (P^q_i) $$
    where $i_k \in I_k$ is defined by first taking $i_q= i$ and iteratively defining $i_{k-1} = \beta^k(i_k)$ for $1\leq k \leq q$. Upon group completion we get the regulator map
    $$\operatorname{tr}: K(\CC_{hG}) := \Omega B |B(\CC_{hG})| \to (HA)_{hG}$$
    or equivalently the map on homotopy groups
    $$\operatorname{tr}: K_n(\CC_{hG}) \to H_n(G;A)$$
\end{construction}

At simplicial level, the regulator map tracks the sequence of symmetries from $G$ that are applied to a piece $P^q_i$ of $P^q$ as it moves through $P^r$s along with its measure.

\begin{construction}\label{const: T_M(C)}
    We construct a simplicial map $$ T_{\mM} : HH(\bbZ\mM(\CC_{hG})) \to B(\bbZ, \bbZ[G], A).$$ 
    \begin{itemize}
    \item First note that the set of $q$-simplices $HH_q (\bbZ \mM(\CC_{hG}))$ is
    \[
    \bigoplus_{P^0,P^1,\ldots,P^q} 
    \Hom_{\mathbb{Z}\mathcal{M}(\mathcal{C}_{hG})}(P^0,P^q) 
    \otimes_{\mathbb{Z}}
    \Hom_{\mathbb{Z}\mathcal{M}(\mathcal{C}_{hG})}(P^1,P^0)
    \otimes_{\mathbb{Z}}
    \cdots
    \otimes_{\mathbb{Z}}
    \Hom_{\mathbb{Z}\mathcal{M}(\mathcal{C}_{hG})}(P^{q},P^{q-1}),
    \] 
    and is generated by $q+1$ composable morphisms (simple tensors) $Z^0\otimes Z^1\otimes\cdots\otimes Z^{q-1}\otimes Z^q$ in $\mathcal{M}(\mathcal{C}_{hG})$,
    \begin{equation}\label{eq: q-simplex in MC}
        P^q
        \xleftarrow{\,Z^0\,}
        P^0
        \xleftarrow{\,Z^1\,}
        P^1
        \xleftarrow{\,Z^2\,}
        \cdots
        \xleftarrow{\,Z^{q-1}\,}
        P^{q-1}
        \xleftarrow{\,Z^q\,}
        P^q.
    \end{equation}
    where $Z^m$ is the equivalence class of zigzags 
    \[
    P^q
    \xhookleftarrow{\beta^0,b^0,h^0}
    Q^0
    \xhookrightarrow{\alpha^0,a^0,g^0}
    P^0, \quad \text{for $m=0$}.
    \]
    \[
    P^{m-1}
    \xhookleftarrow{\beta^m,b^m,h^m}
    Q^m
    \xhookrightarrow{\alpha^m,a^m,g^m}
    P^m,\quad \text{for $1\le m\le q$}
    \]
    Let $R$ be the common refinement of the zigzags of partial cover of the cycle \eqref{eq: q-simplex in MC}. 
    \item We define $T_\mM$ on these simple tensors as follows and extend linearly over $\bbZ$ to the entire Hochschild complex,
    \begin{equation}\label{eq: expanded q-simplex in MC}
    \begin{tikzcd}
    	& {P^1} && {P^2} & \\
    	{P^0} & {\textcolor{blue}{Q^1}} & {\textcolor{blue}{Q^2}} & {\textcolor{blue}{Q^3}} & \vdots \\
    	& {\textcolor{blue}{Q^0}} & {\textcolor{red}{R}} \\
    	{P^q} & {\textcolor{blue}{Q^q}} && {\textcolor{blue}{Q^m}} & {P^{m-1}} \\
    	& \cdots && {P^{m}}
    	\arrow[color={rgb,255:red,68;green,116;blue,238}, hook, from=2-2, to=1-2]
    	\arrow[color={rgb,255:red,68;green,116;blue,238}, hook', from=2-2, to=2-1]
    	\arrow[color={rgb,255:red,68;green,116;blue,238}, hook', from=2-3, to=1-2]
    	\arrow[color={rgb,255:red,68;green,116;blue,238}, hook, from=2-3, to=1-4]
    	\arrow[color={rgb,255:red,68;green,116;blue,238}, hook', from=2-4, to=1-4]
    	\arrow[color={rgb,255:red,68;green,116;blue,238}, hook, from=2-4, to=2-5]
    	\arrow[color={rgb,255:red,68;green,116;blue,238}, hook, from=3-2, to=2-1]
    	\arrow[color={rgb,255:red,68;green,116;blue,238}, hook', from=3-2, to=4-1]
    	\arrow[color={rgb,255:red,214;green,92;blue,92}, hook', from=3-3, to=2-2]
    	\arrow[color={rgb,255:red,214;green,92;blue,92}, hook', from=3-3, to=2-3]
    	\arrow[color={rgb,255:red,214;green,92;blue,92}, hook', from=3-3, to=2-4]
    	\arrow[color={rgb,255:red,214;green,92;blue,92}, hook', from=3-3, to=3-2]
    	\arrow[color={rgb,255:red,214;green,92;blue,92}, hook', from=3-3, to=4-2]
    	\arrow["{(\gamma^m, c^m,r^m)}", color={rgb,255:red,214;green,92;blue,92}, hook', from=3-3, to=4-4]
    	\arrow[color={rgb,255:red,68;green,116;blue,238}, hook, from=4-2, to=4-1]
    	\arrow[color={rgb,255:red,68;green,116;blue,238}, hook', from=4-2, to=5-2]
    	\arrow["{(\beta^m,b^m,h^m)}", color={rgb,255:red,68;green,116;blue,238}, hook', from=4-4, to=4-5]
    	\arrow["{(\alpha^m,a^m,g^m)}"', color={rgb,255:red,68;green,116;blue,238}, hook, from=4-4, to=5-4]
    \end{tikzcd} 
    \end{equation}

    \vspace{1em}

    \[
    T_\mM: Z^0\otimes Z^1\otimes\cdots\otimes Z^{q-1}\otimes Z^q \longmapsto 
    \sum_{k\in K}
    w^1_k
    \otimes
    \cdots
    \otimes
    w^q_k
    \otimes
    \mu\!\left((g^q_{\gamma^q(k)} r^q_k) R_k\right),
    \]
    where $w^m_k= h^m_{\gamma^m(k)}(g^m_{\gamma^m(k)})^{-1} $. Each piece $R_k$ in $R$ lands in the piece $Q^m_{\gamma^m(k)}$ of $Q^m$ and this piece $Q^m_{\gamma^m(k)}$ moves from $P^m$ to $P^{m-1}$ by first being pulled by $g^m_{\gamma^m(k)}$ from $P^m$ and then being pushed by $h^m_{\gamma^m(k)}$ to $P^{m-1}$. $T_\mM$ tracks these symmetries analogous to $T_\wW$. One needs to be careful regarding which measure to use. We first push the piece $R_k$ to $P^q$ via the appropriate group element $g^q_{\gamma^q(k)}r^q_k$ and then take its measure. This is necessary to ensure that $T_\mM$ remains well defined for any $A\in \bbZ[G]\text{-}\operatorname{Mod}$. 
    \end{itemize}
    We show below (\cref{prop: well-defined}, \cref{prop: simplicial}) that this construction is well-defined and in fact simplicial.
\end{construction}

\begin{remark}
    The common refinement of a cycle  of zigzags as in \eqref{eq: expanded q-simplex in MC} can be found by first opening up the cycle at say $Z^0$ and taking a common refinement (composing in $\mM_{\CC_{hG}}$) to get say $P^0 \xhookleftarrow{} R^1 \xhookrightarrow{} P^q$, and then composing with the unused zigzag $Z^0=P^0 \xhookleftarrow{} Q^0\xhookrightarrow{} P^q $ to get $R$.
\end{remark}

\begin{remark}\label{rem: TM depends on cyclicity}
    In the definition of $T_\mM$ it might appear that as $$w^0_k=h^0_{\gamma^0(k)}(g^0_{\gamma^0(k)})^{-1}$$ does not appear, one could define $T_\mM$ by opening up the cycle $Z^0 \otimes Z^1 \otimes \cdots \otimes Z^q$, throwing away $Z^0$ and taking a refinement (or composition in $\mM(\CC_{hG})$) $\widetilde{R}$ of the remaining pieces which would produce a similar looking formula as above. But the main difference is this refinement would be in general strictly larger i.e. $R \xhookrightarrow{} \widetilde{R}$ and hence would have more summands. In fact, this map would not be a simplicial map because it would not commute with $d_0$.
\end{remark}

\begin{proposition}\label{prop: well-defined}
    The map $T_\mM$ is a well defined. 
\end{proposition}

\begin{proof}
Verification of well definedness of $T_\mM$ requires us to show that it is independent of choice of representatives of zigzags $Z^m$ and is compatible with the additive relations imposed by the $\bbZ$-enrichment in $\bbZ\mM(\CC_{hG})$ (as in \cref{def: category of scissors correspondences}).

Suppose we choose an equivalent representative of $Z^m$ say
\[
P^{m-1}
\xhookleftarrow{\widetilde{\beta}^m,\; \widetilde{b}^m,\; \widetilde{h}^m}
\widetilde{Q}^m
\xhookrightarrow{\widetilde{\alpha}^m,\; \widetilde{a}^m,\; \widetilde{g}^m}
P^m.
\]

There exists a common refinement $\widehat{Q}^m$ which covers $Q^m$ and $\widetilde{Q}^m$ such that the following diagram commutes
% https://q.uiver.app/#q=WzAsNixbMiwxLCJQXm0iXSxbMSwxLCJcXHdpZGVoYXR7UX1ebSJdLFsxLDAsIlFebSJdLFsxLDIsIlxcd2lkZXRpbGRle1F9Xm0iXSxbMCwxLCJQXnttLTF9Il0sWzQsMSwiUiJdLFsyLDAsIlxcYWxwaGFebSwgYV5tLCBnXm0gIiwwLHsic3R5bGUiOnsidGFpbCI6eyJuYW1lIjoiaG9vayIsInNpZGUiOiJib3R0b20ifX19XSxbMywwLCJcXHdpZGV0aWxkZXtcXGFscGhhfV5tLCBcXHdpZGV0aWxkZXthfV5tLCBcXHdpZGV0aWxkZXtnfV5tIiwyLHsic3R5bGUiOnsidGFpbCI6eyJuYW1lIjoiaG9vayIsInNpZGUiOiJib3R0b20ifX19XSxbMSwyLCJcXGRlbHRhXm0sIGRebSwgbF5tICIsMV0sWzEsMywiXFx3aWRldGlsZGV7XFxkZWx0YX1ebSwgXFx3aWRldGlsZGV7ZH1ebSwgXFx3aWRldGlsZGV7bH1ebSIsMV0sWzIsNCwiXFxiZXRhXm0sIGJebSwgaF5tIiwyLHsic3R5bGUiOnsidGFpbCI6eyJuYW1lIjoiaG9vayIsInNpZGUiOiJib3R0b20ifX19XSxbMyw0LCJcXHdpZGV0aWxkZXtcXGJldGF9Xm0sIFxcd2lkZXRpbGRle2J9Xm0sIFxcd2lkZXRpbGRle2h9Xm0iLDAseyJzdHlsZSI6eyJ0YWlsIjp7Im5hbWUiOiJob29rIiwic2lkZSI6InRvcCJ9fX1dLFs1LDIsIlxcZ2FtbWFebSwgY15tLCByXm0iLDIseyJjdXJ2ZSI6Mywic3R5bGUiOnsidGFpbCI6eyJuYW1lIjoiaG9vayIsInNpZGUiOiJib3R0b20ifX19XSxbNSwzLCJcXHdpZGV0aWxkZXtcXGdhbW1hfV5tLCBcXHdpZGV0aWxkZXtjfV5tLCBcXHdpZGV0aWxkZXtyfV5tIiwwLHsiY3VydmUiOi0zLCJzdHlsZSI6eyJ0YWlsIjp7Im5hbWUiOiJob29rIiwic2lkZSI6ImJvdHRvbSJ9fX1dLFs1LDEsIlxcd2lkZWhhdHtcXGdhbW1hfV5xLCBcXHdpZGVoYXR7Y31ecSwgXFx3aWRlaGF0e3J9XnEiLDIseyJjdXJ2ZSI6Miwic3R5bGUiOnsidGFpbCI6eyJuYW1lIjoiaG9vayIsInNpZGUiOiJib3R0b20ifX19XV0=
\begin{equation}\label{diagram: well definedness of TM}
\begin{tikzcd}
	& {Q^m} &&& \\
	{P^{m-1}} & {\widehat{Q}^m} & {P^m} && R \\
	& {\widetilde{Q}^m}
	\arrow["{\beta^m, b^m, h^m}"', hook', from=1-2, to=2-1]
	\arrow["{\alpha^m, a^m, g^m }", hook', from=1-2, to=2-3]
	\arrow["{\delta^m, d^m, l^m }"{description}, from=2-2, to=1-2]
	\arrow["{\widetilde{\delta}^m, \widetilde{d}^m, \widetilde{l}^m}"{description}, from=2-2, to=3-2]
	\arrow["{\gamma^m, c^m, r^m}"', curve={height=18pt}, hook', from=2-5, to=1-2]
	\arrow["{\widehat{\gamma}^q, \widehat{c}^q, \widehat{r}^q}"', curve={height=12pt}, hook', from=2-5, to=2-2]
	\arrow["{\widetilde{\gamma}^m, \widetilde{c}^m, \widetilde{r}^m}", curve={height=-18pt}, hook', from=2-5, to=3-2]
	\arrow["{\widetilde{\beta}^m, \widetilde{b}^m, \widetilde{h}^m}", hook, from=3-2, to=2-1]
	\arrow["{\widetilde{\alpha}^m, \widetilde{a}^m, \widetilde{g}^m}"', hook', from=3-2, to=2-3]
\end{tikzcd}
\end{equation}
where $R$ remains (maybe after taking a finer cover) the overarching common refinement, and partially covers $Q^m$, $\widetilde{Q}^m$ and now $\widehat{Q}_m$. A piece $R_k$ of $R$ is pulled and pushed in this $m$-th level by two a priori different pair of group element: $h^m_{\gamma^q(k)}(g^m_{\gamma^q(k)})^{-1}$ and $\widetilde{h}^m_{\widetilde{\gamma}^q(k)}(\widetilde{g}^m_{\widetilde{\gamma}^q(k)})^{-1}$ for $w^m_k$. We can rewrite them using the commutative diagram above as
$$h^m_{\gamma^q(k)}(g^m_{\gamma^q(k)})^{-1} = h^m_{\delta^m\circ\widehat{\gamma}^q(k)}l^m_{\widehat{\gamma}^q(k)} (g^m_{\delta^m\circ\widehat{\gamma}^q(k)}l^m_{\widehat{\gamma}^q(k)})^{-1}= (h^m\circ l^m)_{\widehat{\gamma}(k)}(g^m \circ l^m)_{\widetilde{\gamma}(k)}^{-1}$$
and 
$$ \widetilde{h}^m_{\widetilde{\gamma}^q(k)}(\widetilde{g}^m_{\widetilde{\gamma}^q(k)})^{-1} = (\widetilde{h}^m \circ \widetilde{l}^m)_{\widehat{\gamma}(k)} (\widetilde{g}^m \circ \widetilde{l}^m)_{\widehat{\gamma}(k)}^{-1}$$
which are clearly equal as \eqref{diagram: well definedness of TM} commutes. In a similar fashion, it is easy to check that $$g^q_{\gamma^q(k)}r^q_k = \widetilde{g}^q_{\widetilde{\gamma}^q(k)}\widetilde{r}^q_k$$ and hence $$\mu((g^q_{\gamma^q(k)}r^q_k)R_k) = \mu((\widetilde{g}^q_{\widetilde{\gamma}^q(k)}\widetilde{r}^q_k)R_k).$$ As $\mu$ is additive for covering families the expression for $T_M$ remains same for a finer cover $R$, and as both the group trajectories and the measure remain same irrespective of the choice of representative of $Z^m$, $T_\mM$ is well defined.

We now check comparability with the $\bbZ$-enrichment. If any of the $Z^m$'s is the zero map $$P^{m-1}\xhookleftarrow{} \varnothing \xhookrightarrow{} P^{m},$$ then the common refinement $R=\varnothing$ and the sum in the expression of 
$$T_M(Z^0\otimes Z^1\otimes\cdots\otimes Z^{q-1}\otimes Z^q)$$ is over an empty set making the sum $0$, respecting the zero relation. Similarly if a zigzag $Z^m$ satisfies $Z^m = Z^m_1+Z^m_2$, the common refinement $R$ would split into disjoint components tracing through $Q^m_1$ and $Q^m_2$. Therefore we have
$$T_\mM(Z^0\otimes Z^1\otimes\cdots\otimes Z^{m}\otimes \cdots\otimes Z^q)= T_\mM(Z^0\otimes Z^1\otimes \cdots\otimes Z_1^m \otimes\cdots \otimes Z^q)+T_\mM(Z^0\otimes Z^1\otimes \cdots\otimes Z_2^m \otimes\cdots \otimes Z^q)$$ 
showing that $T\mM$ is compatible with both the relations.   
\end{proof}

\begin{proposition}\label{prop: simplicial}
    $T_\mM$ is a simplicial map.
\end{proposition}

\begin{proof}
First let us recall the face and degeneracy maps of the Hochschild complex as well as the two sided bar construction. 

The face and degeneracy maps at $q$-th level  $N^\otimes_q(\bbZ, \bbZ[G], A)$ are
\begin{align*}
    d_0(g_1\otimes g_2 \otimes \cdots  \otimes g_q \otimes a) & = g_2\otimes \cdots \otimes g_q \otimes a\\
    d_m(g_1\otimes \cdots \otimes g_m\otimes g_{m+1}\otimes \cdots\otimes g_q\otimes a) & = g_1\otimes\cdots \otimes (g_mg_{m+1}) \otimes \cdots \otimes g_q \otimes a\\
    &\quad \quad \text{for $0< m < q$}\\
    d_q(g_1\otimes \cdots \otimes g_{q-1} \otimes g_q \otimes a) & = g_1 \otimes \cdots \otimes g_{q-1} \otimes (g_q\cdot a)\\
    s_{m}(g_1\otimes \cdots \otimes g_q \otimes a) & = g_1\otimes \cdots g_m \otimes e \otimes g_{m+1} \otimes \cdots \otimes g_q \otimes a\\
    & \quad \quad \text{ for $0\leq m \leq q$}
\end{align*}

Here $d_0$ absorbs the $g_1$ owing to the right module action on $\bbZ$ being trivial. The face and degeneracy map on a simple tensor in $HH_{q}(\bbZ\mM(\CC_{hG}))$ is given by 
\begin{align*}
    d_m(Z^0\otimes \cdots \otimes Z^m \otimes Z^{m+1} \otimes \cdots \otimes Z^q) & = Z^0 \otimes \cdots \otimes (Z^m\circ Z^{m+1}) \otimes \cdots \otimes Z^q\\
    & \quad \quad \text{for $0\leq m < q$ }\\
    d_q(Z^0 \otimes Z^1 \otimes \cdots \otimes Z^{q-1} \otimes Z^q) & = (Z^q\circ Z^0)\otimes Z^1 \otimes \cdots \otimes Z^{q-1}\\
    s_m(Z^0 \otimes \cdots \otimes Z^q) & = Z^0 \otimes \cdots Z^m \otimes \operatorname{id}_{P^m} \otimes Z^{m+1} \otimes \cdots \otimes Z^q\\
    & \quad \quad \text{for $0\leq m \leq q$}
\end{align*}
It is easy to check that $T_\mM$ commutes with the degeneracies $s_m$ for $0\leq m \leq q$ since introducing $\operatorname{id}_{P^m}$ amount to introducing an extra group element $e$ by $T_\mM$. We now check that it commutes with the internal face maps $d_m$ for $0 < m < q$. Let $Z^0 \otimes \cdots \otimes Z^q$ be a generating $q$-simplex. The face map $d_m$ algebraically composes the zig-zags $Z^m$ and $Z^{m+1}$ by taking their common refinement over $P^m$, as illustrated in the following commutative diagram in $\mM(\CC_{hG})$,
% https://q.uiver.app/#q=WzAsNyxbMCwwLCJQXnttLTF9Il0sWzIsMCwiUF57bX0iXSxbNCwwLCJQXnttKzF9Il0sWzEsMSwiUV5tIl0sWzMsMSwiUV57bSsxfSJdLFsyLDIsIlEiXSxbMiwzLCJSIl0sWzYsNSwiKFxcZ2FtbWEsIGMsIHIpIiwxLHsic3R5bGUiOnsidGFpbCI6eyJuYW1lIjoiaG9vayIsInNpZGUiOiJib3R0b20ifX19XSxbNCwxLCIoXFxiZXRhXnttKzF9LGJee20rMX0saF57bSsxfSkiLDEseyJzdHlsZSI6eyJ0YWlsIjp7Im5hbWUiOiJob29rIiwic2lkZSI6ImJvdHRvbSJ9fX1dLFs0LDIsIihcXGFscGhhXnttKzF9LGFee20rMX0sZ157bSsxfSkiLDIseyJzdHlsZSI6eyJ0YWlsIjp7Im5hbWUiOiJob29rIiwic2lkZSI6InRvcCJ9fX1dLFszLDEsIihcXGFscGhhXm0sIGFebSxnXm0pIiwwLHsic3R5bGUiOnsidGFpbCI6eyJuYW1lIjoiaG9vayIsInNpZGUiOiJ0b3AifX19XSxbMywwLCIoXFxiZXRhXm0sYl5tLGhebSkiLDAseyJzdHlsZSI6eyJ0YWlsIjp7Im5hbWUiOiJob29rIiwic2lkZSI6ImJvdHRvbSJ9fX1dLFs2LDMsIiIsMix7ImN1cnZlIjotMywic3R5bGUiOnsidGFpbCI6eyJuYW1lIjoiaG9vayIsInNpZGUiOiJib3R0b20ifSwiYm9keSI6eyJuYW1lIjoiZG90dGVkIn19fV0sWzYsNCwiIiwwLHsiY3VydmUiOjMsInN0eWxlIjp7InRhaWwiOnsibmFtZSI6Imhvb2siLCJzaWRlIjoidG9wIn0sImJvZHkiOnsibmFtZSI6ImRvdHRlZCJ9fX1dLFs1LDQsIihcXGFscGhhLCBhLCBnKSIsMix7InN0eWxlIjp7InRhaWwiOnsibmFtZSI6Imhvb2siLCJzaWRlIjoidG9wIn19fV0sWzUsMywiKFxcYmV0YSwgYiwgaCkiLDAseyJzdHlsZSI6eyJ0YWlsIjp7Im5hbWUiOiJob29rIiwic2lkZSI6ImJvdHRvbSJ9fX1dXQ==
\[\begin{tikzcd}
	{P^{m-1}} && {P^{m}} && {P^{m+1}} \\
	& {Q^m} && {Q^{m+1}} \\
	&& Q \\
	&& R
	\arrow["{(\beta^m,b^m,h^m)}", hook', from=2-2, to=1-1]
	\arrow["{(\alpha^m, a^m,g^m)}", hook, from=2-2, to=1-3]
	\arrow["{(\beta^{m+1},b^{m+1},h^{m+1})}"{description}, hook', from=2-4, to=1-3]
	\arrow["{(\alpha^{m+1},a^{m+1},g^{m+1})}"', hook, from=2-4, to=1-5]
	\arrow["{(\beta, b, h)}", hook', from=3-3, to=2-2]
	\arrow["{(\alpha, a, g)}"', hook, from=3-3, to=2-4]
	\arrow[curve={height=-18pt}, dotted, hook', from=4-3, to=2-2]
	\arrow[curve={height=18pt}, dotted, hook, from=4-3, to=2-4]
	\arrow["{(\gamma, c, r)}"{description}, hook', from=4-3, to=3-3]
\end{tikzcd}\]
Because refinement is associative, $R$ remains the overarching refinement for the entire composed sequence of zig-zags. Applying $T_\mM \circ d_m$ to the $q$-simplex evaluates the sequence with $Z^m \circ Z^{m+1}$ merged. For a given piece $k \in K$, the new accrued group element at this composed step is formed by the outer legs of the diamond $Q$. Thus, we have,
\begin{align*}
    & T_\mM\circ d_m (Z^0\otimes \cdots \otimes Z^m \otimes Z^{m+1} \otimes \cdots \otimes Z^q) \\
    =& T_\mM(Z^0 \otimes \cdots \otimes (Z^m\circ Z^{m+1})\otimes \cdots \otimes Z^q)\\
    =& \sum_{k\in K}
    w^1_k
    \otimes
    \cdots
    \otimes 
    w^{m-1}_k \otimes w^m\otimes w^{m+2}_k
    \otimes
    \cdots
    \otimes
    w^q_k\otimes 
    \mu\!\left((g^q_{\gamma^q(k)} \circ r_k) R_k\right)
\end{align*}
Here $w=\Big( h^{m}_{\beta\circ \gamma(k)}h_{\gamma(k)} (g^{m+1}_{\alpha\circ \gamma (k)}g_{\gamma(k)})^{-1} \Big)$. On the other hand, in the target two-sided bar complex, $d_m \circ T_\mM$ acts by multiplying the $m$-th and $(m+1)$-th group elements in the tensor product. For a fixed $k$, this product is
\[ w^m_k w^{m+1}_k= \Big( h^m_{\gamma^m(k)}(g^m_{\gamma^m(k)})^{-1} \Big) \Big( h^{m+1}_{\gamma^{m+1}(k)}(g^{m+1}_{\gamma^{m+1}(k)})^{-1} \Big). \]
To see that these two expressions $w$ and $w^m_k w^{m+1}_k$ are identical, we use the commutativity of the refinement diamond over $P^m$. The left and right paths from $Q$ to $P^m$ must accrue the same group elements, giving the relation
\[ g^{m}_{\beta\circ \gamma (k)} h_{\gamma(k)} = h^{m+1}_{\alpha\circ \gamma(k)} g_{\gamma(k)}. \]
Rearranging this by right-multiplying by $g_{\gamma(k)}^{-1}$ and left-multiplying by $(g^{m}_{\beta\circ \gamma (k)})^{-1}$ yields
\[ h_{\gamma(k)} g_{\gamma(k)}^{-1} = (g^{m}_{\beta\circ \gamma (k)})^{-1}  h^{m+1}_{\alpha\circ \gamma(k)}. \]
Substituting this identity into the middle factor of our $T_\mM \circ d_m$ expression gives:
\begin{align*}
    h^{m}_{\beta\circ \gamma(k)}h_{\gamma(k)} (g^{m+1}_{\alpha\circ \gamma (k)}g_{\gamma(k)})^{-1}
    &= h^{m}_{\beta\circ \gamma(k)} \Big[ h_{\gamma(k)} g_{\gamma(k)}^{-1} \Big] (g^{m+1}_{\alpha\circ \gamma (k)})^{-1}\\
    &= h^{m}_{\beta\circ \gamma(k)} \Big[ (g^{m}_{\beta\circ \gamma (k)})^{-1}  h^{m+1}_{\alpha\circ \gamma(k)} \Big] (g^{m+1}_{\alpha\circ \gamma (k)})^{-1}\\
    &= \Big( h^{m}_{\gamma^{m}(k)}(g^{m}_{\gamma^{m}(k)})^{-1} \Big) \Big( h^{m+1}_{\gamma^{m+1}(k)}(g^{m+1}_{\gamma^{m+1}(k)})^{-1} \Big)
\end{align*}
Hence, $T_\mM \circ d_m = d_m \circ T_\mM$ for $0< m < q$. %%%%%

For $d_0$, the Hochschild face map composes the $0$-th and $1$-st zig-zags, $(Z^0 \circ Z^1) \otimes Z^2 \otimes \cdots \otimes Z^q$. The common refinement of the new $q-1$ cycle $(Z^0 \circ Z^1) \otimes Z^2 \otimes \cdots \otimes Z^q)$ is same as that of $(Z^0\otimes Z^1 \otimes \cdots \otimes Z^q)$. So, 
\begin{align*}
    T_\mM\circ d_0 (Z^0\otimes Z^1 \otimes \cdots \otimes Z^q)
    = \sum_{k\in K}
    w^2_k
    \otimes
    \cdots
    \otimes
    w^q_k
    \otimes
    \mu\!\left((g^q_{\gamma^q(k)} r^q_k) R_k\right).
\end{align*}
On the other hand, $d_0$ in the two-sided bar complex absorbs the first group element due to the trivial right module action on $\bbZ$. Thus, it drops $w^1_k = h^1_{\gamma^1(k)}(g^1_{\gamma^1(k)})^{-1}$ entirely,
\begin{align*}
    & d_0 \circ T_\mM (Z^0\otimes Z^1 \otimes \cdots \otimes Z^q) \\
    =& d_0 \left( \sum_{k\in K} w^1_k \otimes w^2_k \otimes \cdots \otimes w^q_k \otimes \mu\!\left((g^q_{\gamma^q(k)} r^q_k) R_k\right) \right) \\
    =& \sum_{k\in K} w^2_k \otimes \cdots \otimes w^q_k \otimes \mu\!\left((g^q_{\gamma^q(k)} r^q_k) R_k\right)
\end{align*}
Because the overarching refinement $R$ remains same, the two expressions are identical, and $T_\mM \circ d_0 = d_0 \circ T_\mM$.

Finally, we check the cyclic face map $d_q$. In $HH_q(\bbZ\mM(\CC_{hG}))$, $d_q$ wraps around and composes $Z^q$ with $Z^0$. The common refinement is still $R$. So,
\begin{align*}
    & T_\mM\circ d_q (Z^0\otimes Z^1 \otimes \cdots \otimes Z^q) \\
    =& T_\mM((Z^q \circ Z^0) \otimes Z^1 \otimes \cdots \otimes Z^{q-1})\\
    =& \sum_{k\in K}
    w^1_k
    \otimes
    \cdots
    \otimes
    w^{q-1}_k
    \otimes
    \mu\!\left((g^{q-1}_{\gamma^{q-1}(k)} r^{q-1}_k) R_k\right)
\end{align*}
In the two-sided bar complex, $d_q$ acts by letting the last group element act on the module $A$. Thus,
\begin{align*}
    & d_q \circ T_\mM (Z^0\otimes Z^1 \otimes \cdots \otimes Z^q) \\
    =& \sum_{k\in K} w^1_k \otimes \cdots \otimes w^{q-1}_k \otimes \Big( w^q_k \cdot \mu\!\left((g^q_{\gamma^q(k)} r^q_k) R_k\right) \Big)
\end{align*}
To see that these are equal, recall that $w^q_k = h^q_{\gamma^q(k)}(g^q_{\gamma^q(k)})^{-1}$. Because the measure $\mu$ is $G$-equivariant, we can absorb the action of $w^q_k$ inside the measure,
\begin{align*}
    w^q_k \cdot \mu\!\left((g^q_{\gamma^q(k)} r^q_k) R_k\right) 
    &= \mu\!\left( h^q_{\gamma^q(k)}(g^q_{\gamma^q(k)})^{-1} (g^q_{\gamma^q(k)} r^q_k) R_k \right) \\
    &= \mu\!\left( (h^q_{\gamma^q(k)} r^q_k) R_k \right)
\end{align*}
These two measures are equal because the following diagram commutes and hence the group acting on $R_k$ are the same
% https://q.uiver.app/#q=WzAsNixbMCwwLCJcXHRleHRjb2xvcntyZWR9e1J9Il0sWzIsMCwiXFx0ZXh0Y29sb3J7Ymx1ZX17UV57cS0xfX0iXSxbMCwyLCJcXHRleHRjb2xvcntibHVlfXtRXntxfX0iXSxbMiwyLCJQXntxLTF9Il0sWzMsMSwiXFx0ZXh0Y29sb3J7cmVkfXtSX2t9Il0sWzUsMSwiUF57cS0xfSJdLFswLDIsIihcXGdhbW1hXnEsIGNecSxyXnEpIiwyLHsiY29sb3VyIjpbMCw2MCw2MF0sInN0eWxlIjp7InRhaWwiOnsibmFtZSI6Imhvb2siLCJzaWRlIjoidG9wIn19fSxbMCw2MCw2MCwxXV0sWzIsMywiKFxcYmV0YV5xLGJecSxoXnEpIiwyLHsiY29sb3VyIjpbMjQ1LDg3LDYwXSwic3R5bGUiOnsidGFpbCI6eyJuYW1lIjoiaG9vayIsInNpZGUiOiJ0b3AifX19LFsyNDUsODcsNjAsMV1dLFswLDEsIihcXGdhbW1hXntxLTF9LCBjXntxLTF9LHJee3EtMX0pIiwwLHsiY29sb3VyIjpbMCw2MCw2MF0sInN0eWxlIjp7InRhaWwiOnsibmFtZSI6Imhvb2siLCJzaWRlIjoidG9wIn19fSxbMCw2MCw2MCwxXV0sWzEsMywiKFxcYWxwaGFee3EtMX0sYV57cS0xfSxnXntxLTF9KSIsMix7ImNvbG91ciI6WzI0NSw4Nyw2MF0sInN0eWxlIjp7InRhaWwiOnsibmFtZSI6Imhvb2siLCJzaWRlIjoidG9wIn19fSxbMjQ1LDg3LDYwLDFdXSxbNCw1LCIoYV57cS0xfV97XFxnYW1tYV57cS0xfShrKX1cXGNpcmMgZ157cS0xfV97XFxnYW1tYV57cS0xfShrKX1hXntxLTF9X2ssZ157cS0xfV97XFxnYW1tYV57cS0xfShrKX0gcl57cS0xfV9rKSIsMCx7ImN1cnZlIjotM31dLFs0LDUsIihiXnFfe1xcZ2FtbWFecShrKX1cXGNpcmMgaF5xX3tcXGdhbW1hXnEoayl9Y15xX3trfSxoXnFfe1xcZ2FtbWFecShrKX0gcl5xX2spIiwyLHsiY3VydmUiOjN9XV0=
\[\begin{tikzcd}
	{\textcolor{red}{R}} && {\textcolor{blue}{Q^{q-1}}} &&& \\
	&&& {\textcolor{red}{R_k}} && {P^{q-1}} \\
	{\textcolor{blue}{Q^{q}}} && {P^{q-1}}
	\arrow["{(\gamma^{q-1}, c^{q-1},r^{q-1})}", color={rgb,255:red,214;green,92;blue,92}, hook, from=1-1, to=1-3]
	\arrow["{(\gamma^q, c^q,r^q)}"', color={rgb,255:red,214;green,92;blue,92}, hook, from=1-1, to=3-1]
	\arrow["{(\alpha^{q-1},a^{q-1},g^{q-1})}"', color={rgb,255:red,79;green,64;blue,242}, hook, from=1-3, to=3-3]
	\arrow["{(a^{q-1}_{\gamma^{q-1}(k)}\circ g^{q-1}_{\gamma^{q-1}(k)}a^{q-1}_k,g^{q-1}_{\gamma^{q-1}(k)} r^{q-1}_k)}", curve={height=-18pt}, from=2-4, to=2-6]
	\arrow["{(b^q_{\gamma^q(k)}\circ h^q_{\gamma^q(k)}c^q_{k},h^q_{\gamma^q(k)} r^q_k)}"', curve={height=18pt}, from=2-4, to=2-6]
	\arrow["{(\beta^q,b^q,h^q)}"', color={rgb,255:red,79;green,64;blue,242}, hook, from=3-1, to=3-3]
\end{tikzcd}.\]    
\end{proof}

Thus we have a simplicial map from the Hochschild complex associated to $\CC$ to the group homology complex associated to $\CC$. 
\begin{remark}
  As pointed out in \cref{rem: TM depends on cyclicity}, this map does depend on the cyclic nature of Hochschild complex. The composite $T_\mM\circ DTr$ is thus a trace invariant of $\CC$, giving us trace invariants for $\CC$ associated to each measure $A$ that exists on $\CC$.   
\end{remark}
Now, $T_\mM$ induces a homomorphism
\begin{equation}\label{eqn: tr_M}
    \operatorname{tr}_\mM: HH_n(\bbZ\mM(\CC_{hG})) \to H_n(G;A).
\end{equation}
We now show the factoring of the regulator map via this $\operatorname{tr}_\mM$:

\begin{theorem}\label{thm: regulator factors through Dennis trace}
    The regulator map of \cite{bohmann2024trace} factors through the Dennis trace map, i.e. we have a commutative diagram
    \[\begin{tikzcd}
	{K_n(\CC_{hG})} && {HH_n(\bbZ\mM(\CC_{hG}))} \\
	& {H_n(G; A)}
	\arrow["Dtr_n", from=1-1, to=1-3]
	\arrow["{\operatorname{tr}}"', from=1-1, to=2-2]
	\arrow["{\operatorname{tr}_\mM}", from=1-3, to=2-2]
\end{tikzcd}\]
\end{theorem}

\begin{proof}
Let $P^0 \xleftarrow{\beta^1,\,b^1,\,h^1} P^1 \xleftarrow{\beta^2,\,b^2,\,h^2} \cdots \xleftarrow{\beta^q,\,b^q,\,h^q} P^q$ be a $q$-simplex in $B\wW(\CC_{hG})$. At simplicial level, $Dtr$ maps this to the following simplex of $\gG(\CC_{hG})$,
% https://q.uiver.app/#q=WzAsOCxbNiwwLCJQXnEiXSxbMiwwLCJQXjEiXSxbMCwwLCJQXjAiXSxbNSwxLCJcXHRleHRjb2xvcntibHVlfXtQXnF9Il0sWzEsMSwiXFx0ZXh0Y29sb3J7Ymx1ZX17UF4xfSJdLFs0LDAsIlxcY2RvdHMiXSxbMywxLCJcXHRleHRjb2xvcntibHVlfXtQXjJ9Il0sWzMsMiwiXFx0ZXh0Y29sb3J7cmVkfXtQXnF9Il0sWzMsMCwiPSIsMSx7ImNvbG91ciI6WzIyOCwxMDAsNjBdfSxbMjI4LDEwMCw2MCwxXV0sWzQsMSwiPSIsMSx7ImNvbG91ciI6WzIyOCwxMDAsNjBdfSxbMjI4LDEwMCw2MCwxXV0sWzQsMiwiXFxiZXRhXjEsXFwsYl4xLFxcLGheMSIsMSx7ImNvbG91ciI6WzIyOCwxMDAsNjBdfSxbMjI4LDEwMCw2MCwxXV0sWzYsMSwiXFxiZXRhXjIsXFwsYl4yLFxcLGheMiIsMSx7ImNvbG91ciI6WzIyOCwxMDAsNjBdfSxbMjI4LDEwMCw2MCwxXV0sWzYsNSwiPSIsMSx7ImNvbG91ciI6WzIyOCwxMDAsNjBdfSxbMjI4LDEwMCw2MCwxXV0sWzMsNSwiXFxiZXRhXnEsXFwsYl5xLFxcLGhecSIsMSx7ImNvbG91ciI6WzIyOCwxMDAsNjBdfSxbMjI4LDEwMCw2MCwxXV0sWzcsNCwiXFxnYW1tYV4xLGNeMSxyXjEiLDEseyJjb2xvdXIiOlswLDYwLDYwXX0sWzAsNjAsNjAsMV1dLFs3LDYsIlxcZ2FtbWFeMixjXjIscl4yIiwxLHsiY29sb3VyIjpbMCw2MCw2MF19LFswLDYwLDYwLDFdXSxbNywzLCI9IiwxLHsiY29sb3VyIjpbMCw2MCw2MF19LFswLDYwLDYwLDFdXV0=
\[\begin{tikzcd}
	{P^0} && {P^1} && \cdots && {P^q} \\
	& {\textcolor{blue}{P^1}} && {\textcolor{blue}{P^2}} && {\textcolor{blue}{P^q}} \\
	&&& {\textcolor{red}{P^q}}
	\arrow["{\beta^1,\,b^1,\,h^1}"{description}, color={rgb,255:red,51;green,92;blue,255}, from=2-2, to=1-1]
	\arrow["{=}"{description}, color={rgb,255:red,51;green,92;blue,255}, from=2-2, to=1-3]
	\arrow["{\beta^2,\,b^2,\,h^2}"{description}, color={rgb,255:red,51;green,92;blue,255}, from=2-4, to=1-3]
	\arrow["{=}"{description}, color={rgb,255:red,51;green,92;blue,255}, from=2-4, to=1-5]
	\arrow["{\beta^q,\,b^q,\,h^q}"{description}, color={rgb,255:red,51;green,92;blue,255}, from=2-6, to=1-5]
	\arrow["{=}"{description}, color={rgb,255:red,51;green,92;blue,255}, from=2-6, to=1-7]
	\arrow["{\gamma^1,c^1,r^1}"{description}, color={rgb,255:red,214;green,92;blue,92}, from=3-4, to=2-2]
	\arrow["{\gamma^2,c^2,r^2}"{description}, color={rgb,255:red,214;green,92;blue,92}, from=3-4, to=2-4]
	\arrow["{=}"{description}, color={rgb,255:red,214;green,92;blue,92}, from=3-4, to=2-6]
\end{tikzcd}\]
As $P^q$ is the finest cover in this chain, it acts as the common refinement and the map $(\gamma^m,c^m,r^m)$ to $P^m$ is just composition of the covering maps in the chain that occur before $P^m$. The $cyc$ map in \cref{const: Dennis trace map} composes the chain and attaches its inverse to complete the cyclic chain. By \cref{rmk: inverse in GC}, it amounts to completing the cycle by flipping the composition zigzag. Let $(\gamma^0, c^0, r^0)$ be the composition of all the map in the chain in line with the previous notation. Then $cyc$ evaluates the above chain to  
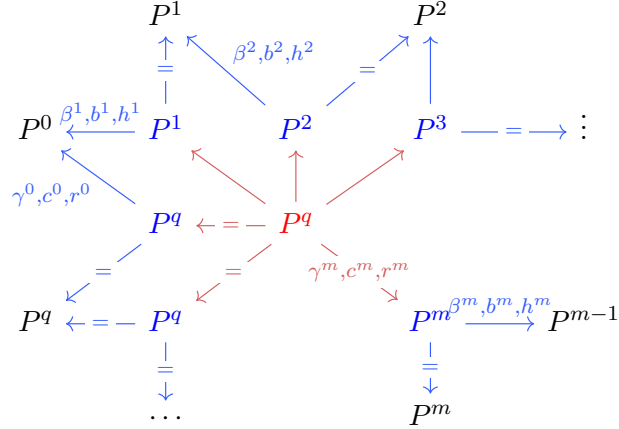
\[\begin{tikzcd}
	& {P^1} && {P^2} & \\
	{P^0} & {\textcolor{blue}{P^1}} & {\textcolor{blue}{P^2}} & {\textcolor{blue}{P^3}} & \vdots \\
	& {\textcolor{blue}{P^q}} & {\textcolor{red}{P^q}} \\
	{P^q} & {\textcolor{blue}{P^q}} && {\textcolor{blue}{P^m}} & {P^{m-1}} \\
	& \cdots && {P^m}
	\arrow["{=}"{description}, color={rgb,255:red,51;green,92;blue,255}, from=2-2, to=1-2]
	\arrow["{\beta^1, b^1, h^1}"', color={rgb,255:red,51;green,92;blue,255}, from=2-2, to=2-1]
	\arrow["{\beta^2, b^2, h^2}"', color={rgb,255:red,51;green,92;blue,255}, from=2-3, to=1-2]
	\arrow["{=}"{description}, color={rgb,255:red,51;green,92;blue,255}, from=2-3, to=1-4]
	\arrow[color={rgb,255:red,51;green,92;blue,255}, from=2-4, to=1-4]
	\arrow["{=}"{description}, color={rgb,255:red,51;green,92;blue,255}, from=2-4, to=2-5]
	\arrow["{\gamma^0, c^0, r^0}", color={rgb,255:red,51;green,92;blue,255}, from=3-2, to=2-1]
	\arrow["{=}"{description}, color={rgb,255:red,51;green,92;blue,255}, from=3-2, to=4-1]
	\arrow[color={rgb,255:red,214;green,92;blue,92}, from=3-3, to=2-2]
	\arrow[color={rgb,255:red,214;green,92;blue,92}, from=3-3, to=2-3]
	\arrow[color={rgb,255:red,214;green,92;blue,92}, from=3-3, to=2-4]
	\arrow["{=}"{description}, color={rgb,255:red,214;green,92;blue,92}, from=3-3, to=3-2]
	\arrow["{=}"{description}, color={rgb,255:red,214;green,92;blue,92}, from=3-3, to=4-2]
	\arrow["{\gamma^m, c^m, r^m}"{description}, color={rgb,255:red,214;green,92;blue,92}, from=3-3, to=4-4]
	\arrow["{=}"{description}, color={rgb,255:red,51;green,92;blue,255}, from=4-2, to=4-1]
	\arrow["{=}"{description}, color={rgb,255:red,51;green,92;blue,255}, from=4-2, to=5-2]
	\arrow["{\beta^m, b^m , h^m}", color={rgb,255:red,51;green,92;blue,255}, from=4-4, to=4-5]
	\arrow["{=}"{description}, color={rgb,255:red,51;green,92;blue,255}, from=4-4, to=5-4]
\end{tikzcd}\]
This is exactly the simple tensor that lands in $HH(\bbZ\mM(\CC_{hG}))$ and $T_\mM$ evaluates it to
$$\sum_{k\in I_q} h^1_{\gamma^1(k)}e^{-1} \otimes h^2_{\gamma^2(k)}e^{-1} \otimes \cdots \otimes h^q_{\gamma^q(k)}e^{-1} \otimes \mu (e\cdot P^q_i) = \sum_{i\in I^q}h^1_{\gamma^1(k)}\otimes \cdots \otimes h^q_{\gamma^q(k)}\otimes \mu(P^q_i)$$
which matches with the image $T_\wW$, completing the proof.
\end{proof}

\begin{remark}
In \cite{bohmann2024trace}, the authors use these trace invariants to conclude various interesting results regarding $K$-theory of certain polytope categories using the map $T_{\mM}\circ DTr$ thus showing such trace invariant computations to be powerful. Currently, we do not know if there are other trace invariants that can be extracted from the Hochschild complex associated to $\CC$.    
\end{remark}

\begin{remark}
    If we take $A$ to be the \emph{universal measure} $K_0(\CC)$ (\cite[Example 7.10]{bohmann2024trace}), then by \cite{Malkiewich2022OnHigherScissors} and \cite[Corollary 1.2]{bohmann2024trace}, $\tr$ is a rational isomorphism:
    $$K_n(\PP_G^X)\otimes \mathbb{Q} \cong H_i(G; K_0(\PP^{X}))\otimes \mathbb{Q} $$
    and thus $Dtr$ is a rational injection and $\tr_\mM$ is a rational surjection
    $$K_n(\PP_G^X)\otimes \mathbb{Q} 
    \rightarrowtail 
    HH_n(\bbZ\mM(\PP^{X})) \otimes \mathbb{Q} 
    \text{ and }
    HH_n(\bbZ\mM(\PP^{X})) \otimes \mathbb{Q} 
    \twoheadrightarrow 
    H_i(G; K_0(\PP^{X}))\otimes \mathbb{Q} .$$  
  From \cref{ex: HH of P_n} we can see that this is not a rational isomorphism as $K_0(\PP^n_{\bbZ\rtimes C_2}) \cong \bbZ$ (a scissors congruence class here is completely determined by length which is given by integer multiple of $1/n$) and $H_0(\bbZ\rtimes C_2;\bbZ) \cong \bbZ$ but $HH_0 \cong \bbZ^2$.
\end{remark}
\end{document}